\documentclass[10pt, a4paper]{amsart}

\usepackage{kantlipsum} % for text filler

\calclayout

\usepackage{amsfonts}
\usepackage{mathrsfs}
\usepackage{amscd,amsmath,amssymb,amsfonts}
\usepackage{mathtools}
\usepackage{enumerate}
\usepackage{stmaryrd}
\usepackage[T1]{fontenc}
\usepackage[all]{xy}
\usepackage{natbib}
\usepackage{tikz-cd}
\usepackage[bottom]{footmisc}
\usepackage{adjustbox}
\usepackage{hyperref}
\usepackage{afterpage}

\theoremstyle{plain}
\newtheorem{thm}{Theorem}
\newtheorem*{thm*}{Theorem}

\newtheorem*{theorem*}{Theorem}
\newtheorem{lem}[thm]{Lemma}
\newtheorem{cor}[thm]{Corollary}
\newtheorem{prop}[thm]{Proposition}
\newtheorem*{prop*}{Proposition}

\theoremstyle{definition}

\newtheorem*{definition*}{Definition}
\newtheorem{defn}[thm]{Definition}

\newtheorem{nota}[thm]{Notation}

\newtheorem*{claim*}{Claim}

\numberwithin{thm}{section} 
\numberwithin{equation}{section}

\theoremstyle{plain}
\newtheorem{THM}{Theorem}
\newtheorem{COR}{Corollary}

\newcommand{\codim}{{\rm codim}}

\newcommand{\Hom}{{\rm Hom}}
\newcommand{\Ext}{{\rm Ext}}

\newcommand{\dlog}{{\rm dlog}}

\newcommand{\Spec}{{\rm Spec \,}}

\newcommand{\Tr}{{\rm Tr}}

\newcommand{\sE}{{\mathcal E}}
\newcommand{\sF}{{\mathcal F}}
\newcommand{\sG}{{\mathcal G}}
\newcommand{\sH}{{\mathcal H}}
\newcommand{\sI}{{\mathcal I}}

\newcommand{\sO}{{\mathcal O}}

\newcommand{\sX}{{\mathcal X}}
\newcommand{\sY}{{\mathcal Y}}

\newcommand{\A}{{\mathbb A}}

\renewcommand{\H}{{\mathbb H}}

\renewcommand{\P}{{\mathbb P}}

\newcommand{\X}{{\mathbb X}}

\newcommand{\Z}{{\mathbb Z}}

\newcommand{\bfT}{{\mathbf T}}

\newcommand{\jm}{\jmath}

	\newcommand{\Sh}{{\rm Sh}}

	\newcommand{\cl}{{\rm cl}}

	\newcommand{\vu}{V^{\ast}}
	\newcommand{\vl}{V_{\ast}}
	\newcommand{\fu}{\textrm{F}^{\ast}}
	\newcommand{\fl}{\textrm{F}_{\ast}}
	\newcommand{\f}{\textrm{F}}
	
	\newcommand{\gu}{\textrm{G}^{\ast}}
	\newcommand{\gl}{\textrm{G}_{\ast}}
	\newcommand{\g}{\textrm{G}}

	\newcommand{\Grab}{\textbf{GrAb}}

	\newcommand{\trdeg}{{\rm tr.deg}}
	\newcommand{\CH}{{\rm CH}}
	\newcommand{\CHl}{{\rm CH}_{\ast}}
	\newcommand{\CHu}{{\rm CH}^{\ast}}
	\newcommand{\rmF}{{\rm F}}
	\newcommand{\rmFu}{{\rm F}^{\ast}}
	\newcommand{\rmFl}{{\rm F}_{\ast}}
	\newcommand{\iinf}{i_{\infty}}
	\newcommand{\phik}{\phi^{'}_{(X,\Phi)}}
	\newcommand{\phikk}{\phi^{'}_{(Y,\Psi)}}
	\newcommand{\Zsup}{Z_{\Phi}}
	\newcommand{\Zupp}{Z_{\Psi}}
	
	\newcommand{\tX}{\tilde{X}}
	\newcommand{\tD}{\tilde{D}}
	\newcommand{\tW}{\tilde{W}}
	\newcommand{\ti}{\tilde{i}}
	
	\newcommand{\we}{W_{\epsilon}}

	\newcommand{\phic}{\phi^{'}}
	\newcommand{\ie}{i_{\epsilon}}
	\newcommand{\aep}{\alpha_{\epsilon}}
	\newcommand{\be}{\beta_{\epsilon}}

	\newcommand{\TT}{{\bf T}}
	
	\newcommand{\one}{{\bf 1}}

	\newlength\longest
	
	\newcommand{\NS}{\mathcal{N}_S}

	\newcommand{\ug}{\underline{\Gamma}}
	\newcommand{\ugz}{\underline{\Gamma}_{Z}}
	\newcommand{\ugzinv}{\underline{\Gamma}_{f^{-1}(Z)}}
	
	\newcommand{\ugps}{\underline{\Gamma}_{\Psi}}
	\newcommand{\ugpsinv}{\underline{\Gamma}_{f^{-1}(\Psi)}}

	\newcommand{\rug}{{\rm R}\underline{\Gamma}}
	\newcommand{\rugz}{{\rm R}\underline{\Gamma}_{Z}}
	
	\newcommand{\rugph}{{\rm R}\underline{\Gamma}_{\Phi}}
	\newcommand{\rugps}{{\rm R}\underline{\Gamma}_{\Psi}}
	\newcommand{\rugpsinv}{{\rm R}\underline{\Gamma}_{f^{-1}(\Psi)}}

	\newcommand{\omxs}{\Omega_{X/S}}
	\newcommand{\omxsb}{\Omega_{\bx/S}}

	\newcommand{\omys}{\Omega_{Y/S}}
	
	\newcommand{\omzs}{\Omega_{Z/S}}

	\newcommand{\hu}{\textrm{H}^{\ast}}
	\newcommand{\hl}{\textrm{H}_{\ast}}
	\newcommand{\h}{\textrm{H}}
	\newcommand{\hpu}{\textrm{HP}^{\ast}}
	\newcommand{\hpl}{\textrm{HP}_{\ast}}
	\newcommand{\hp}{\textrm{HP}}

	\newcommand{\qcx}{\rm Qcoh($X$)}
	\newcommand{\qcy}{\rm Qcoh($Y$)}

	\newcommand{\fpush}{f_{\ast}}
	\newcommand{\fbpush}{\bar{f}_{\ast}}

	\newcommand{\gpush}{g_{\ast}}

	\newcommand{\rfbpush}{{\rm R}\fbpush}
	\newcommand{\rgpush}{{\rm R}\gpush}
	
	\newcommand{\rfpush}{{\rm R}\fpush}

	\newcommand{\rmR}{\rm R}

	\newcommand{\fushr}{f^{!}}
	
	\newcommand{\pixushr}{\pi_X^{!}}

	\newcommand{\pixbushr}{\pi_{\bx}^!}

	\newcommand{\rshHom}{\mathrm{R}\mathcal{H}om}
	\newcommand{\rshHomx}{\mathrm{R}\mathcal{H}om_{\sO_X}}
	
	\newcommand{\rshHomy}{\mathrm{R}\mathcal{H}om_{\sO_Y}}

	\newcommand{\bx}{\bar{X}}

	\newcommand{\arr}{\arrow}

	\newcommand{\supp}{{\rm Supp}}
	\newcommand{\Cor}{{\rm Cor}}

	\title{Higher Direct Images of the Structure Sheaf Over a Dedekind Domain}

\author{Grétar Amazeen}
\email{g.amazeen@gmail.com}

\begin{document}

\maketitle

\begin{abstract}
		We prove that for Noetherian, smooth, separated, integral, finite type schemes $X$ and $Y$ 
		over an excellent Dedekind domain $R$, 	that are properly birational over $R$,
		we have $R^i\fpush\sO_X \cong R^i\gpush \sO_Y$ and
		$R^i \fpush \Omega_{X/S}^{d} \cong R^i\gpush
		\Omega_{Y/S}^d$, where $d$ is the relative dimension of $X$  and $Y$ over 
		$S=\Spec(R)$, and $f$ and $g$ are the structure maps of $X$ 
		and $Y$, respectively, as $S$-schemes. As a corollary we obtain the vanishing of higher direct images of 
		the structure sheaf for proper birational morphisms beteween such schemes. These 
		results extend those obtained by Chatzistamatiou--R\"ulling over perfect fields
		of positive characteristic and we obtain them by extending their method of algebraic 	
		correspondences. We furthermore obtain as a corollary that if $K$ is a number field and $\sO_K$ its ring of integers and if $X$ is a smooth and proper
		$K$-scheme with $\sX$ and $\sY$ two smooth proper models of $X$ over some dense 
		open subscheme $U \subseteq S = \Spec(\sO_K)$, that if 
		 $\h^j(\sX,\sO_{\sX})$ is $\sO_S(U)$-torsion-free we have
		$
		\h^j(\sX_t,\sO_{\sX_t}) = 	\h^j(\sY_t,\sO_{\sY_t}),
		$
		for all closed points $t \in U$. 
\end{abstract}
\tableofcontents

\section*{Introduction}

Recall that two integral schemes $X$ and $Y$ over a base scheme $S$  
are called 
properly birational over $S$ if there exists an integral scheme $Z$ over
$S$ and proper birational $S$-morphisms
$$
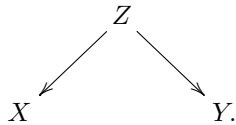
\xymatrix{
	&Z \ar[dr]\ar[dl] 
	\\X 
	&& Y.
}
$$

The main result proved in this article is the following.

\begin{THM} \label{maintheorem} (See Theorem \ref{Thefinaltheorem})
	Let $S$ be a 
	Noetherian, excellent, regular, separated, irreducible scheme of dimension at 
	most 1.  Let $S'$ be a separated $S$-scheme of finite type,
	and let $X$ and $Y$ be integral, smooth, separated $S$-scheme of finite type,  and
	$f:X\to S'$ and $g:Y\to S'$ be morphisms of $S$-schemes
	such that $X$ and $Y$ are
	properly birational over $S'$. Let $Z$
	be an integral scheme and let $Z\to X$ and $Z\to Y$ be
	proper birational morphisms
	such that 
	$$
	\xymatrix{
		&Z \ar[dr]\ar[dl] 
		\\X\ar[dr]_-f 
		&& Y\ar[dl]^-g
		\\&S'
	}	
	$$
	commutes. We denote the image of $Z$ in $X\times_{S'} Y$ by 
	$Z_0$. Then the action of $Z_0$ induces isomorphisms of $\sO_{S'}$-
	modules
	\begin{align*}
		R^i\fpush\sO_X &\xrightarrow{\cong} R^i\gpush \sO_Y \,\,\, \text{and} \\
		R^i \fpush \Omega_{X/S}^{d} &\xrightarrow{\cong} R^i\gpush
		\Omega_{Y/S}^d,
	\end{align*}
	for all $i$, where $d := \dim_S(X) = \dim_S(Y)$. 
\end{THM}

By considering $X = Y, S' = Y$ and $g= id_Y$ we get the 
following corollary.

\begin{COR} \label{maincorollary}
	Let $S$ be a 
	Noetherian, excellent, regular, separated, irreducible scheme of dimension at 
	most 1, let $X,Y$ be integral smooth $S$-schemes, and $f:X\to Y$ a 
	proper birational morphism. Then
	$$
	R^if_{\ast}\sO_X = 0 
	$$
	for all $i > 0$.
\end{COR}

Over a field of characteristic 0, 
these results are well known and follow from Hironaka's resolution of 
singularities. Over 
perfect fields of positive characteristic, this result was proven 
in \cite{KayAndre}. Since the resolution of singularities in 
positive characteristics is still an open problem, Chatzistamatiou 
and R\"ulling developed different methods, namely by algebraic 
correspondences, to prove Theorem \ref{maintheorem} and in this article we extend this 
strategy to the case of a base scheme $S$ that is a 
Noetherian, excellent, regular, separated, irreducible scheme of dimension at 
most 1. This includes all fields, so in particular reproves these results over fields 
of characteristic 0 without relying on the resolution of singularities,
and all excellent Dedekind domains.

In \cite{KayAndreVanishing}, a vanishing result similar to 
Corollary \ref{maincorollary} is proven 
without any smoothness assumptions on $X$ and $Y$, but 
with the assumption that $f:X \to Y$ is birational and 
\textit{projective}. In \textit{loc. cit.}, Chatzistamatiou and 
R\"ulling pose the question of whether their vanishing result can
be extended to the case of a proper birational morphism. 
We have therefore partially answered that question in the 
affirmative.  In \cite{Lodh} such vanishing is proved for $X,Y$ 
normal, $Y$ regular and quasi-excellent, and $f$ proper and
birational, but only for $R^1f_{\ast}$. The vanishing for higher
direct images does not hold in this generality, and in \cite{Lodh}
a counter example for $i = 2$ and $\dim(Y) = 3$ is cited from
\cite{Cutkosky}.

Another corollary of Theorem \ref{maintheorem} we obtain by cohomology and base change
is the following
\begin{COR}
	Let $K$ be a number field and $\sO_K$ be its ring of integers. Let $X$ be a smooth and proper
	$K$-scheme. Let $\sX$ and $\sY$ be two smooth proper models of $X$ over some dense 
	open subscheme $U \subseteq S = \Spec(\sO_K)$. Then for any dense open $V\subseteq U$
	such that $\h^j(\sX,\sO_{\sX})$ is $\sO_S(V)$-torsion-free we have
	$$
	\h^j(\sX_t,\sO_{\sX_t}) = 	\h^j(\sY_t,\sO_{\sY_t})
	$$
	for all closed points $t \in V$. 
\end{COR}
This is something new that we see over a Dedekind domain and hope has some interesting 
arithmetic applications, as the fibres do not have to be properly birational, even though the 
models are.

We will refer the reader to \cite{KayAndre} for details about setting up 
the algebraic correspondences, but we recall here the basic objects. We consider two categories $\vl$ and $\vu$. They have the
same objects, namely pairs $(X, \Phi)$ where $X$ is a smooth, 
separated $S$-scheme of finite type (for conciseness 
we will refer to such schemes as $\NS$-schemes throughout this 
article) and $\Phi$ is a family of supports on $X$. A morphism 
$f: (X,\Phi) \to (Y,\Psi)$ in $\vl$ is a morphism $f$ of $S$-schemes 
such that $f|_{\Phi}$ is proper and $f(\Phi) \subseteq \Psi$, and a 
morphism $g: (X,\Phi) \to (Y,\Psi)$ in $\vu$ is a morphism $g$ of 
$S$-schemes such that $g^{-1}(\Psi) \subseteq \Phi$. A \textit{weak cohomology theory with supports} (abbreviated as WCTS in this 
article) is a quadruple $\f = (\fl, \fu, T, e)$ where  
\begin{align*}
		\fl&: \vl \to \Grab,\,\,\,\text{and} \\
			\fu&: (\vu)^{op} \to \Grab,
\end{align*}
	such that for any $X\in ob(\vl) = ob(\vu)$
	we have 
	$
	\fl(X) = \fu(X) =: \f(X),
	$
	T is a morphism of graded Abelian
		groups for \textit{both} gradings
	$$
		T_{X,Y} : \f(X) \otimes\f(Y) \to \f(X\otimes Y),
	$$
	for every two objects $X,Y \in ob(\vl) = ob(\vu)$, and
	$e$ is a morphism of Abelian groups
	$$
		e:\Z \to F(S).
	$$
	This quadruple is required to satisfy the following conditions:
		\begin{enumerate}
				\item The covariant ``homology'' functor $\fl$
				preserves coproduct and the contravariant
				``cohomology'' functor $\fu$ maps
				coproducts to products. Moreover if we have
				objects $(X,\Phi_1)$ and $(X,\Phi_2)$ with 
				the same underlying scheme and such that 
				the supports don't intersect, $\Phi_1\cap
				\Phi_2 = \emptyset$, then the map
				$$
				\fu(\jm_1)+\fu(\jm_2):\fu(X,\Phi_1)\oplus
				\fu(X,\Phi_2) \to \fu(X,\Phi_1\cup \Phi_2)
				$$
				is required to be an isomorphism. Here 
				the maps $\jm_1$ and $\jm_2$ are the maps
				in $\vu$ 
				$$
				\jm_1:(X,\Phi_1\cup\Phi_2) \to(X,\Phi_1),
				$$
				and 
				$$
				\jm_2:(X,\Phi_1\cup\Phi_2) \to (X,\Phi_2),
				$$
				induced by the identity map on the underlying
				scheme $X$.
				\item The subdata $(\fl,T,e)$ and 
				$(\fu,T,e)$ define (right-lax) symmetric
				monoidal functors.
				\item Let $(X,\Phi) \in ob(\vl)=ob(\vu)$
				be an object such that the underlying scheme
				$X$ is connected. Then the gradings on 
				$\fl(X,\Phi)$ and $\fu(X,\Phi)$ are 
				connected by the equality
				$$
				\f_i(X,\Phi) = \f^{2\dim_S(X)-i}(X,\Phi).
				$$
				\item For all Cartesian diagrams 
				$$
				\xymatrix{
						(X',\Phi') \ar[d]_{g_X} \ar[r]^{f'} &
						(Y',\Psi') \ar[d]^{g_Y} \\
						(X,\Phi) \ar[r]^{f} &
						(Y,\Psi)
					}
				$$	
				of objects in $ob(\vl) = ob(\vu)$ and maps
				$g_X,g_Y \in \vu$ and $f, f' \in \vl$ 
				such that either
				\begin{itemize}
						\item $g_Y$ is smooth, or
						\item $g_Y$ is a closed immersion and 
						$f$ is transversal to $g_Y$
					\end{itemize}
				the following equality holds
				$$
				\fu(g_Y)\circ \fl(f) = \fl(f')\circ \fu(g_X).
				$$
		\end{enumerate}

			Let $\f = (\fl,\fu,T,e)$ and $\g = (\gl,\gu,U,\epsilon)$ be weak cohomology theories
			with supports. A morphism
			$$
			\phi: \f \to \g
			$$
			is a family $\{\phi_X\}$ of morphisms
			$
			\phi_X: \f(X) \to \g(X)
			$
			of graded Abelian groups (for both gradings)
			such that $\phi$ induces a natural transformation
			of (right-lax) symmetric monoidal functors
			$$
			\phi: (\fl,T,e) \to (\gl,U,\epsilon),
			$$
			and 
			$$
			\phi: (\fu,T,e) \to (\gu,U,\epsilon).
			$$
			The category of weak cohomology theories with
			supports and these morphisms is denoted by
			$\TT$.

		Let
		$(X,\Phi_1), (X,\Phi_2) \in ob(\vl)$ be two
		objects with the same underlying $\NS$-scheme
		$X$ and let $\f = (\fl,\fu,T,e) \in \TT$. We define
		$$
		\cup: \f(X,\Phi_1)\otimes_{\Z}\f(X,\Phi_2)
		\xrightarrow[]{T} \f(X\times_S X,\Phi_1\times_S \Phi_2) \xrightarrow[]{\fu(\Delta_X)} \f(X,\Phi_1 \cap \Phi_2),
		$$
		where $\Delta_X: (X,\Phi_1\cap \Phi_2) \to 
		(X\times_S X, \Phi_1 \times_S \Phi_2)$ 
		is induced by the diagonal immersion.
		%\end{defn}
		This cup product is  distributative over addition, associative,
		graded-commutative, and it respects the pullback functor. Furthermore we have the following \textit{projection
			formulas}. Let
			\begin{enumerate}
				\item  $f_1:(X,\Phi_1) \to (Y,\Phi_2)$ in $\vl$, \,\,\, \text{and}
				\item $f_2:(X,f^{-1}(\Psi)) \to (Y,\Psi)$ in $\vu$,
			\end{enumerate}
			be induced by $f$.
			Then $f$ also induces a morphism
			$$
			f_3: (X,\Phi_1\cap f^{-1}(\Psi)) \to (Y,\Phi_2\cap \Psi)
			$$
			in $\vl$, and for all $a\in \f(X,\Phi_1)$ and
			$b \in \f(Y,\Psi)$ the following formulas hold 
			in $\f(Y,\Phi_2\cap \Psi)$
			\begin{enumerate}
				\item  $\fl(f_3)(a \cup \fu(f_2)(b)) = \fl(f_1)(a)\cup b$,
				\item  $\fl(f_3)(\fu(f_2)(b)\cup a) = b \cup \fl(f_1)(a)$.
			\end{enumerate}

		In Section 1 we show that Chow groups give an example 
		of a WCTS. The pushforward $\CHl$ is the 
		proper pushforward and the pullback is given by Fulton's refined
		Gysin morphism. The absence of a well-developed theory of an  
		exterior product over a higher dimensional base scheme 
		forces us to restrict the dimension of $S$ to at most 1. 
		
		In Section 2 we show that Hodge cohomology provides us 
		with a second example of a WCTS. For a morphism
		$f:X\to Y$ of $S$-schemes of finite type, the pullback is 
		essentially induced in cohomology by the map
		$$
		\Omega_{Y/S}^j \to f_{\ast}\Omega_{X/S}^j \to \rfbpush\Omega_{X/S}^j.
		$$
		The pushforward is more involved. We start by defining a  
		\textit{proper pushforward}. Namely, assume
		we have a diagram of separated, finite type $S$-schemes
		$$
		\xymatrix{
			X \ar[rr]^f\ar[dr]_{\pi_X} && Y \ar[dl]^{\pi_Y} \\
			& S}
		$$
		then we define a pushforward 
		$$
		\fpush: \rfpush	\rshHomx(\omxs^k,\pi_X^{!}\sO_S) \to 
		\rshHomy(\omys^k,\pi_Y^{!}\sO_S)
		$$
		for any $k\geq 0$  as the composition of the 
		pullback
		$$
		\rfpush	\rshHomx(\omxs^k,\pi_X^{!}\sO_S) 
		\xrightarrow{(f^{\ast})^{\vee}} 
		\rshHomy(\omys^k,\rfpush\fushr\pi_Y^!\sO_S),
		$$
		and the trace
		$$
		\rshHomy(\omys^k,\rfpush\fushr\pi_Y^!\sO_S)
		\xrightarrow{\Tr_f}
		\rshHomy(\omys^k,\pi_Y^{!}\sO_S).
		$$
		The pushforward is then defined via the Nagata compactification and 
		this proper pushforward.
		
		We then introduce the \textit{pure part} of the Hodge cohomology,
		$(\hpl, \hpu, \times, 1)$
		and show that it satisfies a \textit{semi-purity} 
		condition necessary for the Existence Theorem to hold.

		Apart from some dimension considerations in Lemma 
		\ref{SuslinVoevodsky} and Corollary 
		\ref{ChowExteriorRespectsGrading}, the material in the first 
		two sections is virtually unchanged from \cite{KayAndre}. We note
		that in these first two sections we have used that $S$ 
		has dimension at most 1 for the
		exterior product and that it is excellent (or at least 
		universally catenary), regular, and irreducible for the above-mentioned
		lemma and corollary.

		In Section 3 we encounter the first of two main technical points 
		of this article. We prove the following theorem that gives 
		sufficient conditions for the existence of a morphism 
		of weak cohomology theories with supports $\CH \to \f$.
		\begin{THM}\label{themaintheorem} (See Theorem \ref{existencetheorem})
			Let $S$ be a Noetherian, excellent, 
			regular, separated and irreducible scheme of Krull-dimension
			at most 1. Let  
			$\rmF \in \bfT$ be a weak cohomology theory with
			supports satisfying the semi-purity
			condition in Definition  \ref{semipurity}. Then $\Hom_{\bfT}(\CH,\rmF)$ is non-empty 
			if the following conditions
			hold.
			\begin{enumerate}
				\item \label{condunit} For the $0$-section
				$\imath_0:S \to \P_S^1$ and the 
				$\infty$-section $\imath_{\infty}:S\to 
				\P_S^1$ the following equality holds:
				$$
				\rmFl(\imath_0)\circ e = \rmFl(\imath_{\infty})\circ e.
				$$
				\item \label{condonclasselement} If 
				$X$ is an $\mathcal{N}_S$-scheme and
				$W\subset 
				X$ is an 
				integral closed subscheme then there exists 
				a cycle-class element $\cl(W,X) \in 
				\rmF_{2\dim_S(W)}(X,W)$, and if $W\subset X$
				is any closed subscheme we define
				$$
				\cl(W,X) = \sum_i n_i \cl(W_i,X),
				$$
				where the $W_i$ are the irreducible components of $W$ and $\sum_i n_i [W_i]$
				is the fundamental cycle of $W$\footnotemark, such
				that the following conditions hold: 
				\footnotetext{Technically we should write $\cl(W,X) = \sum_i n_i \fl(i_{W_i})\cl(W_i,X)$ 
					where $i_{W_i}:(X,W_i)\to (X,W)$ enlarges the supports. Notice also that when $W$ is not 
					pure $S$-dimensional the cycle-class element
					$\cl(W,X)$ does not live in $\f_{2\dim_S W}(X,W)$.}
				\begin{enumerate}[i)]
					\item  \label{axiomopenimmersion} For any 
					open $U\subseteq X$ such that
					$U \cap W$ is regular, we have
					$$
					\fu(j)(\cl(W,X)) = \cl(W\cap U,U),
					$$
					where $j:(U,U\cap W) \to (X,W)$ is 
					induced by the open immersion $U\subseteq X$.
					\item \label{axiomsmoothmorphism} If $f:X 
					\to Y$ is a smooth morphism between
					$\NS$-schemes $X$ and $Y$, and 
					$W \subset Y$ 
					is a regular closed subset, then
					$$
					\fu(f)(\cl(W,Y)) = \cl(f^{-1}(W),X).
					$$
					\item \label{axiomdivisor}
					Let $i:X\to Y$ be the closed immersion 
					of an irreducible, regular, closed 
					$S$-subscheme $X$ into an 
					$\NS$-scheme $Y$. For any 
					effective smooth divisor $D \subset Y$
					such that 
					\begin{itemize}
						\item $D$ meets $X$ properly, thus 
						$D\cap X := D\times_Y X$ is a divisor
						on $X$,
						\item $D' := (D\cap X)_{\rm red}$ 
						is regular and irreducible, so 
						$D\cap X = n\cdot D'$ as divisors
						(for some $n\in \Z, n\geq 1$).
					\end{itemize}
					We define
					$g:(D,D') \to (Y,X)$ in $\vu$ as the 
					map induced by the
					inclusion $D\subset Y$. Then the following
					equality holds:
					$$
					\rmFu(g)(\cl(X,Y)) = n\cdot \cl(D',D).
					$$	
					\item \label{axiomfinite} Let $f: X \to 
					Y$ be a 
					morphism of $\NS$-schemes. Let $W\subset X$ be a 
					regular closed 
					subset such that the restricted map
					$$
					f|_W:W \to f(W)
					$$
					is proper and finite of degree $d$. Then 
					$$
					\fl(f)(\cl(W,X)) = d \cdot \cl(f(W),Y).
					$$
					\item \label{axiomproduct} Let $X,Y$ be
					$\NS$-schemes and let
					$W \subset X$ and $V \subset Y$ be
					regular, integral closed subschemes. 
					Then the following equation holds
					$$
					T(\cl(W,X)\otimes_S \cl(V,Y)) = 
					\begin{cases}
						\cl(W\times_S V,X\times_S Y) \qquad\, \text{if $W$ or $V$ is dominant over $S$},
						\\
						0 \qquad\qquad \qquad\qquad \text{otherwise}.
					\end{cases}
					$$
					
					\item \label{axiomone} For the base 
					scheme $S$ we have $\cl(S,S) =1_S$.	
				\end{enumerate}
			\end{enumerate}
		\end{THM}
		In this theorem we first see the fundamental difference between
		the case of working over a perfect field like in \cite{KayAndre}, and 
 		in the more general case treated in this article. Namely, the 
 		conditions in \cite[Theorem 1.2.3.]{KayAndre}, 
 		the corresponding existence theorem, are simpler to state. This 
 		is because over a perfect field 
		one can construct the fundamental cycle by a 
		pushforward from its smooth locus and then spreading out.
 		 In the
 		more general case, we can only guarantee a non-empty 
 		\textit{regular locus} and there we have schemes that are not
 		smooth over $S$ and therefore not in $ob(\vl)$.
 		Thus the fundamental class is constructed in an ad hoc way which is not directly related to pushforward along a closed immersion.
 		The proof of Theorem \ref{themaintheorem} follows along the 
 		same lines as in \cite{KayAndre}, namely; A map is first defined
 		on cycles, then shown to pass to a map on the Chow groups, then 
 		that this extends to a natural transformation of right-lax symmetric
 		monoidal functors $(\CHl, \times_S, 1) \to (\fl,T,e)$. This 
 		natural transformation is extended to a morphism of weak cohomology
 		theories with supports by showing that for smooth morphisms, then 
 		closed immersions, and then putting that together to get the general
 		case since all morphisms between smooth $S$-schemes are l.c.i.
 		morphisms.

		In Section 4 we encounter the second main technical point of this 
		article when we construct the cycle class map to Hodge 
		cohomology and show that with this cycle class map, the 
		pure part $\hp$ satisfies the conditions of the Existence Theorem. In 
		\cite{KayAndre} this cycle class map is defined via the pushforward along a closed immersion
		between smooth schemes, and this can not be done here. We use a classical construction of 
		a cycle class map, see \cite[\S III. 3.1]{ElZein1978}, and have to be 
		quite explicit to make sure our construction is in line with duality theory and its signs as
		laid out in \cite{Conrad}. This is done by first considering a regular closed
		subscheme $i: Z \hookrightarrow X$ of codimension $c$ in $X$. We obtain a map
		\begin{equation*} 
			\sO_Z \cong	\bigwedge^c\sI/\sI^2 \otimes_{\sO_Z} \omega_{Z/X}  \xrightarrow{} 
			i^{\ast}\Omega_{X/S}^c\otimes_{\sO_Z} \omega_{Z/X} \cong
			i^{!}(\Omega_{X/S}^c)[c],
		\end{equation*} 
		induced by the map
		\begin{align*}
			\sI/\sI^2 &\to i^{\ast}(\Omega_{X/S}^1) \\
			\bar{a} &\mapsto da.
		\end{align*}
		Adjunction of $Ri_{\ast}$ and $i^!$ yields 
		\begin{equation*}
			i_{\ast}\sO_Z \to \Omega_{X/S}^c[c].
		\end{equation*}
		which induces in cohomology a map 
		\begin{equation*} \label{cycleclassfinalmapregular}
			H^0(Z,\sO_Z) \xrightarrow{\gamma_Z} H^{c}_Z(X,\Omega_{X/S}^c),
		\end{equation*}
		and we define
		$$
		\cl(Z,X) := \gamma_Z(1).
		$$
		This cycle class is then used to obtain a cycle class for 
		irreducible closed subschemes by spreading out from the 
		regular locus. The semi-purity condition, Definition \ref{semipurity}, 
		then ensures that this cycle class is well defined. When $\dim(S) = 1$, $X$ is an 
		$\NS$-scheme, and 
		$V\subset X$ is a regular integral closed subscheme that lies in the fiber over 
		a closed point of $S$, then this cycle class vanishes.
		
		In Section 5 we prove Theorem \ref{maintheorem}. Having set up the
		machinery of the algebraic correspondences in Sections 1-4 the proof
		now is essentially the same as in \cite{KayAndre} but is recorded for
		completeness. Namely, we first prove Theorem 
		\ref{applicationvanishing}, a vanishing theorem that for connected 
		$\NS$-schemes $X$ and $Y$
		and a correspondence $\alpha$ from $X$ to $Y$ then
		\begin{enumerate}
			\item if $\alpha$ projects to an $r$-codimensional subset of 
			$Y$ then $\alpha$ acts as 0 on all $\h^i(X,\Omega_{X/S}^j)$ where
			$j$ is less than $r$, and
			\item if $\alpha$ projects to an $r$-codimensional subset of $X$ 
			then $\alpha$ acts as 0 on all $\h^i(X,\Omega_{X/S}^j)$ for all
			$j$ greater than or equal to $r$.
		\end{enumerate}
		Then the proof of Theorem \ref{maintheorem} follows by  reducing to 
		the case $S' = S$ and considering the compositions
		\begin{align*}
			[Z]&\circ [Z^t], \,\,\,\text{and} \\
			[Z^t]&\circ [Z],
		\end{align*}
		where $[Z]$ is the correspondence from $X$ to $Y$ 
		defined by $Z$ and $[Z^t]$ is 
		the correspondence from $Y$ to $X$ defined by the image 
		of $Z$ under the canonical isomomorphism 
		$X\times_S Y \xrightarrow{\cong} Y\times_S X$. 
		Here we reap the benefits of considering the supports throughout, 
		since we don't need to assume properness for $X$ or $Y$ over $S$, 
		to have proper projections from $X\times_S Y$ to $X$ and $Y$, it is 
		enough that the projections are proper \textit{when restricted to the
		subset $Z$}. We show that 
		these compositions yield 
		\begin{align*}
					\Delta_{Y/S} &+ \varepsilon_Y, \,\,\, \text{and} \\
					\Delta_{X/S} &+\varepsilon_X,
		\end{align*}
		respectively, where the diagonals are the identities as 
		correspondences and $\varepsilon_Y$ and $\varepsilon_X$ are 
		some cycles that act as 0 by the Vanishing Theorem.

\subsection*{Acknowledgements}
The results in this article were part of my PhD. thesis, started at 
the Freie Universit\"at Berlin and finished at the Bergische Universit\"at 
Wuppertal. First and foremost I want to thank my advisor Kay R\"ulling for
 suggesting this
interesting problem, and for his continued and unfailing support, enthusiasm, and 
patience.  I would like to thank the Berlin Mathematical School for their
support as I was supported by a Phase II scholarship for a part of the time this work was conducted. 
I would like to thank Raju Krishnamoorthy, H\'el\`ene Esnault, Fei Ren, Marco 
d'Addezio,
Pedro Castillejo Blasco, Lei Zhang, and all other members of  Freie Universit\"at Berlin
and Bergische Universit\"at Wuppertal that I have had the pleasure of talking to, and 
working with, over the last few years.

\section{Chow Groups as a Weak Cohomology Theory With Supports}

We define two functors:
\begin{align*}
	\CHl: &\vl \to \Grab, \,\,\text{and} \\
	\CHu: &(\vu)^{op} \to \Grab.
\end{align*}
Let
$Z_{\Phi}(X)$ be the free Abelian group on the closed integral subschemes
that lie in $\Phi$, and $Rat_{\Phi}(X)$ the free Abelian group generated by
cycles of the form $div_W(f)$ with $f \in R(W)^{\times}$ and $W\in \Phi$.
and we set 
$$
\CH(X/S,\Phi) := Z_{\Phi}(X)/Rat_{\Phi}(X). 
$$
On each object $\CH(X/S,\Phi)$ we have a grading by $S$-dimension
$$
\CH(X/S,\Phi) := \bigoplus_{d\geq 0} \CH_d(X/S,\Phi)[2d],
$$
where $\CH_d(X/S,\Phi)$ is the subgroup of $\CH_d(X/S)$,  consisting of those $d$-cycles that lie in $\CH(X/S,\Phi)$. 
The bracket $[2d]$ indicates that the group $\CH_d(X/S,\Phi)$ lies in degree $2d$.
Furthermore we have a grading by codimension. Namely if
$X$ is connected we set
$$
\CHu(X/S,\Phi) := \bigoplus_{c\geq 0}\CH^{c}(X/S,\Phi)[2c],
$$
where $\CH^{c}(X/S,\Phi)$ is the subgroup of $\CH(X/S,\Phi)$ consisting of 
cycles $\sum n_i[V_i]$ where each $V_i$ has codimension $c$ in $X$. If
$X$ is not connected, say $X = \amalg X_i$ is a decomposition into 
connected components, then we set
$$
\CHu(X/S,\Phi) := \bigoplus_{i} \CHu(X_i/S,\Phi\cap \Phi_{X_i}).
$$ 
The following lemma follows immediately from
the definitions of the gradings and 
\cite[Lemma 20.1.(2)]{Fultonbook}.

\begin{lem} \label{Chowrelationbetweengradings}
	For a connected $\NS$-scheme $X$ and a 
	family of supports $\Phi$ on $X$ we have 
	$$
	\CH_i(X/S,\Phi) = \CH^{\dim_S(X)-i}(X/S,\Phi).
	$$
\end{lem} \qed

We have functions on objects $\CHl$ resp. $\CHu$ 
sending $(X,\Phi)$ in $\vl$ resp. $(\vu)^{op}$
to $\CHl(X/S,\Phi)$ resp. $\CHu(X/S,\Phi)$.
We define $\CHl$ and $\CHu$ on morphisms and extend $\CHu$ and $\CHl$ to functors.

\subsection{Pushforward}
Let $f:(X,\Phi) \to (Y,\Psi)$ be a morphism in $\vl$ and let $V\in \Phi$ be a 
closed subscheme of $X$. By construction $f|_{\Phi}$ is proper so we  get a pushforward
\begin{align*}
	f_{\ast}:Z_{\Phi}(X) &\subset Z_{\ast}(X) \to Z_{\ast}(Y), \\
	f_{\ast}([V]) &= \deg(V/f(V))\cdot[f(V)].
\end{align*}
Furthermore since $f$ is a morphism in $\vl$ we have $f(V) \in \Psi$ so this 
gives a pushforward on cycles \linebreak
$
f_{\ast}:Z_{\Phi}(X) \to Z_{\Psi}(Y).
$
This pushforward factors through cycles rationally equivalent to zero
and we have a pushforward
$$
\CHl(f):\CH(X/S,\Phi) \to \CH(Y/S,\Psi),
$$
induced by $f_{\ast}$. 
This pushforward respects the grading, and the functoriality of the proper pushforward
shows that this gives us a functor \footnotemark
$$
\CHl: \vl \to \Grab. 
$$

\subsection{Pullback} \label{pullbackchowsection}

We first note that if we let $X$ be a regular $S$-scheme and $Y$ be an $\NS$-scheme. 
	Then any morphism $f:X\to Y$
	over $S$ is an l.c.i. morphism.
This is clear from the factorization 
	of $f$ as 
	$$
	\xymatrix{
		X \ar@/_2pc/[rr]^-{f}\ar[r]^-{\Gamma_f} 
		& X\times_S Y \ar[r]^-{pr_2} & Y,
	}
	$$
	where $\Gamma_f:X\to X\times_S Y$ the graph morphism, which is a closed 
	immersion since $Y\to S$ is separated
	and $pr_2$
	is the projection.

Now we can use the refined Gysin homomorphisms for l.c.i. morphisms, see 
 \cite[\S 6.6]{Fultonbook}, to construct a pullback. 

\begin{defn}
	Let $f:(X,\Phi) \to (Y,\Psi)$ be a morphism in $\vu$. The refined Gysin homomorphism
	defines a morphism for any $V \in \Psi$,
	\begin{align*}
		\CH(Y/S,V) &= \CH(V/S) \\
		&\xrightarrow{f^{!}} \CH(f^{-1}(V)/S)\\
		&= \CH(X,f^{-1}(V)/S)\\ 
		&\xrightarrow{\phi_{f^{-1}(V)}} \varinjlim_{W\in \Phi} \CH(X/S,W) \\
		&= \CH(X/S,\Phi),
	\end{align*}
	where $\phi_{f^{-1}(V)}$ is the natural morphism $\CH(X/S,f^{-1}(V)) \to 
	\varinjlim_{W\in \Phi} \CH(X/S,W)$. Note that if \linebreak$i:V_1 \hookrightarrow V_2$ is a
	closed immersion between two closed subschemes $V_1$ and  $V_2$ of $Y$ s.t.
	$V_1,V_2\in \Psi$ and  \linebreak $j:f^{-1}(V_1) \hookrightarrow f^{-1}(V_2)$ is the 
	induced closed immersion of closed subschemes of $X$, then the square
	$$
	\xymatrix{
		\CH(V_1/S) \ar[r]^-{f^!} \ar[d]^-{i_{\ast}} 
		& \CH(f^{-1}(V_1)/S) \ar[d]^-{j_{\ast}}
		\\ \CH(V_2/S) \ar[r]^-{f^!} 
		& \CH(f^{-1}(V_2)/S),
	}
	$$
	commutes by \cite[Proposition 6.6.(c)]{Fultonbook}.
	Therefore, the universal property of direct limits tells us that there is a 
	unique morphism 
	$$
	\CHu(f): \CH(Y/S,\Psi) \to \CH(X/S,\Phi)
	$$
	compatible with the refined Gysin homomorphisms.
\end{defn}

We furthermore see that this homomorphism $\CHu(f)$ respects the grading by 
codimensions.

The refined Gysin is functorial and it is clear that $\CHu(id) = id$ so
$$
\CHu(f):(\vu)^{op} \to \Grab
$$
is a functor.

\subsection{$\CH = (\CHl,\CHu,\times_S,1)$ is a weak cohomology theory with supports} \, \\
We start by defining the unit $1$, and the product $\times_S$ for Chow groups. 

\begin{defn} 
	The \textit{unit} is the group homomorphism 
	\begin{align*}
		1: \Z &\to \CH(S/S), \\
		1 & \mapsto [S].
	\end{align*}
	Recall the definition of the exterior product from \cite[\S 20.2]{Fultonbook}. 
	Let $(X,\Phi)$ and $(Y,\Psi)$ be $\NS$-schemes 
	with families of supports and let $V \in \Phi$
	and $W \in \Psi$ be integral. The product
	$[V] \times_S [W] \in \CH(X\times_S Y/S,\Phi\times_S\Psi)$ is given by
	$$
	[V] \times_S [W] =
	\begin{cases}
		[V\times_S W], & \text{if}\ V\ \text{or}\ W \ \text{is flat over}\ S,\\
		0, & \text{otherwise.}
	\end{cases}
	$$
\end{defn}
This definition doesn't mention
the gradings on the Chow groups, however to 
see that Chow groups give an example of 
a WCTS we have to show 
that it respects both gradings.  For that we need the following lemma.
\begin{lem} \label{SuslinVoevodsky}
	Let $W\to S$ and $V \to S$ be integral 
	and of finite type and assume
	$S$ is irreducible and Noetherian of Krull
	dimension 0 or 1. Then $W\times_S V$ has 
	pure $S$-dimension.
\end{lem}
\begin{proof}
	If $\dim(S) = 0$ then $S = \Spec(K)$ for some 
	field and the product of two irreducible algebraic
	schemes is again irreducible.
	
	If $\dim(S) = 1$ we have three possibilities.
	\begin{enumerate}[i)]

		\item $W$ maps to the closed point $x \in S$,
		$V$ maps to the closed point $y \in S$ and 
		$x \neq y$,
		\item both $W$ and $V$ map to the same closed 
		point $s \in S$, or
		\item One of the $W, V$ is dominant over $S$.
	\end{enumerate}
	Possibility $i)$ is trivial since
	in this case we have  $W \times_S V = \emptyset$. If possibility
	$ii)$ holds then by the $0$-dimensional case above
	we have that $W\times_S V\to \Spec(k_s)$ is 
	equidimensional (where $k_s$ is the residue field
	of the image $s\in S$) and by \cite[Proposition 2.1.3.(v)]{AndreasWeberThesis} we have
	for any irreducible component $Z$ of $W \times_S
	V$ that
	$$
	\dim_S(Z) = \dim_S(s) + \dim_s(Z) = \dim_s(Z) - 1
	$$
	so $W\times_S V$ has pure $S$-dimension. 
	If possibility $iii)$ holds, then without loss
	of generality we may assume $V\to S$ is dominant
	and hence flat. If $W\times_S V = \emptyset\footnotemark$ then 
	it is again trivially of pure $S$-dimension. \footnotetext{This can happen
		if $W$ is not dominant over $S$ and lies in a fibre over a point
		that is not in the image of $V \to S$.} So we assume $W\times_S V \neq 
	\emptyset$ and let $\eta$ be the generic
	point of $S$. Then the generic fiber
	$V_{\eta} = V\times_S \Spec(k_{\eta})$ is irreducible and the projection
	$
	V\times_S \Spec(k_{\eta}) \to \Spec(k_{\eta})
	$
	is equidimensional. By \cite[Proposition
	2.1.8]{SuslinVoevodsky} we have that 
	$V \to S$ is universally equidimensional of 
	dimension $r:= \dim(V_{\eta})$. This
	implies that the projection
	$
	W\times_S V \to W 
	$
	is equidimensional of dimension $r$. Now consider
	any irreducible component $Z$ of $W\times_S V$. We
	know that $r = \dim(Z_{\mu_W})$, where $\mu_W$ is the generic 
	point of $W$\footnote{See discussion after Definition 2.1.2 in 
		\cite{SuslinVoevodsky}.}, and by \cite[Proposition 2.1.3.(vii)]{AndreasWeberThesis} we have 
	that $\dim(Z_{\mu_W}) = \dim_W(Z)$ so by \cite[Proposition 2.1.3.(v)]{AndreasWeberThesis} we
	get
	\begin{align*}
		\dim_S(Z) &= \dim_S(W)+\dim_W(Z) \\
		&= \dim_S(W)+r.
	\end{align*}
\end{proof}

\begin{cor} \label{ChowExteriorRespectsGrading}
	The exterior product $\times_S$ respects both 
	gradings on the Chow groups.
\end{cor}
\begin{proof}
	We first consider the covariant grading. Let 
	$(X,\Phi), (Y,\Psi)$ be $\NS$-schemes with 
	supports and let $V\in \Phi$ and $W \in \Psi$ be
	integral, and say $\dim_SV=i$ and $\dim_SW = j$.
	Then $[V] \in \CH_{2i}(X/S,\Phi)$ and $[W] \in 
	\CH_{2j}(Y/S,\Psi)$ and we want to show 
	that 
	$$
	[V]\times_S [W] \in \CH_{2i+2j}(X\times_S Y/S,
	\Phi\times_S \Psi).
	$$
	If neither $V$ nor $W$ is 
	flat over $S$, then $[V]\times_S [W] = 0$ which lies
	in $\CH_{2i+2j}(X\times_S Y/S,\Phi\times_S \Psi)$.
	If, without loss of generality, $V\to S$ is flat,
	then $[V]\times_S [W] = [V\times_S W]$. Either 
	$V\times_SW = \emptyset$ and then
	$$
	[V]\times_S [W] = 0 \in \CH_{2i+2j}(X\times_S Y/S,
	\Phi\times_S \Psi),
	$$
	or  $V\times_S W \neq \emptyset$ and in this case  
	we see,
	by Proposition \ref{SuslinVoevodsky}, that
	$V\times_S W$
	has pure $S$-dimension equal to $\dim_S W + r$ 
	where $V_{\eta}$ is 
	the generic fiber and $r = \dim V_{\eta}$. By \cite[Proposition 2.1.3.(vii)]{AndreasWeberThesis} we have 
	$r = \dim V_{\eta} = \dim_S V - \dim_S S = \dim_S V$, and so $[V\times_S W] \in  \CH_{2i+2j}(X\times_S Y/S,
	\Phi\times_S \Psi)$.
	
	For the contravariant grading, we assume
	that as before we have $\NS$-schemes with
	families of supports $(X,\Phi)$ and $(Y,\Psi)$.
	We may assume that both $X$ and $Y$ are connected.
	Furthermore we assume we have integral 
	$V\in \Phi$ and $W \in Y$ such that $[V] \in 
	\CH^{2i}(X/S,\Phi)$ and $[W] \in \CH^{2j}(Y/S,\Psi)$.
	By definition this means that $\codim(V,X) = i$ 
	and $\codim(W,Y)=j$ and by Lemma 
	\ref{Chowrelationbetweengradings} we have 
	$[V] \in \CH_{2\dim_S X-2i}(X/S,\Phi)$ and 
	$[W]\in \CH_{2\dim_S X-2j}(Y/S,\Psi)$. As before,
	the case where neither $V\to S$ nor $W \to S$ 
	are flat is trivial, so we assume without loss
	of generality that $V\to S$ is flat. By what
	we showed above, we have 
	$$
	[V\times_S W] \in 
	\CH_{2\dim_S X-2i+2\dim_S Y-2j}(X\times_S Y/S, 
	\Phi\times_S \Psi).
	$$
	By Proposition \ref{SuslinVoevodsky}, both
	$X\times_S Y$ and $V\times_S W$ have pure
	$S$-dimension equal to $\dim_S X + \dim_S Y$
	and $\dim_S V+ \dim_S W$ respectively. Because $X\times_S Y$ and $V\times_S W$
	are of pure $S$-dimension, we can restrict to looking
	at one irreducible component $[T]$ of $V\times_SW$
	which lies inside an irreducible component $[Z]$ of 
	$X\times_S Y$, and $[V\times_S W]$ will lie inside
	the graded piece $\CH^c(X\times_SY/S, \Phi\times_S \Psi)$ where $c = \codim(T,Z)$. 
	That is to say, we may restrict to the case where $X\times_S Y$ is connected and 
	then apply
	Lemma \ref{Chowrelationbetweengradings} and get
	\begin{align*}
		[V]\times_S [W] &\in 
		\CH_{2\dim_S X-2i+2\dim_S Y-2j}(X\times_S Y/S, 
		V\times_S W)  \\
		&=\CH_{2\dim_S X\times_S Y-(2i+2j)}(X\times_S Y/S,V\times_S W) \\
		&= \CH^{2i+2j}(X\times_S Y/S, V\times_S W) \\
		&\subset \CH^{2i+2j}(X\times_S Y/S, \phi \times_S \Psi).
	\end{align*}
\end{proof}

It is now easy to show that 

\begin{prop}
	The quadruple $(\CHl,\CHu,\times_S,1)$ is a 
	weak cohomology theory with supports.
\end{prop}

\section{Hodge Cohomology as a Weak Cohomology Theory With Supports}

\subsection{Objects and Grading}
Let $(X,\Phi)$ be an $\NS$-scheme with a family of supports
$\Phi$. We define
$$
H(X,\Phi) = \bigoplus_{i,j}H_{\Phi}^i(X,\omxs^j),
$$
and call this $\Gamma(S,\sO_S)$-module \textit{the
	Hodge cohomology of $X$ with supports in $\Phi$}. We denote by $\hu(X,\Phi)$
the graded abelian group given in degree $n$ by
$$
H^n(X,\Phi) = \bigoplus_{i+j=n} H^i_{\Phi}(X,\omxs^j).
$$ 
We also want a ``covariant grading''. Let $X = \amalg_r X_r$ be the 
decomposition of $X$ into its connected components, then we define 
$\hl(X,\Phi)$ to be the graded abelian group that in degree $n$ is
$$
H_n(X,\Phi) = \bigoplus_r H^{2\dim_S X_r - n}(X_r,\Phi).
$$
\begin{defn} \label{definitionofeforhodge}
	We define a morphism of abelian groups $e: \Z \to H(S,S)$ via the canonical
	ring homomorphism
	$$
	\Z \to \Gamma(S,\sO_S) = H^0(S,\sO_S) \subset H(S,S).
	$$
\end{defn}

\subsection{Pullback}
In this section we want to define a pullback in Hodge cohomology, 
so extend the map of objects $\hu$ to a functor
$$
\hu: \vu\to \Grab
$$ 

We start with a lemma telling us that the functor
$\ugps$ commutes with direct
images.
\begin{lem} \label{hodgepullbacklemmadirectimagecommutes}
	Let $Y$ be a smooth $S$-scheme of finite type with
	a family of supports $\Psi$, let $X$ be a smooth
	$S$-scheme, and let $f:X\to Y$
	be a morphism of $S$-schemes of finite type. Then 
	we have an equality 
	$$
	\ugps\circ f_{\ast} = f_{\ast}\circ \ugpsinv
	$$
	of functors $\Sh(X)\to \Sh(Y)$, or $\qcx \to 
	\qcy$.  
\end{lem}
\begin{proof}
	We prove this here for functors $\Sh(X) \to \Sh(Y)$, the case for $\qcx \to \qcy$ is 
	the same with $j^{-1}$ replaced by $j^{\ast}$.
	We start by proving this when the support is a
	closed subset $Z \subset Y$. 
	We denote the compliment $Y\setminus Z$ by $U$ 
	and the canonical open immersion $U \to Y$ by
	$j$. Then for any sheaf $\sF$ of abelian groups 
	on $Y$ we have an exact sequence
	\begin{equation} \label{exactsequencelemmahodgepullback1}
		0 \to \ugz(\sF) \to \sF \to j_{\ast}j^{-1}\sF
	\end{equation}
	Let $\sG$ be a sheaf of abelian groups on $X$.
	We on the one hand plug $f_{\ast}\sG$ in for $\sF$ into 
	\eqref{exactsequencelemmahodgepullback1} and on the other hand 
	apply the left-exact functor $f_{\ast}$ to
	\begin{equation}
		\label{exactsequencelemmahodgepullback2}
		0 \to \ugzinv(\sG) \to \sG \to j^{\prime}_{\ast}(j^{\prime})^{-1}\sG
	\end{equation}
	which is the analog of \eqref{exactsequencelemmahodgepullback1}
	for sheaves on $X$. Now X has support $f^{-1}(Z)$ and $j^{\prime}:X\setminus f^{-1}(Z) \hookrightarrow X$ is the canonical open immersion. 
	In light of \eqref{exactsequencelemmahodgepullback1} and 
	$f_{\ast}$ applied to \eqref{exactsequencelemmahodgepullback2} we see 
	that to show $\ugz  \circ f_{\ast} = f_{\ast}\circ \ugzinv$,
	it suffices to see $j_{\ast}j^{-1}f_{\ast}\sG = 
	f_{\ast}j^{\prime}_{\ast}(j^{\prime})^{-1}\sG$
	for all sheaves $\sG$ of abelian groups on $X$,
	and this is clear.
	
	The desired result
	$
	\ugps \circ \fpush = \fpush \circ \ugpsinv
	$
	follows from taking the direct limit $\varinjlim_{Z \in \Psi}$.
\end{proof}

Let $f:(X,\Phi) \to (Y,\Psi)$ be a morphism in $\vu$. 
For any $j\geq 0$ we have a map
\begin{align*}
	\omys^j &\to \fpush\omxs^j \\
	a\cdot db &\mapsto f^{\ast}(a)\cdot df^{\ast}(b) 
\end{align*}
and a natural map
$$
\fpush\omxs^j \to \rfpush\omxs^j.
$$
Applying $\rugps$ to the composition gives us a map
$$
\rugps\omys^j \to \rugps\rfpush\omxs^j.
$$
If $\sH$ is a flasque sheaf of abelian groups on $X$ then $\fpush\sH$ is
a flasque sheaf on $Y$ and flasque
sheafs are acyclic for $\ugps$ so we have, see
\cite[Theorem III.7.1]{gelfman} for example,
$$
\mathrm{R}(\ugps \circ \fpush) = \rugps\circ \rfpush.
$$
Similary, $\ugpsinv\sH$ is flasque
if $\sH$ is flasque and flasque sheaves are acyclic for the direct image
so we have 
$$
\mathrm{R}(\fpush \circ \ugpsinv) = \rfpush\circ \rugpsinv.
$$
Combining this with Lemma \ref{hodgepullbacklemmadirectimagecommutes},
enlarging supports, and applying
$\mathrm{R}\Gamma(Y,-)$ 
we obtain a map 

\begin{align*}
	\mathrm{R}\Gamma(Y,\rugps\omys^j) &\to \mathrm{R}\Gamma(Y,\rfpush\rugph\omxs^j) \\
	&= \mathrm{R}\Gamma(X,\rugph\omxs^j).
\end{align*}
Summing over the induced maps on the cohomology groups gives the desired map

$$
\hu(f): H(Y,\Psi) \to H(X,\Phi).
$$

The map $\hu(f)$, constructed above, is functorial.
\begin{prop}
	Let $(X,\Phi)$, $(Y,\Psi)$ and $(Z,\Xi)$ be smooth $S$-schemes of finite type 
	and let $f: (X,\Phi) \to (Y,\Psi)$  and $g: (Y,\Psi) \to 
	(Z,\Xi)$ be morphisms in $\vu$. Then 
	\begin{enumerate}[i)]
		\item $\hu(id): H(X,\Phi) \to H(X,\Phi)$ is the identity homomorphism.
		\item $\hu(g\circ f) = \hu(f)\circ \hu(g)$ as morphisms $H(Z,\Xi) \to H(X,\Phi)$. 
	\end{enumerate}
\end{prop}
\begin{proof}
	
	\begin{enumerate}[i)]
		\item This is clear.
		\item This follows easily from the definition of 
		$\hu$ and Lemma 
		\ref{hodgepullbacklemmadirectimagecommutes}.
	\end{enumerate}
\end{proof}

\subsection{Pushforward}

By assumption $S$ is Noetherian, regular and 
has Krull-dimension at most 1. It is therefore
Gorenstein of finite Krull dimension and $\sO_S$
is a dualizing complex for $S$. Furthermore any 
smooth scheme $X$ of finite type over $S$ is 
also Gorenstein and of finite Krull dimension
so $\pi_X^{!}\sO_S$ is a dualizing complex for $X$, 
where $\pi_X:X\to S$ is the structure map.

\subsubsection{A Pushforward Map for Proper Morphisms}  \, \\

Assume
we have a diagram of separated, finite type $S$-schemes
$$
\xymatrix{
	X \ar[rr]^f\ar[dr]_{\pi_X} && Y \ar[dl]^{\pi_Y} \\
	& S}
$$
where $f$ is a proper morphism.
We want to be careful with labeling morphisms so 
we recall the following notation:
\begin{nota} \label{notation1} 
	\,
	\begin{itemize}
		\item $c_{f,g}: (gf)^{!}\xrightarrow{\cong} f^{!}g^{!}$. \,\,\,\,\, \text{(See \cite[(3.3.14--3.3.15)]{Conrad})}
		\item  $\Tr_f:\rfpush \fushr \to id$ is the 
		trace map. \,\,\,\,\, \text{(See \cite[\S 3.4.]{Conrad})}
		\item $\beta_u: u^{\ast}\rshHom(-,-) \xrightarrow{\cong} \rshHom(u^{\ast}(-),u^{\ast}(-))$ is the natural
		isomorphism for any \'etale $u$.
		\item $e_f:f^{\#}\xrightarrow{\cong} f^!$ for
		any separated smooth $f$. \,\,\,\,\, \text{(See \cite[(3.3.21)]{Conrad})}\footnote{
			Note that when $u$ is \'etale, then it is clear from the definition of 
			$u^{\#}$ that $u^{\#} = u^{\ast}$, see \cite[(2.2.7)]{Conrad}. In this case
			we also write $e_u: u^{\ast}\xrightarrow{\cong} u^!$ for this isomorphism.}
		\item $h_u:u^{\ast}\omys^k \xrightarrow{\cong} 
		\omxs^k$ for any \'etale $S$-morphism $u:X\to Y$.
	\end{itemize}
\end{nota}

Let $D_X(\sF)$ denote $\rshHomx(\sF,\pi_X^{!}\sO_S)$.
We define a pushforward
$$
\fpush: \rfpush D_X(\omxs^k) \to D_Y(\omys^k)
$$
for any $k\geq 0$ as the composition

\begin{align} \label{definitionofproperpush}
	\rfpush D_X(\omxs^k) 
	&\xrightarrow{c_{f,\pi_Y}} 
	\rfpush\rshHomx(\omxs^k, \fushr\pi_Y^{!}\sO_S) 
	\\
	&\to \rshHomy(\rfpush \omxs^k,\rfpush\fushr\pi_Y^!\sO_S) 
	\notag\\
	&\xrightarrow{(f^{\ast})^{\vee}} \rshHomy(\omys^k,\rfpush\fushr\pi_Y^!\sO_S) 
	\notag \\
	&\xrightarrow{\Tr_f}
	D_Y(\omys^k). \notag
\end{align}

The following proposition tells us that this
pushforward is functorial.

\begin{prop}(\cite[Proposition 2.2.7.]{KayAndre}) 
	\label{properpushpropositionkayandre}
	\begin{enumerate}[(a)]
		\item $id_{\ast} = id$.
		\item Let $f:X\to Y$ and $g:Y\to Z$ be two
		proper morphisms of $\NS$-schemes. Then 
		$$
		(g\circ f)_{\ast} = \gpush\circ \rgpush(\fpush): \rgpush\rfpush D_X(\omxs^k)
		\to D_Z(\omzs^k).
		$$
		\item Let 
		$$
		\xymatrix{
			X'\ar[r]^-{u'} \ar[d]_-{f'} &
			X\ar[d]^-f \\
			Y'\ar[r]^-{u} &
			Y}
		$$
		be a Cartesian square of separated, finite type $S$-schemes with $f$ proper, $u$ 
		\'etale and $X$ of pure $S$-dimension $d$. 
		Then the following diagram commutes
		$$
		\xymatrix{
			u^{\ast}\rfbpush D_X(\omxs^k) 
			\ar[r]^-{u^{\ast}(\fpush)} \ar[d]_-{\cong} &
			u^{\ast}D_Y(\omys^k) \ar[d]^-{\cong} \\
			\rfpush'D_X(\omega_{X'/S}^k) 
			\ar[r]^-{f_{\ast}'} &
			D_{Y'}(\omega_{Y'/S}^k).}
		$$
		The left vertical isomorphism is given by
		\begin{equation} 
			\label{leftverticalisomorphism}
			c_{u',\pi_Y}^{-1}\circ e_{u'}\circ
			(h_{u'}^{\vee})^{-1} \circ
			\beta_{u'} \circ \alpha_{u,f}
		\end{equation}
		where 
		$$
		\alpha_{u,f}: u^{\ast}\rfpush 
		\xrightarrow{\cong}
		\rfpush(u')^{\ast},
		$$
		and the right vertial isomorphism is given by
		\begin{equation}
			\label{rightvertialisomorphism}
			c_{u,\pi_Y}^{-1}\circ e_u  \circ 
			(h_u^{\vee})^{-1}\circ \beta_u.
		\end{equation}
	\end{enumerate}
\end{prop}

\begin{proof}
	The proof in our relative case is exactly like the
	proof in \cite[2.2.7.]{KayAndre} with the obvious change
	that the definition of the residual complex 
	(in the proof of \cite[Lemma 2.2.12.]{KayAndre}) is 
	defined as 
	$$
	K = \pi_Y^{\Delta}\sO_S.
	$$
\end{proof}

\subsubsection{General Pushforward}
Now we look at the case of a morphism 
$$
f:(X,\Phi)\to (Y,\Psi)
$$
in $\vl$. That is, we have a morphism $f$ of $\NS$-schemes 
such that $f|_{\Phi}$ is proper and $f(\Phi) 
\subseteq \Psi$. As before
we denote the $S$-dimension of $X$ by $d_X$, the 
$S$-dimension of $Y$ by
$d_Y$ and the relative $S$-dimension of $f$ by 
$r=d_X-d_Y$.

Recall the Nagata compactification theorem.

\begin{thm*} \label{Nagata}
	Let $X$ be a separated $S$-scheme of finite type 
	with 
	$S$ quasi-compact and quasi-separated. Then there 
	exists an open immersion
	of $S$-schemes $X\to \bx$ such that $X$ is a 
	dense open in $\bx$ and 
	$\bx \to S$ is proper. Furthermore we may chose 
	$\bx$ to be reduced.
\end{thm*}
See \cite{Nagata} for a proof in the Noetherian 
case, 
using valuation theory,
and \cite{ConradNagata} for the more general case, 
using scheme-theoretic 
methods.
\vspace{1em}

We consider the Nagata compactification for the 
$Y$-scheme $f:X\to Y$ and 
obtain a $Y$-morphism $j:X\to \bx$ where $\bar{f}: 
\bx \to Y$ is proper and 
$\bx$ is reduced. Since 
$j:X\to \bx$ is a separated morphism of finite type 
over the Notherian
base $Y$ and since each $Z\in \Phi$ is proper over 
$Y$, the image 
$j(Z) \subset \bx$, with the induced subscheme 
structure, is a proper 
subscheme over $Y$ via $\bar{f}:\bx\to Y$. We can 
then view $\Phi$ as 
a family of supports on $\bx$ and the morphism 
$\bar{f}$ 
$$
\bar{f}:(\bx,\Phi) \to (Y,\Psi)
$$
in $\vl$. Furthermore, the structure morphism 
$\bar{\pi}:\bx \to S$ is 
flat. If $\dim S = 0$ this is trivial. If $\dim S = 
1$ then $\bar{\pi}$
is flat if and only if each generic point of $\bx$ 
is 
sent to the 
generic point of $S$. But $X\subset \bx$ is an open 
dense subset 
so any generic point of $\bx$ lies in $X$. The 
morphism $\pi_X:X\to S$ is 
flat by assumption and $\bar{\pi}\circ j = \pi_X$ so 
any generic point
$\eta$ of $\bx$ is sent by $\bar{\pi}$ to 
$\pi_X(\eta)$ which is 
the generic point of $S$.

\vspace{1em}
Our aim is to construct a morphism 
$$
H^i_{\Phi}(X,\omxs^j) \to H^{i-r}_{\Psi}(Y,\omys^{j-r})
$$
Note that we have a morphism
\begin{align*} \label{hodgepushforwardomegastohom}
	m_X:\omxs^j &\to \rshHomx(\omxs^{d_X-j},\omxs^{d_X})\\
	\alpha &\mapsto (\beta \mapsto \alpha \wedge 
	\beta) \notag
\end{align*}
for any $j$, which is an isomorphism since $\pi_X:X\to 
S$ is smooth. Furthermore, again since $X\to S$ is 
smooth, we have an isomorphism \footnote{The maps $m_X$
	and $l_X$ of course depend on
	$j$ so the notation is a bit ambiguous, but this shouldn't cause any problems.}
\begin{equation*} \label{hodgepushforwardomegaospiuppershriek}
	l_X:\omxs^{d_X} \xrightarrow{\cong} \pixushr \sO_S[-d_X]
\end{equation*}
and combining these, we have an isomorphism
\begin{equation*} \label{selfdualityisomorphism}
	\omxs^j \xrightarrow{\cong} D_X(\omxs^{d_X-j})[-d_X].	
\end{equation*}

Consider the following composition:
\begin{align*} \label{definitionofpushfirstcomposition}
	D_X(\omxs^{d_X-j})[-d_X]
	&\xrightarrow{c_{j,\pi_{\bx}}} 
	\rshHomx(\omxs^{d_X-j},j^!\pixbushr\sO_S)[-d_X]\\
	&\xrightarrow{e_j^{-1}}
	\rshHomx(\omxs^{d_X-j},j^{\ast}\pixbushr\sO_S)[-d_X] \notag\\
	&\xrightarrow{h_j^{\vee}} 
	\rshHomx(j^{\ast}\omxsb^{d_X-j},j^!\pixbushr\sO_S)[-d_X] \notag\\
	& \xrightarrow{\beta_j^{-1}}
	j^{\ast}D_{\bx}(\omxsb^{d_X-j})[-d_X], \notag
\end{align*}
where $h_j: j^{\ast}\omxsb^k \to \omxs^k$ is the 
canonical restriction isomorphism for any $k \geq 0$.
Taking the $i$-th cohomology with supports $\Phi$, using
excision, and enlarging supports
gives us a morphism
\begin{equation} \label{definitionofpush0}
	H^i_{\Phi}(\omxs^j) \to H^{i-d_X}_{f^{-1}(\Psi)}(D_{\bx}(\omxsb^{d_X-j})).
\end{equation}

By Lemma \ref{hodgepullbacklemmadirectimagecommutes}
we have 
\begin{equation} \label{definitionofpush4}
	H^{i-d_X}_{f^{-1}(\Psi)}(D_{\bx}(\omxsb^{d_X-j})) 
	= H^{i-d_X}_{\Psi} (\rfbpush D_{\bx}(\omxsb^{d_X-j})).
\end{equation} 
We now use the pushforward for the proper map $\bar{f}$ 
that we constructed in \ref{definitionofproperpush}
to obtain
\begin{equation} \label{definitionofpush5}
	H^{i-d_X}_{\Psi}(\rfbpush D_{\bx}(\omxsb^{d_X-j}))
	\xrightarrow{\fbpush}
	H^{i-d_X}_{\Psi}(D_Y(\omys)^{d_X-j}).
\end{equation}
Finally we use that $\pi_Y:Y\to S$ is smooth to make 
the identification
\begin{align} \label{definitionofpush6}
	H_{\Psi}^{i-d_X}(D_Y(\omys^{d_X-j})) &= 
	H_{\Psi}^{i-r}(\rshHomy(\omys^{d_X-j},\omys^{d_Y})) \\
	&= H^{i-r}_{\Psi}(\omys^{j-r}) \notag.
\end{align}
The composition of \eqref{definitionofpush0} - \eqref{definitionofpush6} gives
$$ 
H^i_{\Phi}(\omxs^j)
\to
H^{i-r}_{\Psi}(\omys^{j-r}),
$$
which is the pushforward, after we sum over all $i$'s 
and $j$'s.

\vspace{1em}

Now that we have this definition of the 
pushforward, there are two important
things we need to show:
\begin{enumerate}[i)]
	\item That this pushforward is well defined. 
	Namely, in the definition we make a choice of a 
	compactification and we need to show
	that the pushforward is independent of this 
	choice.
	\item That this pushforward is functorial.
\end{enumerate}
These are proven in \cite[Proposition 2.3.3.]{KayAndre} and the discussion above it.
\vspace{1em}

\subsection{Hodge Cohomology as a Weak Cohomology Theory with Supports}

	\subsubsection{K\"unneth Morphism}
	
	We wish to construct a map
	$$
	T:H(X,\Phi) \otimes H(Y,\Psi) \to H(X\times_S Y, \Phi\times_S \Psi),
	$$
	for any $\NS$-schemes with supports $(X,\Phi)$ and $(Y,\Psi)$.  We do this by defining a 
	map 
	\begin{equation} \label{Kunnethequation1}
		\times: H^n_{\Phi}(X,\Omega_{X/S}^i)\times H^m_{\Psi}(Y,\Omega_{Y/S}^j)
		\to H^{n+m}_{\Phi\times\Psi}(X\times_S Y,\Omega_{X\times_SY/S}^{i+j}),
	\end{equation}
	and then defining 
	$$
	T(\alpha_{n,i}\otimes\beta_{m,j}) = (-1)^{(n+i)m}(\alpha_{n,i}\times\beta_{m,j}).
	$$
	The map \eqref{Kunnethequation1} is defined as a composition
	\begin{align*}
		H^n_{\Phi}(X,\Omega_{X/S}^i)&\times H^m_{\Psi}(Y,\Omega_{Y/S}^j) \\ &\xrightarrow{\hu(p_1)\times \hu(p_2)}
		H^n_{\Phi \times_S Y}(X\times_S Y,\Omega_{X\times_X Y/S}^i)\times H^m_{X\times_S\Psi}(X\times_S Y,\Omega_{X\times_S Y/S}^j) \\
		&\xrightarrow{t'} H^{n+m}_{\Phi\times_S\Psi}(X\times_S Y, \Omega_{X\times_S Y/S}^i\otimes_{\sO_{X\times_S Y}}^L \Omega_{X\times_S Y/S}^j) \\
		&\xrightarrow{m}  H^{n+m}_{\Phi\times_S\Psi}(X\times_S Y, \Omega_{X\times_S Y/S}^{i+j}),
	\end{align*}
	where the first map is induced by the projections, 
	\begin{align*}
		p_1:X\times_S Y &\to X, \,\,\, \text{and}\\
		p_2:X\times_S Y&\to Y,
	\end{align*}
	and the map $m$ is induced by the wedge product. It is the map $t'$ that we wish to construct. 
	We first construct it for the case where $\Phi = \{V\}$ and $\Psi = \{W\}$, the general
	case follows from this by
	taking colimits in cohomology.
	Let $X$ be some $\NS$-scheme, $V,W$ some closed subsets in $X$ and $\sF$ and 
	$\sG$ be $\sO_X$-modules. Then to find $t'$ it is sufficient to find a map  
	\begin{equation*}
		\Hom_{D(X)}(\sO_X, \rug_V(\sF^{\bullet}))\times \Hom_{D(X)}(\sO_X, \rug_W(\sG^{\bullet}))
		\to \Hom_{D(X)}(\sO_X,\rug_{V\cap W}(\sF^{\bullet}\otimes_{\sO_X}^L\sG^{\bullet})),
	\end{equation*}
	for any complexes $\sF^{\bullet}$ and $\sG^{\bullet}$ of $\sO_X$-modules. To construct $t'$ we then
	use this construction specifically for $\sF^{\bullet} = \Omega_{X\times_S Y/S}^i[n]$ and 
	$\sG^{\bullet} = \Omega_{X\times_S Y/S}^j[m]$. 
	That is, we wish to construct a map
	\begin{equation}\label{Kunnethequation2}
		\sO_X\to \rug_{V\cap W}(\sF^{\bullet}\otimes_{\sO_X}^L\sG^{\bullet}),
	\end{equation}
	from given maps 
	\begin{align*}
		\sO_X &\to\rug_V(\sF^{\bullet})  \,\,\, \text{and} \\
		\sO_X &\to \rug_W(\sG^{\bullet}).
	\end{align*}
	This is essentially just the derived tensor product. Namely, we have a natural map
	$$
	\sO_X \cong \sO_X \otimes_{\sO_X}^L \sO_X \to \rug_V(\sF^{\bullet})\otimes_{\sO_X}^L \rug_W(\sG^{\bullet}),
	$$
	so constructing \eqref{Kunnethequation2} boils down to showing that there 
	exists a natural map
	\begin{equation} \label{Kunnethequation3}
		\rug_V(\sF^{\bullet})\otimes_{\sO_X}^L \rug_W(\sG^{\bullet}) \to \rug_{V\cap W}(\sF^{\bullet}\otimes_{\sO_X}^L\sG^{\bullet}).
	\end{equation}
	
	Let $j_{V}: U_V := X\setminus V \hookrightarrow X$ be the open immersion. 
	Then for any $C^{\bullet} \in D(X)$ we have an exact triangle
	\begin{equation*}
		\rug_V(C^{\bullet}) \to C^{\bullet}  \to Rj_{V\ast}j_V^{\ast}(C^{\bullet}),
	\end{equation*}
	and similarly for $W$ and $V\cap W$. Now let us consider specifically
	$$
	C^{\bullet} = \rug_V(\sF^{\bullet})\otimes_{\sO_X}^L\rug_W(\sG^{\bullet}).
	$$
	Then $j_V^{\ast}(C^{\bullet})
	= 0$ because $j_V^{\ast}$ commutes with the derived tensor product and 
	$j_V^{\ast}\rug_V(\sF^{\bullet}) = 0$, and therefore it follows from the exact triangle that 
	$ \rug_V(C^{\bullet}) = C^{\bullet}$. Similarly we have $\rug_W(C^{\bullet})
	= C^{\bullet}$. To construct \eqref{Kunnethequation3} we have
	\begin{align*} \label{Kunnethequation4}
		C^{\bullet} &= \rug_W(C^{\bullet}) \\ \notag
		&=\rug_V(\rug_W(C^{\bullet})) \\ \notag
		&=\rug_{V\cap W}(C^{\bullet}) \\ \notag
		&\to \rug_{V\cap W}(\sF^{\bullet}\otimes_{\sO_X}^L\sG^{\bullet}), \notag
	\end{align*}
	where the third equality follows from definition and the map is just the 
	composition of the natural enlarging of supports maps $\rug_V(\sF^{\bullet})\to \sF^{\bullet}$
	and $\rug_W(\sG^{\bullet}) \to \sG^{\bullet}$.

	\begin{prop} \label{symmetricmonoidalfunctors}
		The triples $(\hl, T, e)$ and $(\hu,T,e)$ define right-lax symmetric
		monoidal functors, where $e: \Z \to H(S,S)$ is the morphism defined 
		in Definition \ref{definitionofeforhodge}.
	\end{prop}
	\begin{proof}
		The proof of \cite[Proposition 2.4.1.]{KayAndre} carries over to our situation with the 
		obvious changes of dimensions to relative dimensions etc.
	\end{proof}

	\subsubsection{Summary and Pure Hodge Cohomology} \, \\
	
	The following proposition is \cite[Proposition 2.3.7]{KayAndre} and the proof is the same as
	in their case. That is; \cite[Lemma 2.2.16.]{KayAndre}, \cite[Proposition 2.2.19.]{KayAndre}, 
	\cite[Corollary 2.2.22.]{KayAndre}, and  \cite[Lemma 2.3.4]{KayAndre} all extend to our
	case with the obvious change
	that we work with the dualizing complex $\pi_Y^!\sO_S$ and not $\pi_Y^!k$, and 
	the proof of the proposition follows from these results. 
	
	\begin{prop} \label{Hodgepushpullproposition}
		Let
		$$
		\xymatrix{
			(X^{\prime},\Phi^{\prime}) \ar[r]^-{f^{\prime}}  \ar[d]_-{g_X} &
			(Y^{\prime}, \Psi^{\prime}) \ar[d]^-{g_Y} \\
			(X,\Phi) \ar[r]^-{f} &
			(Y,\Psi)
		}
		$$
		be a Cartesian square with $f,f^{\prime} \in \vl$ and $g_X, g_Y \in \vu$. Assume 
		that either $g_Y$ is flat or $g_Y$ is a closed immersion with $f$ transversal to $Y^{\prime}$. Then
		$$
		\hu(g_Y) \circ \hl(f) = \hl(f^{\prime})\circ \hu(g_X).
		$$
	\end{prop}
	
	\begin{prop}
		The quadruple $(\hl, \hu, T, e)$ is a weak cohomology theory with supports.
	\end{prop}
	\begin{proof}
		This follows from
		Proposition \ref{symmetricmonoidalfunctors}
		and  Proposition \ref{Hodgepushpullproposition}.
	\end{proof}

	We now define the \textit{pure part} of $\textrm{H}$. Namely consider for any
	$(X,\Phi) \in obj(\vl) = obj(\vu)$ 
	the graded abelian group $\hpu(X,\Phi)$  that is given in degree $2n$ as
	$$
	\textrm{HP}^{2n}(X,\Phi) = \h^n_{\Phi}(X,\Omega_{X/S}^n),
	$$
	and that is zero in odd degrees. We let $\hpl(X,\Phi)$ be the graded
	abelian group which in degree 2n equals
	$$
	\hp_{2n}(X,\Phi) = \bigoplus_{r}\hp^{2\dim_SX_r - n}(X_r,\Phi),
	$$
	where $X = \amalg X_r$ is the decomposition of $X$ into its connected components.

	We now have a quadruple $(\hpl, \hpu, T, e)$ where $T$ and $e$ are the same
	as in $(\hl,\hu,T,e)$ and this defines a WCTS and there is a natural inclusion map 
	$(\hpl, \hpu, T, e) \to (\hl,\hu,T,e)$ that clearly defines a morphism in $\TT$.

\section{Existence Theorem}

We recall from \cite{KayAndre} the following \textit{semi-purity condition}
that we need to assume in the following Existence Theorem.

\begin{defn}[Semi-Purity] \label{semipurity}
	We say that a weak cohomology theory with
	supports $\f = (\fl,\fu,T,e)$ satisfies the
	semi-purity condition if the following 
	two conditions hold.
	\begin{enumerate}
		\item For all $\NS$-schemes $X$ and 
		all irreducible closed subschemes $W \subset X$ the groups 
		$\f_i(X,W)$ vanish if $i > 2\dim_S(W)$.
		\item For all $\NS$-schemes $X$,
		closed subsets $W \subset X$ and open 
		subsets $U \subset X$ such that $U$ 
		contains the generic point of every
		irreducible component of $W$, the map
		$$
		\fu(j): \f_{2\dim_S W}(X,W) \to \f_{2\dim_S W}(U,U\cap W)
		$$
		is injective, where $j:(U,U\cap W) \to 
		(X,W)$ is induced by the open immersion
		$U \subset X$.
	\end{enumerate}
\end{defn}

\begin{thm} \label{existencetheorem}
	Let $S$ be a Noetherian, excellent, 
	regular, separated and irreducible scheme of Krull-dimension
	at most 1. Let  
	$\rmF \in \bfT$ be a weak cohomology theory with
	supports satisfying the semi-purity
	condition in Definition  \ref{semipurity}. Then $\Hom_{\bfT}(\CH,\rmF)$ is non-empty 
	if  the following conditions
	hold.
	\begin{enumerate}
		\item \label{condunit} For the $0$-section
		$\imath_0:S \to \P_S^1$ and the 
		$\infty$-section $\imath_{\infty}:S\to 
		\P_S^1$ the following equality holds:
		$$
		\rmFl(\imath_0)\circ e = \rmFl(\imath_{\infty})\circ e.
		$$
		\item \label{condonclasselement} If 
		$X$ is an $\mathcal{N}_S$-scheme and
		$W\subset 
		X$ is an 
		integral closed subscheme then there exists 
		a cycle-class element $\cl(W,X) \in 
		\rmF_{2\dim_S(W)}(X,W)$, and if $W\subset X$
		is any closed subscheme we define
		$$
		\cl(W,X) = \sum_i n_i \cl(W_i,X),
		$$
		where the $W_i$ are the irreducible components of $W$ and $\sum_i n_i [W_i]$
		is the fundamental cycle of $W$\footnotemark, such
		that the following conditions hold: 
		\footnotetext{Technically we should write $\cl(W,X) = \sum_i n_i \fl(i_{W_i})\cl(W_i,X)$ 
			where $i_{W_i}:(X,W_i)\to (X,W)$ enlarges the supports. Notice also that when $W$ is not 
			pure $S$-dimensional the cycle-class element
			$\cl(W,X)$ does not live in $\f_{2\dim_S W}(X,W)$.}
		\begin{enumerate}[i)]
			\item  \label{axiomopenimmersion} For any 
			open $U\subseteq X$ such that
			$U \cap W$ is regular, we have
			$$
			\fu(j)(\cl(W,X)) = \cl(W\cap U,U),
			$$
			where $j:(U,U\cap W) \to (X,W)$ is 
			induced by the open immersion $U\subseteq X$.
			\item \label{axiomsmoothmorphism} If $f:X 
			\to Y$ is a smooth morphism between
			$\NS$-schemes $X$ and $Y$, and 
			$W \subset Y$ 
			is a regular closed subset, then
			$$
			\fu(f)(\cl(W,Y)) = \cl(f^{-1}(W),X).
			$$
			\item \label{axiomdivisor}
			Let $i:X\to Y$ be the closed immersion 
			of an irreducible, regular, closed 
			$S$-subscheme $X$ into an 
			$\NS$-scheme $Y$. For any 
			effective smooth divisor $D \subset Y$
			such that 
			\begin{itemize}
				\item $D$ meets $X$ properly, thus 
				$D\cap X := D\times_Y X$ is a divisor
				on $X$,
				\item $D' := (D\cap X)_{\rm red}$ 
				is regular and irreducible, so 
				$D\cap X = n\cdot D'$ as divisors
				(for some $n\in \Z, n\geq 1$),
			\end{itemize}
			we define
			$g:(D,D') \to (Y,X)$ in $\vu$ as the 
			map induced by the
			inclusion $D\subset Y$. Then the following
			equality holds:
			$$
			\rmFu(g)(\cl(X,Y)) = n\cdot \cl(D',D).
			$$	
			\item \label{axiomfinite} Let $f: X \to 
			Y$ be a 
			morphism of $\NS$-schemes. Let $W\subset X$ be a 
			regular closed 
			subset such that the restricted map
			$$
			f|_W:W \to f(W)
			$$
			is proper and finite of degree $d$. Then 
			$$
			\fl(f)(\cl(W,X)) = d \cdot \cl(f(W),Y).
			$$
			\item \label{axiomproduct} Let $X,Y$ be
			$\NS$-schemes and let
			$W \subset X$ and $V \subset Y$ be
			regular, integral closed subschemes. 
			Then the following equation holds
			$$
			T(\cl(W,X)\otimes_S \cl(V,Y)) = 
			\begin{cases}
				\cl(W\times_S V,X\times_S Y) \qquad\, \text{if $W$ or $V$ is dominant over $S$},
				\\
				0 \qquad\qquad \qquad\qquad \qquad \, \,\, \text{otherwise}.
			\end{cases}
			$$
			
			\item \label{axiomone} For the base 
			scheme $S$ we have $\cl(S,S) =1_S$.	
		\end{enumerate}
	\end{enumerate}
\end{thm}

\vspace{2em}

\begin{proof}
We break the proof of the theorem into 9 steps. The first four combine to show
that given our assumptions we can construct a natural
transformation of (right-lax) symmetric monoidal 
functors $(\CHl,\times_S,1) \to (\fl,T,e)$. Then the
latter five steps consist of extending this 
natural transformation to a morphism in $\bfT$. The structure of the proof is 
the same as of \cite[Theorem 1.2.3]{KayAndre}, the corresponding existence theorem
over perfect fields of positive characteristic, although each step differs somewhat from 
loc. cit. because of the difference in the conditions assumed. In particular, in steps 1 and 4 we
see dimensional issues that do not occur in the zero-dimensional case.

We start by defining a family of homomorphisms of
Abelian groups
$$
\phik: \Zsup(X) \to \rmF(X,\Phi)
$$
indexed by the elements $(X,\Phi) \in obj(\vu) = obj(\vl)$. Now $\Zsup(X)$ is a free-Abelian group
so it suffices to give the definition of $\phik$
on the generators, which are $[W]$ for the integral
closed subschemes $W \subset X$ such that $W \in \Phi$. 
For these $[W]$ we define
$$
\phik([W]) = \rmFl(i_W)(\cl(W,X)),
$$
where $i_W: (X,W) \to (X,\Phi)$ is induced by $id_X$.

We show in four steps that this family of homomorphisms
extends to the desired natural transformation of (right-lax)
symmetric monoidal functors. In the first step
we show on the level of cycles, this family
$\phi'$ is functorial with the pushforwards. 
Step 2 is a 
technical step to be used in Step 3, wherein we
show that these morphisms $\phi'$ send 
cycles that are rationally equivalent to zero, to 
0 and therefore that the naturality diagram
from Step 1 extends from cycles to the Chow
groups and thus that this family defines a 
natural transformation $\phi$. Finally in Step 4 we 
show that this natural transformation is a 
natural transformation of right-lax symmetric
monoidal functors by showing that it respects
the unit and the product.
\vspace{1em}

\subsection*{\textbf{\textit{Step 1:}}}

We show that for any morphism 
$f:(X,\Phi) \to (Y,\Psi)$ in
$\vl$,  the following 
square commutes\footnotemark:
\begin{equation} \label{step1whattoshow}
	\xymatrix{
		\Zsup(X) \ar[r]^{\phik} \ar[d]_{f_{\ast}}
		& \rmF(X,\Phi) \ar[d]^{\rmFl(f)}\\
		\Zupp(Y)\ar[r]^{\phikk}  & \rmF(Y,\Psi),
	}
\end{equation}

\footnotetext{
	The fact that $f_{\ast}$ is well-defined is 
	clear from the definitions of $\Zsup$, $\Zupp$,
	$\vl$ and the definition of proper 
	pushforwards.}

There are two cases to cover\footnotemark 
\begin{enumerate}[i)]
	\item $\dim_S(f(W)) < \dim_S(W)$, \,and
	\item $\dim_S(f(W)) = \dim_S(W)$.
\end{enumerate}

\footnotetext{
	By \cite[Proposition 2.1.3. (iii)]{AndreasWeberThesis} we can't 
	have $\dim_S(W) < \dim_S(f(W))$.}

\begin{enumerate}[i)]
	\item Let $d:= \dim_S(W)$ and  
	$a:= \phik([W])$.
	By definition
	$$
	a := \phik([W])  = \rmFl(i_W)(cl_{(X,W)}),
	$$
	and since $\cl(W,X) \in \rmF_{2d}(X,W)$,
	and all morphisms are graded of degree 0, we 
	have 
	$$
	a \in \rmF_{2d}(X,\Phi).
	$$
	Furthermore we have a commutative square
	\begin{equation} \label{propstep1case1basicsquare}
		\xymatrix{
			\rmFl(X,W) \ar[rr]^{\rmFl(i_W)} 
			\ar[d]_{\rmFl(f)} &&
			\rmFl(X,\Phi)\ar[d]^{\rmFl(f)} \\
			\rmFl(Y,f(W)) 
			\ar[rr]_-{\rmFl(i_{f(W)})} 
			&& \rmFl(Y,\Psi).
		}
	\end{equation}
	But then $\rmFl(f)(a) = \rmFl(i_{f(W)})(\rmFl(f)(\cl(W,X)))$, 
	and by semi-purity 
	$\rmFl(f)(\cl(W,X)) \in \rmF_{2d}(Y,f(W)) = 0$,
	so 
	$$
	\rmFl(\phik([W])) = \rmFl(f)(a) = 0.
	$$
	On the other hand since $f: (X,\Phi) \to (Y,\Psi)$ is in $\vl$, it is proper when
	restricted to $W \in \Phi$, so by the 
	definition of proper pushforwards we have
	$$
	f_{\ast}([W]) = \deg(W/f(W))[f(W)] = 0
	$$
	since $\deg(W/f(W)) = 0$ because the $S$-dimension
	drops, and this shows that the square \eqref{step1whattoshow}
	commutes when $\dim_S(f(W)) < \dim_S(W)$.

	\item For an integral closed subscheme $W \in \Phi$ Lemma \ref{equaldimlemma} allows us to choose an open 
	$U\subset Y$ such that
	\begin{itemize}
					\item $U \cap f(W) \neq \emptyset$,
					\item $U\cap f(W)$ is regular,
					\item $f^{-1}(U)\cap W$ is regular, and
					\item The map induced from $f$ by 
					restriction
					$$
					f': f^{-1}(U) \cap W \to U\cap f(W)
					$$
			is finite. 
	\end{itemize}
	Consider the maps  
	\begin{align*}
		j&:(U,f(W)\cap U) \to (Y,f(W))\,\,\, \text{and} \\
		j'&:(f^{-1}(U), W\cap f^{-1}(U)) \to (X,W)
	\end{align*}
	in $\vu$ induced by the open immersions
	$U \hookrightarrow Y$ and $f^{-1}(U) \hookrightarrow X$ respectively. 
	By condition (\ref{axiomfinite}) we have 
	\begin{equation*} \label{propstep1case2eq1}
		\fl(f|_{f^{-1}(U)})(\cl(f^{-1}(U)\cap W,f^{-1}(U))) = d\cdot \cl(U\cap f(W),U)
	\end{equation*}
	where $d$ is the degree of the finite morphism.
	Condition (\ref{axiomopenimmersion}) now tells us that 
	\begin{align*} \label{propstep1case2eq2}
		\fu(j)(\cl(f(W),Y)) & = \cl(U\cap f(W),U), \,\,\, \text{and}\\
		\fu(j')(\cl(W,X)) & = \cl(f^{-1}(U)\cap W,f^{-1}(U)). \notag
	\end{align*}	
	Combining these  we 
	obtain
	\begin{equation} \label{propstep1case2eq3}
		\fl(f|_{f^{-1}(U)})(\fu(j')(\cl(W,X))) = d\cdot \fu(j)(\cl(f(W),Y)).
	\end{equation}
	Consider the fibre square
	$$
	\xymatrix{
		f^{-1}(U) \ar[d]_-{j'} \ar[rr]^-{f|_{f^{-1}(U)}} &&
		U \ar[d]^-{j} \\
		X \ar[rr]_{f} &&
		Y.
	}
	$$
	The morphism $j$ is an open immersion, hence smooth, so the from the definition of WCTS we get
	$$
	\fl(f|_{f^{-1}(U)})(\fu(j')(\cl(W,X))) = \fu(j)(\fl(f)(\cl(W,X))).
	$$
	Substituting this into equation \eqref{propstep1case2eq3} we 
	get
	\begin{align} \label{propstep1case2eq4}
		\fu(j)(\fl(f)(\cl(W,X))) &= d\cdot \fu(j)(\cl(f(W),Y)) \\
		& = \fu(j)(d\cdot \cl(f(W),Y)). \notag
	\end{align}
	We have that
	$ \fl(f)(\cl(W,X))$ and  $d\cdot \cl(f(W),Y)$ are in $\f_{2\dim_S(f(W))}(Y,f(W))$
	so by semi-purity, equation \eqref{propstep1case2eq4} implies
	\begin{equation} \label{propstep1case2eq5}
		\fl(f)(\cl(W,X)) = d\cdot \cl(f(W),Y).
	\end{equation}
	We now apply $\fl(i_{f(W)})$ to both sides of 
	\eqref{propstep1case2eq5}, where 
	$i_{(f(W))}:(Y,f(W)) \to (Y,\Psi)$ is 
	induced by $id_Y$, to obtain
	\begin{align} \label{propstep1case2eq6}
		\fl(i_{f(W)})(\fl(f)(\cl(W,X))) 
		%&= \fl(i_{f(W)})(d\cdot \cl(f(W),Y)) \\ \notag
		& = d\cdot \fl(i_{f(W)})(\cl(f(W),Y)) \\ \notag
		& = d\cdot \phic_{(Y,\Psi)}([f(W)]) \\ \notag
		& =\phic_{(Y,\Psi)}\circ f_{\ast}([W]). \notag
	\end{align}
	This last equality holds because 
	by definition we have 
	\begin{align*}
		\deg(f') &= \deg(W\cap f^{-1}(U)/f(W)\cap U)\\ 
		&:= [R(W\cap f^{-1}(U)) : R(f(W)\cap U)],
	\end{align*}
	and since $W\cap f^{-1}(U)$ is an open dense 
	subset of $W$ and $f(W)\cap U$ is an open dense
	subset of $f(W)$, we have 
	\begin{align*}
		R(W\cap f^{-1}(U)) &= R(W), \,\,\,\text{and} \\
		R(f(W)\cap U) &= R(f(W)),
	\end{align*}
	so 
	\begin{equation*}
		d = \deg(f') 
	%	&= [R(W\cap f^{-1}(U)) : R(f(W)\cap U)] \\
	%	&= [R(W):R(f(W))] \\
		= \deg(f).
	\end{equation*} 
	Furthermore, by looking at the commutative square 
	\eqref{propstep1case1basicsquare} we see that 
	\begin{align} \label{propstep1case2eq7}
		\fl(i_{f(W)})(\fl(f)(\cl(W,X))) & = \fl(f)(\fl(i_W)(\cl(W,X))) \\
		& = \fl(f)\circ \phic_{(X,\Phi)}([W]). \notag
	\end{align}
	Combining \eqref{propstep1case2eq6} and 
	\eqref{propstep1case2eq7}
	we obtain
	\begin{equation*} \label{propstep1case2finaleq}
		\fl(f)\circ \phic_{(X,\Phi)}([W]) = 
		\phic_{(Y,\Psi)}\circ f_{\ast}([W]).
	\end{equation*}

\end{enumerate}

\subsection*{\textbf{\textit{Step 2:}}} 

Now let $X$ be an $\NS$-scheme, $W \subset X$ an
integral closed subscheme, and $D$ a smooth
divisor intersecting $W$ properly, so that 
$W\cap D := W \times_X D$ is an effective Cartier
divisor on $W$. We denote by $[W\cap D]$ the 
associated Weil divisor and we denote $(D\cap 
W)_{red}$ by $D'$. The following equality is
what we want to prove
\begin{equation}\label{step2mainequation}
	\fu(i_D)(\cl(W,X)) = \cl(D\cap W,D),
\end{equation} 
where $i_D:(D,D\cap W) \to (X,W)$ in $\vu$ is the map 
induced by the closed immersion $D \subset X$.

Let $U$ be an open 
subset of $X$ that contains all the generic points of 
$D'$. 
The following diagram in $\vu$ commutes
$$
\xymatrix{
	(U\cap D,(W\cap D)\cap U) \ar[r]^-{\hat{\imath}_D}
	\ar[d]_-{\hat{\jmath}}  & (U,U\cap W) \ar[d]_-{j}\\
	(D,W\cap D) \ar[r]_-{i_D} & (X,W),
}
$$
where 
\begin{itemize}
	\item $j:(U,U\cap W) \to (X,W)$ is induced by
	the inclusion $U \subset X$,
	\item $\hat{\jmath}: (U\cap D,(W\cap D)\cap U) \to
	(D,W\cap D)$ is induced by the inclusion 
	$U\cap D \subset D$,  and
	\item $\hat{\imath}_D:(U\cap D,(W\cap D)\cap U) 
	\to 
	(U,U\cap D)$ is induced by the inclusion 
	$U\cap D \to U$.
\end{itemize}
Applying the contravariant functor $\fu$ gives us
a commutative diagram
\begin{equation} \label{propstep2square}
	\xymatrix{
		\fu(X,W) \ar[r]^-{\fu(j)} \ar[d]^-{\fu(i_D)}
		& \fu(U,W\cap U) \ar[d]^-{\fu(\hat{\imath}_D)} \\
		\fu(D,W\cap D) \ar[r]^-{\fu(\hat{\jmath})} 
		& \fu(U\cap D, (W\cap D)\cap U).
	}
\end{equation}

%\begin{lem} \label{classelementopenlemma}
%	Let $X$ be an $\NS$-scheme and $W \subseteq X$ be
%	an integral closed subscheme. Let $U\subseteq X$
%	be an open subscheme such that $U\cap W \neq
%	\emptyset$. Then
%	$$
%	\fu(j)(\cl(W,X)) = \cl(U\cap W,U).
%	$$
%\end{lem}
%\begin{proof}
%	We know since $W$ is an integral scheme over
%	an excellent base scheme $S$ that it is generically regular. The same 
%	is true for the open subset $ U\cap W \subset 
%	W$.
%	We can thus find and open subset $V \subset U$ 
%	such that $V\cap(U\cap W) = V \cap W$
%	is non-empty and regular. Consider 
%	the map induced by inclusion
%	$j_V:(V,V\cap W) \to (U,U\cap W)$.
%	Notice that since $U \cap W$ is irreducible and 
%	$V\cap W$
%	is a non-empty subset of $U \cap W$ the generic 
%	point of 
%	$U \cap W$  is contained in $V\cap W$. We also have that
%	$\fu(j)(\cl(W,X)),$ and  $\cl(U\cap W,U)$ are in   
%	$\f_{2\dim_S (U \cap W)}(U,U\cap W)$,
%	so in order to prove $\fu(j)(\cl(W,X)) = \cl(U\cap W,U),$
%	it suffices by semi-purity to prove 
%	$$
%	\fu(j_V)(\fu(j)(\cl(W,X)))
%	=\fu(j_V)(\cl(U\cap W,U)).
%	$$
%	Since $V \cap W$ is regular, condition 
%	\ref{axiomopenimmersion} gives us that
%	$\fu(j_V)(\cl(U\cap W,U)) = \cl(V\cap W,V),$
%	and since $\fu(j_V)\circ \fu(j) = \fu(j\circ j_V)$
%	where $j\circ j_V:(V,V\cap W) \to (X,W)$
%	is the morphism induced by the open immersion 
%	$V \subset X$, we have again by condition 
%	\ref{axiomopenimmersion}
%	\begin{align*}
%		\fu(j_V)(\fu(j)(\cl(W,X))) &= 
%		\fu(j\circ j_V)(\cl(W,X)) \\
%		& = \cl(V\cap W,V).
%	\end{align*}
%\end{proof}
%

We have 
$\fu(j)(\cl(W,X)) = \cl(W\cap U,U)$
and 
$\fu(\hat{\jmath})(\cl(W\cap D,D))  = \cl((W\cap D)\cap 
U,U\cap D)$ by Lemma \ref{classelementopenlemma},
so if we can prove 
$$
\fu(\hat{\imath}_D)(\cl(W\cap U,U)) = \cl((W\cap D)\cap U,U\cap D),
$$
then 
$\fu(i_D)(\cl(W,X)) = \cl(D\cap W,D)$
follows from the commutativity of the square 
\eqref{propstep2square} and by semi-purity. 
This shows that we may restrict to any open 
subset that contains all the generic points of
$D'$. Furthermore, since $X$ is 
Noetherian (being of finite type over the
Noetherian scheme $S$) we see that 
$D'$ has finitely many irreducible
components. Therefore the set $A$ of all points 
lying in an intersection of connected components
is a finite union of closed sets and is thus closed.
The set $A$ contains no generic point of $D'$
and we can therefore look
at $U \setminus A$ instead of $U$ and reduce
to the case where the irreducible components
are disjoint. Let $V_1,\ldots, V_r$ be the 
irreducible components of $D'$,
then by the definition of WCTS we have
$$
\bigoplus_{i=1}^r\f(D,V_i) \cong \f(D,W\cap D).
$$
Therefore we may assume $r=1$, i.e. that
$D'$ is irreducible with a generic point
$\eta$.

If $W$ is regular in codimension 1 then (since $D$ intersects
$W$ properly) $\sO_{W,\eta}$ is regular, i.e. $D'$
is generically regular. Then there exists some dense open 
$\tilde{U} \subset D'$ that is regular, i.e. there
exists some open $U \subset X$ such that $U \cap 
D'$
is nonempty and regular. Furthermore we may assume that $U \cap W$
is regular, since $W$ is regular in codimension 1.
By construction $\eta \in U$, so it suffices by the above 
discussion to
prove the equality for $U$, i.e. we can reduce to the 
case 
where $W$ and $D'$ are regular and 
irreducible in $X$. 
By condition (\ref{axiomdivisor})we then have
$\fu(i_D)(\cl(W,X)) = n \cdot \cl(D',D),$
where $n$ is the multiplicity of $D$ along $W$.
Furthermore, we have $n\cdot \cl(D',D) = 
\cl(D\cap W,D)$
so we finally have 
$$ 
\fu(i_D)(\cl(W,X)) = \cl(D\cap W,D).
$$

Recall that normal schemes are regular in 
codimension 1.
We take $W$ that is not necessarily 
normal, look
at its normalization which is regular in 
codimension 1, and deduce
the equation we want to show from that case.

Notice that we can find an affine open $U \subset X$ such that 
$U \cap D' \neq \emptyset$. In this case 
$U\cap D'$ is a non-empty
open subset of $D'$ and thus contains the 
generic point $\eta$.
We can therefore restrict to looking at this $U$, i.e. 
we may assume $X$ is affine. 

The normalization morphism
$\tilde{W} \to W$ is finite, hence projective,  since $W$ is excellent, see for example
\cite[Tag: 035R]{Stacks}.\footnote{Note that this lemma states this for 
	Japanese schemes, but excellent rings are Nagata rings, and Nagata rings 
	are universally Japanese. See for example \cite[Tag: 0334]{Stacks}.}
Since $W$ is closed in  $X$ and we assume $X$ is affine, we can assume $W$ is affine. Therefore there exists some $n$ so that 
we can find a closed immersion
	$$
	\tilde{W} \to W\times_S \P_S^n
	$$
where $\tilde{W}$ denotes the normalization of $W$.

see \cite[Def. 5.5.2]{EGA2}). 

We now set $\tX := X \times_S \P_S^n$, $\tD := D \times_S \P_S^n$,
$\ti:(\tD,\tW \cap \tD) \to (\tX,\tW)$ 	and 
consider the morphism $pr_1:(\tX,\tW)\to (X,W)$ 
induced by
the projection $\tX \to X$. The projection $\P^n_S \to S$ is 
always proper, so $pr_1$ is a morphism in $\vl$. By Step 1  
we have 
$$
\phi'_{(X,W)} \circ pr_{1\ast} = \fl(pr_1)\circ \phi'_{(\tX,\tW)}.
$$
We now evaluate both sides at $[\tW]$ and get 
\begin{align*}
	\cl(X,W) &= \phi'_{(X,W)}([W]) \\
	%&= \phi'_{(X,W)}([pr_1(\tW)]) \\
	&= \phi'_{(X,W)}\circ pr_{1\ast}([\tW]) \\
	&= \fl(pr_1)\circ \phi'_{(\tX,\tW)}([\tW]) \\
	&= \fl(pr_1)(\cl(\tW,\tX)).\\
\end{align*}
Applying $\fu(i)$ to both sides gives
\begin{equation}\label{propstep2eq1}
	\fu(i)(\cl(W,X)) = \fu(i)(\fl(pr_1)(\cl(\tW,\tX))).
\end{equation}
%\vspace{1em}
We have a Cartesian diagram
$$
\xymatrix{
	(\tD,\tD\cap \tW) \ar[r]^-{pr_1|_{\tD}} \ar[d]_-{\ti} & (D,W\cap D) 
	\ar[d]_-{i} \\
	(\tX,\tW) \ar[r]^-{pr_1} & (X,W).
}
$$		
The morphism $i$ is a closed immersion 
and it is transversal to $pr_1$ 
so the definition of WCTS tells us that
$$
\fu(i)\circ\fl(pr_1) = \fl(pr_1|_{\tD})\circ \fu(\ti).
$$
If we now evaluate both sides at $\cl(\tW,\tX)$ we get
\begin{equation} \label{propstep2eq2}
	\fu(i)(\fl(pr_1)(\cl(\tW,\tX))) = \fl(pr_1|_{\tD})( \fu(\ti)(\cl(\tW,\tX))).
\end{equation}
From the first case discussed, where we have 
an integral closed subscheme that is regular in
codimension 1, we have
\begin{equation} \label{propstep2eq3}
	\fu(\ti)(\cl(\tW,\tX)) = \cl(\tW\cap \tD,\tD).
\end{equation}
Note that $pr_1(\tW\cap\tD) = W \cap D$ so
since the projection
$pr_1|_{\tD}$ is proper we have
\begin{equation} \label{propstep2eq4}
	\phi'_{(D,W\cap D)}(pr_{1 \ast}([\tW \cap \tD])) 
	=
	\phi'_{(D,W\cap D)}([W\cap D]).
\end{equation}	
Combining equations \eqref{propstep2eq1}--\eqref{propstep2eq4} we 
obtain

\begin{align*}
	\fu(i)(\cl(W,X)) 
	%&= \fu(i)(\fl(pr_1)(\cl(\tW,\tX))) \\
	%&=\fl(pr_1|_{\tD})(\fl(\tilde{i})(\cl(\tD,\tX)))
	%\\
	&= \fl(pr_1|_{\tD})(\cl(\tW\cap\tD,\tD)) \\
	&= \phi'_{(D,W\cap D)}(pr_{1 \ast}([\tW \cap \tD])) \\
	&= \phi'_{(D,W\cap D)}([W\cap D]) \\
	&= \cl(W\cap D,D).
\end{align*}

\subsection*{\textbf{\textit{Step 3:}}}

Our aim here is to prove that 
$$
\phik(Rat_{\Phi}(X)) = 0.
$$
We use the ``homotopy'' definition of rational 
equivalences, see \cite[Proposition 1.6.]{Fultonbook}.
The additivity of $\phi'$ allows us to see that we can 
reduce to showing that for an irreducible 
closed subset $W \subset X\times_S \P_S^1$ such that 
$pr_1(W) \in \Phi$ and $W\to \P^1_{S_W}$ is dominant,
where $S_W$ is the closure of the image $\pi_X(W)$
in $S$, we 
have
\begin{equation} \label{aimofsection3}
	\phi{'}_{(X,pr_1(W))}([pr_1(W_0)])  
	=
	\phi{'}_{(X,pr_1(W)}([pr_1(W_{\infty})]),
\end{equation}
where we write 
$\we := W \cap (X\times_S \{\epsilon\})$, for $\epsilon \in \{0, \infty\}$.

We introduce the maps
\begin{align*}
	\ie&: (X,pr_1(W)) \to (X\times_S \P_S^1,pr_1(W)\times_S \P_S^1),  \,\,\,\text{and} \\
	\aep&: (X,pr_1(\we)) \to (X\times_S \P_S^1,W), 
\end{align*}
that are both induced by the map $X \to X\times_S \P_S^1$ given 
by the composition
$$
X \xrightarrow{\cong} X\times_S \{\epsilon\} 
\xrightarrow[immersion]{closed} X\times_S \P_S^1.
$$
The maps $\ie$ and $\aep$ are morphisms in both
$\vl$ and $\vu$.
We also define the map 
$$
\be:(X,pr_1(\we)) \to (X,pr_1(W))
$$
in $\vl$ that is induced by $id_X$.

By definition of $\phic$ we have
$\phic_{(X,pr_1(W))}([pr_1(\we)]) = \fl(\be)(\cl(pr_1(\we),X))$
and by \eqref{step2mainequation} we have 
$\fu(\aep)(\cl(W,X\times_S \P_S^1)) = \cl(pr_1(\we),X).$
Combining this we have 
\begin{equation} \label{preliminaryequation}
	\fl(\be)\circ \fu(\aep)(\cl(W,X\times_S \P_S^1)) = \phic_{(X,pr_1(W))}([pr_1(\we)]).
\end{equation}
The following square is Cartesian 
$$
\xymatrix{
	(X,pr_1(\we)) \ar[r]^{\be} \ar[d]_{\aep} & (X,pr_1(W)) \ar[d]^{\ie}\\
	(X\times_S \P_S^1,W) \ar[r]^-{\xi} & (X\times_S \P_S^1,pr_1(W)\times_S \P_S^1),
}
$$	
where $\xi$ is induced by the identity.
The map $\ie$ is a closed immersion and transversal to
the bottom identity morphism.The definition of WCTS thus gives
\begin{equation} \label{propstep3eq1}
	\fl(\be)\circ \fu(\aep) = \fu(\ie)\circ\fl(\xi), 
\end{equation}
and \eqref{preliminaryequation} becomes
$$
\fu(\ie)\circ\fl(\xi)(\cl(W,X\times_S \P_S^1)) = \phic_{(X,pr_1(W))}([pr_1(\we)]).
$$
To prove \eqref{aimofsection3} it is therefore 
sufficient to show that
$$
\fu(i_0)\circ \fl(\xi) = \fu(\iinf)\circ \fl(\xi)
$$
as maps $\f(X\times_S \P_S^1,W) \to \f(X,pr_1(W))$.

We apply the first projection formula, 
with $f_1 =: \ie', 
f_2 =: \aep$ where $\ie': X \to X\times_S \P_S^1$ is induced by the same
closed immersion as $\aep$. This makes $f_3 =: \aep$ 
as well. Now letting 
$b \in \f(X\times_S\P^1_S,pr_1(W))$ be arbitrary and $ a = 1_X$ we 
get 
$$
\fl(\aep)\circ\fu(\aep)(b) = \fl(\ie')(1_X) \cup b.
$$
If $\fl(i_0')(1_X) = \fl(\iinf')(1_X)$ then 
we have shown that 
$$
\fl(\alpha_0)\circ\fu(\alpha_0) = \fl(\alpha_{\infty})\circ \fu(\alpha_{\infty}),
$$
as maps $\f(X\times_S \P_S^1, W) \to \f(X\times_S\P_S^1,W)$.
We know that
$\xi \circ \aep = \ie\circ \be,$ 
and all these maps are in $\vl$, so we have 
$$
\fl(\xi)\circ \fl(\aep) = \fl(\ie)\circ \fl(\be).
$$
We see that
\begin{align*}
	\fl(\xi)\circ \fl(\aep)\circ\fu(\aep) &= \fl(\ie) \circ \fl(\be)\circ \fu(\aep) \\
	& = \fl(\ie)\circ\fu(\ie) \circ \fl(\xi), 
\end{align*}
by \eqref{propstep3eq1}.
So what we have shown is that 
$$
\fl(i_0)\circ\fu(i_0)\circ \fl(\xi) = \fl(\iinf)\circ\fu(\iinf) \circ \fl(\xi)
$$
follows from 
$$
\fl(i_0')(1_X) = 
\fl(\iinf')(1_X).
$$

The diagram
$$
\xymatrix{
	(X,pr_1(W))\ar[r]^-{\ie} \ar[dr]_{id} 
	& (X\times_S \P_S^1,pr_1(W)\times_S \P_S^1) \ar[d]^{pr_1}\\
	& (X,pr_1(W)),
}
$$
in $\vl$ is commutative, and we obtain
$\fl(pr_1)\circ\fl(\ie) = \fl(id)=id.$
Notice that $\fl(pr_1)$ is completely independent
of $\epsilon$ so we can apply this to both sides of 
$
\fl(i_0)\circ\fu(i_0) \circ\fl(\xi) = \fl(\iinf)\circ\fu(\iinf) \circ \fl(\xi)
$
to obtain what we want
$$
\fu(i_0) \circ \fl(\xi) = \fu(\iinf)\circ\fl(\xi).
$$

What remains to be shown in this step is the equality 
$$
\fl(i_0')(1_X) = \fl(\iinf')(1_X).
$$
First we recall that if $\pi_X:X \to S$ is the structure
morphism of $X$ then 
$
\fl(\ie')(1_X) = \fl(\ie')\circ\fu(\pi_X)\circ e(1).
$
Consider the following Cartesian diagram
$$
\xymatrix{
	X \ar[r]^-{\ie'}\ar[d]_{\pi_X} & X \times_S \P_S^1
	\ar[d]^{pr_2} \\
	S \ar[r]_-{p_{\epsilon}} & \P_S^1
}
$$
where $p_{\epsilon}:S\to \P_S^1$ is the zero- or
infinity section. Furthermore
$pr_2$ is smooth, being the base change of $\pi_X$ along 
$\pi_{\P^1}$. It follows from the definition of WCTS that
$\fl(\ie')\circ \fu(\pi_X) = \fu(pr_2)\circ \fl(p_{\epsilon}).$
We recall that condition (\ref{condunit}) says
that$
\fl(p_0)\circ e = \fl(p_{\infty}) \circ e.$
Combining this we obtain
\begin{align*}
	\fl(i_0')(1_X) &= \fl(i_0')\circ\fu(\pi_X)\circ e(1) \\
	&= \fu(pr_2)\circ \fl(p_0) \circ e(1) \\
	&= \fu(pr_2)\circ \fl(p_{\infty}) \circ e(1) \\
	&= \fl(\iinf')\circ\fu(\pi_X)\circ e(1) \\
	&= \fl(\iinf')(1_X).
\end{align*}

\subsection*{\textbf{\textit{Step 4:}}}

We want to show that 
\begin{align*}
	\phi \circ 1 &= e, \,\,\, \text{and}\\
	\phi \circ \times_S & = T\circ(\phi \otimes \phi).
\end{align*}

It is enough to show that these equations hold
on the level of cycles, i.e. to show the equations
\begin{align*}
	\phic \circ 1 &= e, \,\,\, \text{and}\\
	\phic \circ \times_S & = T\circ(\phic \otimes \phic).
\end{align*}
The first equation follows directly from the definition.
For any $n \in \Z$ we have 
\begin{align*}
	(\phic_S \circ 1)(n) &= \phic_S(n\cdot [S]) \\
	&= n\cdot \phic_S([S]) \\
	%&=n \cdot \cl(S,S) \\
	&= n \cdot 1_S \\
	&= n \cdot e(1) \\
	&= e(n).
\end{align*}

Now consider the second equation. What we want to show 
precisely is that for  $\NS$-schemes with 
supports $(X,\Phi)$ and $(Y,\Psi)$ and integral 
closed subschemes 
$W \in \Phi, V \in \Psi$ we have 
\begin{equation}\label{propstep4whattoshowfirst}
	\phic_{(X\times_S Y, \Phi\times_S \Psi)}([W]\times_S [V])
	= T(\phic_{(X,\Phi)}([W])\otimes_S\phic_{(Y,\Psi)}([V]))
\end{equation}
Let $i_W: (X,W) \to (X,\Phi)$ and $i_V:(Y,V) \to (Y,\Psi)$ be the 
maps in $\vl$
induced by the identities. Then so is 
$i_W\times_S i_V:(X\times_S Y,W\times_S V)
\to (X\times_S Y, \Phi\times_S \Psi),$
and by naturality of $T$ we have 
\begin{align*} 
	T(\phic_{(X,\Phi)}([W])\otimes_S\phic_{(Y,\Psi)}([V])) & = T(\fl(i_W)(\cl(W,X))
	\otimes_S \fl(i_V)(\cl(V,Y)))\\
	&= \fl(i_W\times_S i_V)(T(\cl(W,X)\otimes_S \cl(V,Y)) .
\end{align*}
If neither $W$ nor $V$ is flat over $S$, then $[W] \times_S [V] = 0$ by definition
and $T(\cl(W,X)\otimes_S \cl(V,Y)) = 0$ by condition (\eqref{axiomproduct}), and
therefore both sides of \eqref{propstep4whattoshowfirst} are 0. Without loss of 
generality we may assume that $W$ is flat over $S$. Then $[V]\times_S[W] = [V\times_S 
W]$ and since 
$$
\phic_{(X\times_S Y, \Phi\times_S \Psi)}([W\times_S V]) = 
\fl(i_W\times_S i_V)(\cl(W\times_S V,X\times_SY)),\footnotemark
$$
we see from the above equation 
that it is enough to show 
\begin{equation} \label{propstep4whattoshow}
	\cl(W\times_S V,X\times_SY)= T(\cl(W,X)\otimes_S \cl(V,Y)).
\end{equation}
\footnotetext{This equation hold by definition if $W\times_S V$ is integral, and it is clear
	from the definition of the class element in the non-integral case that the equation
	extends to that case to.}

We first consider the case when $W\times_S V = \emptyset$. Then both 
$\cl(W\times_S V,X\times_SY)$ and
$T(\cl(W,X)\otimes_S \cl(V,Y))$ lie in 
$F(X\times_S Y,\emptyset) = 0$ so they trivially
agree.

Now assume that $W\times_S V$ is not empty.
We can find some open $U_X \subset X$ and 
$U_Y \subset Y$
such that 
\begin{align*}
	\emptyset &\neq W \cap U_X \,\, \text{is regular, and 
	} \\
	\emptyset &\neq V \cap U_Y \,\, \text{is regular}
\end{align*}
(because $S$ is excellent and thus in particular 
J-2).Then $U_X\times_S U_Y \subset X \times_S Y$ is open 
and 
the non-empty $(W\times_S V) \cap (U_X\times_S U_Y)$
is regular, if $W\times_S V \neq \emptyset$.
Denote by 
\begin{align*}
	j_X:(U_X,U_X\cap W) &\to (X,W),\,\,\,\text{and} \\
	j_Y: (U_Y,U_Y \cap V) &\to (Y,V)
\end{align*}
the maps in $\vu$ 
induced by the
open immersions $U_X \hookrightarrow X$ and $U_Y \hookrightarrow Y$
respectively. Then 
$$
j_{X\times_S Y}:= j_X\times_S j_Y: (U_X\times_S U_Y,(W\times_S V)\cap (U_X\times_S U_Y))
\to (X\times_S Y, W\times_S V)
$$
is the map in $\vu$ induced
by the open immersion $U_X\times_S U_Y \hookrightarrow X \times_S Y$.
Again we use the naturality of $T$ to see

\begin{align*}
	\fu(j_{X\times_S Y})(T(\cl(W,X)&\otimes \cl(V,Y)) \\
	&= T(\fu(j_X)(\cl(W,X))\otimes \fu(j_Y)(\cl(V,Y))).
\end{align*}
By condition (\ref{axiomopenimmersion}) we have that
\begin{align*}
	\fu(j_X)(\cl(W,X)) &= \cl(U_X\cap W,U_X), 
	\,\,\, \text{and} \\
	\fu(j_Y)(\cl(V,Y)) &= \cl(U_Y\cap V,U_Y),
\end{align*}
so
$$
\fu(j_{X\times_S Y})(T(\cl(W,X)\otimes \cl(V,Y))
= T(\cl(U_X\cap W,U_X) \otimes \cl(U_Y\cap V,U_Y)).
$$
Condition (\ref{axiomproduct}) tells us that
$T(\cl(U_X\cap W,U_X) \otimes \cl(U_Y\cap V,U_Y))
= \cl((U_X\cap W)\times_S(U_Y\cap V,U_X\times_S U_Y))$
and condition (\ref{axiomopenimmersion}) says that
$\cl((U_X\cap W)\times_S(U_Y\cap V),U_X\cap U_Y)
= 
\fu(j_{X\times_S Y})(\cl(W\times_S V,X\times_S Y),$
so we have
$$
\fu(j_{X\times_S Y})(\cl(W\times_S V,X\times_S Y)
=\fu(j_{X\times_S Y})(T(\cl(W,X)\otimes_S \cl(V,Y)))
$$
and \eqref{propstep4whattoshow} follows by semi-purity
if $W\times_S V$ is of pure $S$-dimension, but this follows from
Proposition \ref{SuslinVoevodsky}.
\bigskip

We have constructed a natural transformation of right-lax
symmetric monoidal functors
$$
\phi: (\CHl,\times_S,\one) \to (\fl,T,e),
$$	
that we wish to extend to a morphism $\phi\in \Hom_{\bfT}(\CH,\f)$,
i.e. we want to show that the following diagram commutes for 
all $f:(X,\Phi) \to (Y,\Psi)$ in $\vu$:
\begin{equation} \label{123step1diagram}
	\xymatrix{
		\CH(Y/S,\Psi) \ar[d]_-{\phi_{(Y,\Psi)}} \ar[r]^-{\CHu(f)} & \CH(X/S,\Phi) 
		\ar[d]^-{\phi_{(X,\Phi)}} \\
		\f(Y,\Psi) \ar[r]^-{\fu(f)} & \f(X,\Phi).
	}
\end{equation}
The proof proceeds in 5 steps.  In Step 5
we show that diagram \eqref{123step1diagram} commutes when $f$ is 
a smooth morphism. Steps 6 and 
7 are technical steps that 
we use in Step 8 in which we 
prove that diagram \eqref{123step1diagram}
commutes when $f$ is a closed immersion.
In Step 9 we deduce the
general case from steps 5 and 
8, since a general 
morphism $f$ will be an l.c.i. morphism.

\subsection*{\textbf{\textit{Step 5:}}} \label{step5}
In this step we assume we are given a 
\textit{smooth}  $f: (X,\Phi) \to 
(Y,\Psi)$ in 
$\vu$. We want to show that diagram 
\eqref{123step1diagram} commutes for this $f$.

\vspace{1em}

The additivity of all maps tells us that it 
is enough to show that \eqref{123step1diagram} is 
commutative when evaluated
at $[W]$ for $W \in \Psi$ irreducible.
Secondly, to show \eqref{123step1diagram} is 
commutative when evaluted at $[W]$, 
for $W \in \Psi$ irreducible, it is enough to 
show that the following 
diagram is commutative when evaluated at 
$[W]$\footnotemark.
$$ \label{124step1simplerdiagram}	
\xymatrix{
	\CH(Y/S,W) \ar[rr]^-{\CHu(f_1)} 
	\ar[d]_-{\phi_{(Y,W)}} && 
	\CH(X/S,f_1^{-1}(W)) 
	\ar[d]^-{\phi_{(X,f_1^{-1}(W))}} \\
	\f(Y,W) \ar[rr]^-{\fu(f_1)} && 
	\f(X,f_1^{-1}(W)).
}
$$ 
\footnotetext{
	The map $f_1: (X,f^{-1}(W)) \to (Y,W)$ is induced by the same smooth map as
	$f:(X,\Phi) \to (Y,\Psi)$ is. We also denote this underlying map by $f$.} 
This is easy diagram chasing.

Since $X$ and $Y$ are $\NS$-schemes
and $S$ is regular, they are themselves regular and therefore the disjoint unions 
of irreducible $\NS$-schemes.
If $X = \amalg X_i$ and $Y=\amalg Y_j$ then
for any $i$ there is some $j$ such that 
$f(X_i)\subset Y_j$. We can thus 
reduce to the case where $X$ and $Y$ are 
irreducible $\NS$-schemes. 

We are reduced to showing that 
for irreducible $\NS$-schemes $X$ and $Y$, 
a smooth morphism $f:X\to Y$ and an integral
closed subscheme $W \subseteq Y$ we have
$$ 
\fu(f)(\phi_{(Y,W)})[W]) = \phi_{(X,f^{-1}(W))}(\CHu(f)([W])).
$$
Since $\phi_{(Y,W)}([W]) = \cl(W,Y),$
and	 
\begin{align}
	\phi_{(X,f^{-1}(W))}(\CHu(f)([W])) &= \phi_{(X,f^{-1}(W))}([f^{-1}(W)]) \\ \notag
	&= \phi_{(X,f^{-1}(W))}(\sum_i n_i[V_i]) \\ \notag
	&= \sum_i n_i \phi_{(X,f^{-1}(W))}([V_i])\\\notag 
	&= \sum_i n_i \cl(V_i,X) \\ \notag
	& = \cl(f^{-1}(W),X),		\notag
\end{align}
where $\sum_i n_i [V_i]$ is the fundamental class
of $f^{-1}(W)$,
we are reduced to showing that

\begin{equation} \label{thmstep1whatIwanttoshow}
	\fu(f)(\cl(W,Y)) = \cl(f^{-1}(W),X).	
\end{equation}

By Lemma \ref{genericpointslemma} we can find an open $U \subset Y$ such that
$U$ contains all the generic points of $W$, $U\cap W$ is regular, $f^{-1}(U)$ 
contains all generic points of $f^{-1}(W)$, and $f^{-1}(U) \cap f^{-1}(W)$ is regular. Take such a $U$, and  denote by 
\begin{itemize}
	\item $\jmath_1: (U,U\cap W) \to (Y,W)$, 
	\, and
	\item $\jmath_2: (f^{-1}(U),f^{-1}(U\cap W) \to (X,f^{-1}(W))$
\end{itemize}
the maps in $\vu$ induced by the open immersions $U \hookrightarrow Y$ and 
$f^{-1}(U)\hookrightarrow X$ respectively.
By condition (\ref{axiomsmoothmorphism}) we have 
\begin{equation} \label{thmstep1pullbackfromaxiom}
	\fu(f)(\cl(U\cap W,U)) = \cl(f^{-1}(U\cap W,f^{-1}(U)))
\end{equation}
and we want to deduce \eqref{thmstep1whatIwanttoshow} from this.

By condition (\ref{axiomopenimmersion}) we have 
$\fu(j_1)(\cl(W,Y)) = \cl(U\cap W,U),$
and if, as before, the irreducible 
compontents of $f^{-1}(W)$ are $V_i$ then the 
irreducible 
components of $f^{-1}(U\cap W)$ are 
$f^{-1}(U)\cap V_i$ (since $f^{-1}(U)$ 
contains all the generic points of 
$f^{-1}(W)$). Therefore by condition
(\ref{axiomopenimmersion}) we have
\begin{align*}
	\fu(j_2)(\cl(f^{-1}(W),X)) &= 
	\fu(j_2)(\sum_i n_i \cl(V_i,X)) \\ \notag
	&= \sum_i n_i \fu(j_2)(\cl(V_i,X)) \\ \notag 
	&= \sum_i n_i \cl(f^{-1}(U)\cap V_i,f^{-1}(U)) \\ \notag
	& = \cl(f^{-1}(U\cap W),f^{-1}(U)). \notag
\end{align*}
Substitute this into \eqref{thmstep1pullbackfromaxiom} to obtain
\begin{equation} \label{thmstep1middlestep}
	\fu(f)(\fu(j_1)(\cl(W,Y))) = \fu(j_2)(\cl(f^{-1}(W),X))
\end{equation}
By applying the contravariant functor $\fu$ to the following diagram 
$$
\xymatrix{
	f^{-1}(U) \ar[r]^-{j_2} \ar[d]_-{f} &
	X \ar[d]^-{f} \\
	U \ar[r]^-{j_1} &
	Y,
}
$$

and substituting  into \eqref{thmstep1middlestep} we obtain
\begin{equation} \label{thmstep1panulteq}
	\fu(j_2)(\fu(f)(\cl(W,Y))) = \fu(j_2)(\cl(f^{-1}(W),X)).
\end{equation}
By construction $f^{-1}(U)$ contains all the generic points of
$f^{-1}(W)$. By assumption, $X$ and $Y$ are irreducible
and so our smooth morphism $f:X\to Y$ is 
smooth of relative $S$-dimension $r=\dim_S(X)
-\dim_S(Y)$. This is stable under 
base change so $f^{-1}(W) \to W$ is 
smooth of relative $S$-dimension $r$. In
particular, it has pure $S$-dimension so that
$\fu(f)(\cl(W,Y))$ and 
$\cl(f^{-1}(W),X)$
both lie in $\f_{2\dim_S(f^{-1}(W))}(X,f^{-1}(W))$. 
Therefore by 
semi-purity, equation 
\eqref{thmstep1panulteq} implies
$$
\fu(f)(\cl(W,Y)) = \cl(f^{-1}(W),X).
$$

\subsection*{\textbf{\textit{Step 6:}}} \label{step6}

Consider a vector bundle
$$
p:E \to X
$$
and let $s:X \to E$ be the zero-section. We want to prove that the following diagram commutes
for any closed subscheme $W \hookrightarrow X$:

\begin{equation}\label{step2whatwewanttoshow}
	\xymatrix{ 
		\CH(E/S,p^{-1}(W)) \ar[r]^-{\CHu(s)} \ar[d]_-{\phi} &
		\CH(X/S,W) \ar[d]^-{\phi} \\
		\f(E,p^{-1}(W)) \ar[r]^-{\fu(s)} &
		\f(X,W).
	}
\end{equation} 
This diagram makes 
sense since that $E$ is smooth over $S$.

Proposition \cite[Theorem 3.3.]{Fultonbook} says that
if $p:E \to X$ and $s:X \to E$ are as above, then the flat pullback
$$
\CHu(p)=:p^{*}:\CH_k(X/S) \to \CH_{k+n}(E/S)
$$
is an isomorphism for all $k$ (where $n$ is the rank of the 
vector bundle  $p:E\to X$). Let $a \in 
\CH(E/S,p^{-1}(W))$, which we can write as $a = \CHu(p)(b)$ 
for some 
$b \in \CH(X/S,W)$ by the above isomorphism we get:
\begin{align} \label{thmstep2firstalign}
	\fu(s)\circ \phi_{(E,p^{-1}(W))}(a) & = \fu(s)\circ \phi_{(E,p^{-1}(W))} \circ \CHu(p)(b) \\  \notag
	& = \fu(s)\circ \fu(p) \circ \phi_{(X,W)}(b), \,\,\, 
	\text{by Step 1} 
	\\ \notag
	& = \fu(p\circ s)(\phi_{(X,W)}(b)) \\ \notag
	& = \phi_{(X,W)}(b). 
\end{align}
On the other hand
\begin{align} \label{thmstep2secondalign}
	\phi_{(X,W)}\circ \CHu(s)(a) &= \phi_{(X,W)}\circ \CHu(s) \circ \CHu(p)(b) \\ \notag
	& = \phi_{(X,W)} \circ \CHu(p\circ s)(b) \\ \notag
	& = \phi_{(X,W)}(b).
\end{align}
Combining equations \eqref{thmstep2firstalign} and 
\eqref{thmstep2secondalign}, gives the commutativity of\eqref{step2whatwewanttoshow}.

\subsection*{\textbf{\textit{Step 7:}}} \label{step7}
Now let $W\subset X$ be a closed subscheme  and consider the 
morphisms in both
$\vl$ and $\vu$
\begin{align*}
	& i_0: (X,W) \to (X\times_S\P^1_S,W\times_S \P^1_S), \,\,\,\text{and} \\
	& i_{\infty}:(X,W) \to (X\times_S\P^1_S,W\times_S \P^1_S)  
\end{align*}
induced by the inclusions $(X\times_S \{0\}) \subset X\times_S \P^1_S$ and 
$(X\times_S\{\infty\}) \subset X\times_S\P^1_S$ respectively.
We want to show that
$$
\fu(i_0) = \fu(i_{\infty}).
$$
First of all we notice that 
$$
pr_1 \circ i_{\epsilon} = id: (X,W) \to (X,W)
$$
where $i_{\epsilon}$ denotes either $i_0$ or $i_{\infty}$ and
$pr_1:(X\times_S\P^1_S,W\times_S\P^1_S) \to (X,W)$ is the morphism in $\vl$
induced by the first projection $X\times_S \P^1_S \to X$. Furthermore, in $\CH$ we 
have $1_X = [X]$ and since $\phi(1_X)= 1_X$ we have $\phi([X]) = 1_X$. We can 
therefore write for any $a \in \f(X\times_S\P^1_S,W\times_S\P^1_S)$
\begin{equation}\label{thmstep3firstequation}
	\fu(i_{\epsilon})(a) =\fl(pr_1)\fl(i_{\epsilon})(\phi([X])\cup \fu(i_{\epsilon})(a)).
\end{equation}
The first projection formula
tells us
that
\begin{equation} \label{thmstep3secondequation}
	\fl(i_{\epsilon})(\phi([X])\cup \fu(i_{\epsilon})(a)) = \fl(i_{\epsilon})(\phi([X]))\cup a.
\end{equation}
Combining \eqref{thmstep3firstequation} and 
\eqref{thmstep3secondequation} we obtain 
\begin{equation}
	\fu(i_{\epsilon})(a) = \fl(pr_1)(\fl(i_{\epsilon})(\phi([X]))\cup a).
\end{equation} 		
Now $\phi$ is a natural transformation $\CHl \to \fl$ so we have 
$$
\fl(i_{\epsilon})(\phi([X])) = \phi(\fl(i_{\epsilon})([X])) = \phi([X\times_S \{\epsilon\}]),
$$
but $[X\times_S\{0\}] \sim [X\times_S\{\infty\}]$ as cycles
which shows that 
$$
\fu(i_0) = \fu(i_{\infty}).
$$

\subsection*{\textbf{\textit{Step 8:}}}\label{step8}	

In this step we want to show that $\phi$ commutes with $\fu(f)$ when $f$ is 
a closed immersion. Namely, let $f:X\to Y$ be a closed immersion of smooth $S$-schemes
and $V\subset Y$ be a closed subscheme. Denote the preimage of $V$ by 
$W:=f^{-1}(V) := V\times_Y X$. The immersion $f$ induces a morphism 
$f:(X,W) \to (Y,V)$ in $\vu$ and we want to show that
\begin{equation} \label{123step4whattoshowsquare}
	\xymatrix{
		\CH(Y/S,V) \ar[r]^-{\CHu(f)} \ar[d]^-{\phi_{(Y,V)}} &
		\CH(X/S,W) \ar[d]^-{\phi_{(X,W)}} \\
		\f(Y,V) \ar[r]^-{\fu(f)} &
		\f(X,W)
	}
\end{equation}
commutes.
Because of additivity, we are reduced to
showing
\begin{equation} \label{whattoshow}
	\fu(f)(\phi([V])) = \phi(\CHu(f)([V]))
\end{equation}
for integral $V$. 
Consider the deformation to the normal cone, i.e. consider the following schemes:
\begin{align*}
	M^{0} &:= Bl_{X\times_S \{\infty_S\}}(Y\times_S \P^1_S)\setminus Bl_{X\times_S \{\infty_S\}}(Y\times_S \{\infty_S\}), \,\,\, \text{and} \\
	\tilde{M}^0 &:= Bl_{W\times_S\{\infty_S\}}(V\times_S \P^1_S)\setminus Bl_{W\times_S \{\infty_S\}}(V\times_S \{\infty_S\}).
\end{align*}		
$\tilde{M}^0$ is closed in $M^0$.
We have 
a dominant morphism
$\rho^0: M^0\to \P_S^1,$
such that 
\begin{align*}
	(\rho^0)^{-1}(\A_S^1 &:= \P_S^1\setminus\{\infty_S\}) = 
	Y\times_S\A_S^1, \,\,\, \text{and} \\
	(\rho^0)^{-1}(\infty_S) &= C_X Y := C_{X\times_S\{\infty_S\}} 
	Y\times_S \{\infty_S\},
\end{align*}
and closed immersions
\begin{align*}
	i_X:X\times_S \P_S^1 &\to M^0 \,\,\, \text{and} \\
	i_W:W\times_S \P_S^1 & \to \tilde{M}^0,
\end{align*}
that deform the immersions $X\to Y$ and $W\to V$ respectively 
over $\A_S^1$ to the zero section of the respective normal cones 
$C_X Y$ and $C_W V$. We have $ W \times_S \P_S^1 = \tilde{M}^0 
\cap (X\times_S \P_S^1)$
as closed subschemes of $M^0$, and therefore we obtain a 
morphsm 
in $\vu$ induced by $i_X$
$$
g: (X\times_S \P_S^1, W\times_S \P_S^1) \to (M^0,\tilde{M}^0).
$$
We define for $\epsilon \in \{0_S,\infty_S\}$ 
morphisms
$$
\ie: (X\times_S\{\epsilon\}, W\times_S\{\epsilon\}) \to (M^0,\tilde{M}^0)
$$ 
in $\vu$ by the composition
$$
(X\times_S\{\epsilon\}, W\times_S\{\epsilon\}) \xrightarrow{j_{\epsilon}} 
(X\times_S \P_S^1, W\times_S \P_S^1) \xrightarrow{g} (M^0,\tilde{M}^0),
$$
where $j_{\epsilon}$ is induced by the inclusions $X\times_S \{\epsilon\} \to X\times_S
\P_S^1$. In Step 7, we
showed that  $\fu(j_0) = \fu(j_{\infty})$ so we have 
\begin{equation} \label{123step4eqninulliinfty}
	\fu(i_0) =\fu(i_{\infty}).
\end{equation}
Consider the open immersion $Y\times_S \A_S^1 \to M^0$
and let 
$j: (Y\times_S \A_S^1, V\times_S \A_S^1)
\to (M^0,\tilde{M}^0)$ 
be the induced morphism in $\vu$. Consider also the morphism
$p:(Y\times_S \A_S^1, V\times_S \A_S^1) \to
(Y,V)$
in $\vu$ induced by the first projection $Y\times_S \A_S^1 \to Y$,
We have 
$$
\CHu(p)([V]) = [V\times_S \A_S^1] = \CHu(j)([\tilde{M}^0]).
$$
Combining this with Step 5 (the commutativity of \eqref{123step1diagram} for smooth $f$) gives us
\begin{equation*}\label{123step4eqnpandj}
	\fu(p)(\phi([V])) = \phi(\CHu(p)([V])) = \phi(\CHu(j)([\tilde{M}^0])).
\end{equation*}
The morphism $i_0$ has a factorization in 
$\vu$ 
$$
\xymatrix{
	(X\times_S\{0_S\},W\times_S\{0_S\}) \ar[r]^-{i_0} \ar[d]_-{f} &
	(M^0, \tilde{M}^0) \\
	(Y\times_S \{0_S\},V\times_S\{0\}) \ar[r]_-{\alpha} &
	(Y\times_S\A_S^1,V\times_S\A_S^1)\ar[u]_-{\beta}
}
$$
where $\beta$ is an open immersion and 
$\alpha$ is the closed immersion of the 
effective Cartier divisor $Y\times_S \{0_S\}$
into $Y\times_S \A_S^1$. 
By Step 1 we have 
$$
\fu(\beta)(\cl(\tilde{M}^0,M^0)) = 
\cl(V\times_S\A_S^1,Y\times_S\A_S^1),
$$
and since $Y\times_S \{0_S\}$ is a smooth
Cartier divisor in $Y\times_S \A_S^1$ 
intersecting $V\times_S\A^1_S$ properly we
have by Step 2 that
$$
\fu(\alpha)(\cl(V\times_S\A_S^1,Y\times_S\A_S^1)) = \cl(V\times_S \{0_S\},Y\times_S\{0_S\}).
$$
Furthermore, $\phi([V]) = \cl(V,Y)$ and $\phi([\tilde{M}^0]) = \cl(\tilde{M}^0,M^0)$,
and therefore we have 
$$
\fu(f)(\phi([V])) = \fu(i_0)(\phi([\tilde{M}^0])),
$$
and \eqref{123step4eqninulliinfty} gives us
\begin{equation} \label{fandiinf}
	\fu(f)(\phi([V])) = \fu(i_{\infty})(\phi([\tilde{M}^0])).
\end{equation}
The normal bundle $N_X Y$ is a smooth effective 
Cartier divisor in $M^0$ and
$N_X Y$ intersects $\tilde{M}^0$
properly since
$$
N_X Y \cap \tilde{M}^0 = C_W V.
$$
The morphism $i_{\infty}$ has a factorization in $\vu$
\begin{equation} \label{factorizationiinf}
	\xymatrix{
		(X,W) \ar@/_3pc/[rr]^-{i_{\infty}}\ar[r]^-{s} & (N_X Y, C_W V) \ar[r]^-{t} & (M^0, \tilde{M}^0)
	}
\end{equation}
where $s$ is induced by the zero-section $X\to N_X Y$ of the normal bundle of $X$ in
$Y$, and $t$ is induced by the closed immersion $N_X Y = C_X Y \to M^0$. Step 2 
tells us that
\begin{equation*}
	\fu(t)(\phi_{(M^0, \tilde{M}^0)}([\tilde{M}^0]) = \phi_{(N_X Y, C_W V)}([C_W V]).
\end{equation*} 
Consider the fiber diagram
$$
\xymatrix{
	N_X Y \times_X W \ar[d] \ar[r] &
	W \ar[d] \ar[r]^{f|_W} &
	V\ar[d]\\
	N_X Y \ar[r] &
	X \ar[r]^f &
	Y.
}
$$
Since $W \to X$ is a closed immersion, $N_X Y\times_X W \to N_X Y$ is a closed 
immersion as well and the zero section $X \to N_X Y$ also induces a morphism
$$
s':(X,W) \to (N_X Y, N_X Y \times_X W)
$$
in $\vu$. The identity morphism $N_X Y \to N_X Y$ induces a morphism 
$$
\tau: (N_X Y, C_W V) \to (N_X Y, N_X Y \times_X W)
$$
in $\vl$ and we have a Cartesian diagram
$$
\xymatrix{
	(X,W) \ar[r]^-{id} \ar[d]_-{s} &
	(X,W) \ar[d]^-{s'} \\
	(N_X Y, C_W V) \ar[r]^-{\tau} &
	(N_X Y, N_X Y \times_X W).
}
$$
The morphism $s'$ is induced by the closed immersion $X\to N_X Y$ (this is a closed immersion
since $N_X Y\to X$ is affine and hence separated) and 
$\tau$ is clearly transversal to $s'$ so from the definition
of WCTS we have 
\begin{equation} \label{pushpullzerosection}
	\fu(s) = \fu(s')\circ \fl(\tau).
\end{equation}
We then have
\begin{align*}
	\fu(f)(\phi([V]))&=\fu(i_{\infty})(\phi([\tilde{M}^0])) \,\,\,\text{by}\,\,\, 
	\eqref{fandiinf} \\
	&= \fu(s)\circ \fu(t)(\phi([\tilde{M}^0])) \,\,\,\text{by}\,\,\, \eqref{factorizationiinf} \\
	&= \fu(s')\circ \fl(\tau) \fu(t)(\phi([\tilde{M}^0])) \,\,\,\text{by}\,\,\,
	\eqref{pushpullzerosection} \\
	&= \fu(s')\circ \fl(\tau)(\phi(\CHu(t)([\tilde{M}^0]))) \,\,\,\text{by Step 2}\,\,\,
	\\
	&= \fu(s')(\phi(\CHl(\tau)\circ\CHu(t)([\tilde{M}^0]))) \,\,\,\phi\,\,\text{commutes with pushforwards}\\
	&= \phi(\CHu(s')\circ\CHl(\tau)\circ\CHu(t)([\tilde{M}^0]))) 
	\,\,\,\text{by Step 6}\,\,\,  \\
	&= \phi(\CHu(i_{\infty})(\tilde{M}^0)) \,\,\, \text{by} \,\,\, \eqref{pushpullzerosection} \\
	&= \phi(\CHu(i_0)(\tilde{M}^0))\\
	&= \phi(\CHu(f)([V])).
\end{align*}

\subsection*{\textbf{\textit{Step 9:}}} \label{step9}
To finish the proof we let $f:(X,\Phi) \to (Y,\Psi)$ 
be any morphism in $\vu$. Any morphism between $\NS$-schemes
is an l.c.i. morphism so we can factor 
$f$ as 
$$
(X,\Phi) \xrightarrow{i} (Z,\Omega) \xrightarrow{g} (Y,\Psi)
$$
for some $S$-scheme $Z$ and a some family of supports
$\Omega$ on $Z$. Here $g: (Z,\Omega) \to (Y,\Psi)$ is 
induced by a smooth morphism and $i: (X,\Phi) \to 
(Z,\Omega)$ is induced by a regular closed immersion.

We want to show that
$$
\phi \circ \CHu(f) = \fu(f)\circ \phi
$$
It is enough to show that this holds for any 
$[V]$ where $V\in \Psi$ is irreducible. But then
\begin{align*}
	\phi\circ\CHu(f) & = \phi \circ \CHu(i)\circ \CHu(g) \\
	&= \fu(i)\circ \phi \circ \CH(g) \,\,\, \text{by Step 4}  \\
	& = \fu(i)\circ \fu(g) \circ \phi \,\,\, \text{by Step 1} \\
	&= \fu(f)\circ \phi.
\end{align*}
\end{proof}

\begin{lem}\label{equaldimlemma}
	If $X$ is an $S$-scheme locally of finite 
	type and  $W \subset X$ is an irreducible 
	closed subset, $Y$ is a locally 
	Notherian, locally of finite type 
	$S$-scheme $f:X\to Y$ is a morphism 
	of $S$-schemes such that the restriction
	$f|W:W\to Y$ 
	is proper and $\dim_S(W) = \dim_S(f(W))$, then
	there exists an open $U \subset Y$ such 
	that 
	\begin{itemize}
		\item $U \cap f(W) \neq \emptyset$,
		\item $U\cap f(W)$ is regular,
		\item $f^{-1}(U)\cap W$ is regular, and
		\item The map induced from $f$ by 
		restriction
		$$
		f': f^{-1}(U) \cap W \to U\cap f(W)
		$$
		is finite. 
	\end{itemize}
\end{lem}
\begin{proof}
	By part\cite[Proposition 2.1.3.(iii)]{AndreasWeberThesis} we have 
	$$
	\dim_S(W)  = \dim_S(f(W)) + \trdeg(R(W)/R(f(W)))
	$$
	We assume that $\dim_S(W) = \dim_S(f(W))$ 
	so we have
	$$
	\trdeg(R(W)/R(f(W))) = 0
	$$
	It is known, see for example \cite[Tag: 02NX]{Stacks}, that if $f:X\to Y$ is a dominant morphism, locally of 
	finite type between integral schemes, then 
	the following are equivalent
	\begin{enumerate}
		\item The extension $R(Y) \subseteq R(X)$
		has transcendence degree 0,
		\item There exists a nonempty affine open
		$V \subseteq Y$ such that 
		$$
		f^{-1}(V) \to V
		$$
		is finite.
	\end{enumerate}
	This allows us 
	(since $f|_W$ is proper
	and hence separated)
	to obtain 
	a nonempty affine open subset $U_1 \subset f(W)$ such
	that the restriction 
	$$
	f: f|_W^{-1}(U_1) = f^{-1}(U_1)\cap W \to U_1
	$$
	is finite.
	Now consider the singular locus
	in $W$ (i.e. the locus of points in 
	$W$ that are not regular). Since $S$
	is excellent it is in particular J-2. 
	Any scheme
	that is locally of finite type over
	$S$ is J-2, so given our assumptions,
	$Y$ is J-2 so $W_{reg}$ is open
	and hence the singular locus is closed.
	The restriction $f|_W$ is proper, so $f(W_{sing})$ is closed
	in $f(W)$.
	Let $\sO := f(W)\setminus f(W_{sing})$
	and define $\tilde{U} := U_1\cap \sO \cap f(W)_{reg}$.
	This is a nonempty open subset of $f(W)$ and there exists 
	an open $U \subseteq Y$ such that $U \cap f(W) = \tilde{U}$.
	Now we have already seen that $U\cap f(W) = \tilde{U}$ is nonempty,
	and it is an open subscheme 
	of $f(W)_{reg}$ so it is regular.
	Consider $f^{-1}(U) \subseteq X$.
	We have
	$$
	f^{-1}(U)\cap W \subseteq f^{-1}(U\cap f(W)) 
	= f^{-1}(\tilde{U}) \subseteq 
	f^{-1}(\sO).
	$$
	Since $\sO= f(W)\setminus f(W_{sing}),$
	we see that $f^{-1}(U)\cap W \subset W_{reg}.$
	Note that $f^{-1}(U\cap f(W))\cap W = f^{-1}(U)\cap W,$
	so we see that the map 
	$$
	f': f^{-1}(U)\cap W \to U\cap f(W)
	$$
	is finite as it is obtained from 
	the finite map $f^{-1}(U_1) \cap W \to U_1$
	by base change along $U\cap f(W) \subset U_1$.
\end{proof}

\begin{lem} \label{classelementopenlemma}
	Let $X$ be an $\NS$-scheme and $W \subseteq X$ be
	an integral closed subscheme. Let $U\subseteq X$
	be an open subscheme such that $U\cap W \neq
	\emptyset$. Let $\f$ be a WCTS satisfying semi-purity and condition (2i) from Theorem \ref{themaintheorem}. Then
	$$
	\fu(j)(\cl(W,X)) = \cl(U\cap W,U).
	$$
\end{lem}
\begin{proof}
	We know since $W$ is an integral scheme over
	an excellent base scheme $S$ that it is generically regular. The same 
	is true for the open subset $ U\cap W \subset 
	W$.
	We can thus find and open subset $V \subset U$ 
	such that $V\cap(U\cap W) = V \cap W$
	is non-empty and regular. Consider 
	the map induced by inclusion
	$j_V:(V,V\cap W) \to (U,U\cap W)$.
	Notice that since $U \cap W$ is irreducible and 
	$V\cap W$
	is a non-empty subset of $U \cap W$ the generic 
	point of 
	$U \cap W$  is contained in $V\cap W$. We also have that
	$\fu(j)(\cl(W,X)),$ and  $\cl(U\cap W,U)$ are in   
	$\f_{2\dim_S (U \cap W)}(U,U\cap W)$,
	so in order to prove $\fu(j)(\cl(W,X)) = \cl(U\cap W,U),$
	it suffices by semi-purity to prove 
	$$
	\fu(j_V)(\fu(j)(\cl(W,X)))
	=\fu(j_V)(\cl(U\cap W,U)).
	$$
	Since $V \cap W$ is regular, condition 
	(\ref{axiomopenimmersion}) gives us that
	$\fu(j_V)(\cl(U\cap W,U)) = \cl(V\cap W,V),$
	and since $\fu(j_V)\circ \fu(j) = \fu(j\circ j_V)$
	where $j\circ j_V:(V,V\cap W) \to (X,W)$
	is the morphism induced by the open immersion 
	$V \subset X$, we have again by condition 
	(\ref{axiomopenimmersion})
	\begin{align*}
		\fu(j_V)(\fu(j)(\cl(W,X))) &= 
		\fu(j\circ j_V)(\cl(W,X)) \\
		& = \cl(V\cap W,V).
	\end{align*}
\end{proof}

\begin{lem}\label{genericpointslemma}
	Consider a smooth morphism $f:X\to Y$ between $\NS$-schemes. Let $W\subset Y$ be an 
	irreducible closed subscheme such that
	$f^{-1}(W) \neq \emptyset$.
	Then there exists an open subset $U \subset Y$ such that
	\begin{itemize}
		\item $U$ contains the generic point of $W$,
		\item $U\cap W$ is regular,
		\item $f^{-1}(U)$ contains all the generic points of $f^{-1}(W)$, and
		\item $f^{-1}(U) \cap f^{-1}(W)$ is regular.
	\end{itemize}
\end{lem}	
\begin{proof}
	$W$ is generically regular so there exists an open $U \subset Y$ such that
	$$
	U \cap W = W_{reg},
	$$
	which is open (and hence dense) in $W$. This $U$ satisfies the first 
	two conditions, namely that $U\cap W$ is regular and $U$ contains the 
	generic point of $W$ (any open subset of $W$ does). Furthermore, $f$ is 
	smooth, in particular flat, so any 
	irreducible component of $f^{-1}(W)$
	dominates $W$,
	% (see \cite[Proposition 2.3.4.(iii)]{EGA4})
	i.e. $f$ sends any generic 
	point of $f^{-1}(W)$ to the generic point of $W$. Therefore $f^{-1}(U)$ contains
	all the generic points of $f^{-1}(W)$.
	
	Finally we have that $f:f^{-1}(U\cap W)
	\to U\cap W$ is the base change of
	the smooth morphism $f:X \to Y$ along 
	$U\cap W$
	so $f^{-1}(U\cap W) \to U\cap W$ is smooth. 
	Furthermore $U\cap W$ is regular and locally Noetherian so $f^{-1}(U\cap W)$
	is regular.
	% (see \cite[Theorem 3.36]{Liu}). 
\end{proof}

\section{Cycle Class to Hodge Cohomology} \, \\

We recall some notation 
\begin{nota} \label{notation2}
	\,
	\begin{itemize}
		\item $\eta_i : \sE xt^n_Y(i_{\ast}\sO_X,\sF) \to \omega_{X/Y}\otimes i^{\ast}(\sF)$ is the 
		Fundamental Local Isomorphism, for an l.c.i. morphism $i:X\to Y$ of pure codimension
		$n$. \,\,\,\,\, \text{(See \cite[\S 2.5.]{Conrad})}
		\item $\zeta_{f,g}': \omega_{X/Z}\to \omega_{X/Y}\otimes f^{\ast}\omega_{Y/Z}$ are 
		isomorphisms for morphisms $f:X\to Y$ and $g:Y\to Z$ such that each og $f$, $g$,
		and $g\circ f$ is either separated smooth, or an l.c.i. morphism.\footnote{The precise 
			definition depends on the cases, i.e. whether all are smooth, all are l.c.i. morphisms, etc. 
			The definition in \cite{Conrad} lists the different cases and gives a precise definition in 
			each case.} \,\,\,\,\, \text{(See \cite[\S 2.2.]{Conrad})}
		\item $d_f: f^{\flat} \xrightarrow{\cong} f^!$ is and isomorphism for any finite map $f$.
		\,\,\,\,\, \text{(See \cite[(3.3.19.)]{Conrad})}
		\item $\psi_{g,f}: (f\circ g)^{\sharp} \to
		g^{\flat}\circ f^{\sharp}$ is an isomorphism defined for $f:Y\to Z$ a separated smooth
		morphism, $g: X\to Y$ is a finite morphism, and $f\circ g$ is a smooth separated 
		morphism. \,\,\,\,\, \text{(See \cite[(2.7.5.)]{Conrad})}

	\end{itemize}
\end{nota}
We begin by defining a cycle class for regular, irreducible closed subschemes. Let
$X$ be an $\NS$-scheme and let $i:Z \hookrightarrow X$ be a closed immersion 
of a regular, irreducible closed subscheme $Z$ to $X$ and denote by 
$c := \codim(Z,X)$. Then $i$ is a regular closed immersion of codimension $c$. Let 
$\sI$ be the ideal sheaf of $i$. We have a well defined map
\begin{align*}
	\sI/\sI^2 &\to i^{\ast}(\Omega_{X/S}^1) = \frac{\Omega_{X/S}^1}{\sI} \\
	\bar{a} &\mapsto da,
\end{align*}
and by taking the wedge product $\bigwedge^c :=\bigwedge^c_{\sO_Z}$ and tensoring 
with the inverse $\omega_{Z/X}$ we get a map

\begin{equation*} \label{cycleclassssecondequation}
	\sO_Z \cong	\bigwedge^c\sI/\sI^2 \otimes_{\sO_Z} \omega_{Z/X}  \xrightarrow{\phi \otimes id} 
	i^{\ast}\Omega_{X/S}^c\otimes_{\sO_Z} \omega_{Z/X}.
\end{equation*} 
Since $i$ is a regular closed immersion (so in particular an l.c.i. morphism) we know that
$\omega_{Z/X} \cong i^{!}\sO_X[c]$ and we furthermore have
$$
i^{\ast}\Omega_{X/S}^c\otimes_{\sO_Z} i^{!}\sO_X[c] \cong i^{!}(\Omega_{X/S}^c)[c],
$$
see for example \cite[\S 2.5]{Conrad}, and we therefore have a morphism
\begin{equation}\label{cycleclassthirdequation}
	\sO_Z \to i^{!}(\Omega_{X/S}^c)[c].
\end{equation}
By adjunction of $Ri_{\ast}$ and $i^!$, \eqref{cycleclassthirdequation} gives a map
\begin{equation*}
	i_{\ast}\sO_Z \to \Omega_{X/S}^c[c].
\end{equation*}
Applying $\rugz$ to this and taking the zeroth cohomology gives us 
\begin{equation*} \label{cycleclassfinalmapregular}
	H^0(Z,\sO_Z) \xrightarrow{\gamma_Z} H^{c}_Z(X,\Omega_{X/S}^c),
\end{equation*}
and we define
$$
\cl(Z,X) := \gamma_Z(1).
$$
If the ideal sheaf $\sI$ of $i:Z\hookrightarrow X$ is globally generated by a 
regular sequence $s_1,\ldots,s_c$ then equivalently 
the class element $\cl(Z,X)$ is explicitly defined as the 
image of the map $1 \mapsto \bar{s}^{\vee}_1 \wedge \cdots \wedge \bar{s}^{\vee}_c 
\otimes i_X^{\ast}(
ds_c\wedge\cdots \wedge ds_1) \in \Hom(\sO_Z,\omega_{Z/X}\otimes_{\sO_Z} 
i_X^{\ast}\Omega_{X/S}^c )$ under the composition
\begin{align} \label{spreadingoutclasscomposition}
	\Hom(\sO_Z,\omega_{Z/X}\otimes_{\sO_Z} i_X^{\ast}\Omega_{X/S}^c ) &=
	\Gamma(Z,H^0(R\sH om(\sO_Z,\omega_{Z/X}\otimes_{\sO_Z} i_X^{\ast}\Omega_{X/S}^c 
	)))  \\ \notag
	&\xrightarrow{\eta_{i_X}^{-1}}
	\Gamma(Z,\sE xt^c(\sO_Z,i_X^!\Omega_{X/S}^c)) \\ \notag
	&\cong \Ext^c((i_X)_{\ast}\sO_Z,\Omega_{X/S}^c) \\ \notag
	&\to H_Z^c(X,\Omega_{X/S}^c), \notag
\end{align}
where the map $\eta_{i_X}$ is the Fundamental 
Local Isomorphism, see Notation \ref{notation2}.\footnote{Here we use  adjunction and 
	that $\sO_Z$ is a locally free $\sO_Z$-module to get the 
	isomorphism $\Gamma(X,\sE xt^c(\sO_Z,i_X^!\Omega_{X/S}^c)) \cong 
	\Ext^c((i_X)_{\ast}\sO_Z,\Omega_{X/S}^c)$.}

The following proposition tells us that we can define a cycle class on all irreducible 
closed subschemes $Z$ in $X$ by spreading out from the regular locus.

\begin{prop} \label{spreadingoutclassproposition}
	Let $X$ be an $\NS$-scheme and let $Z\subset X$ be an irreducible closed subset
	of codimension $c$. There is a class $\cl(Z,X) \in H^c_Z(X,\Omega_{X/S}^c)$ such that
	$$
	\hu(j)(\cl(Z,X)) = \cl(U\cap Z,U)
	$$
	for every open $U \subset X$ such that $U\cap Z$ is regular and non-empty and
	$j:(U,U\cap Z) \to(X,Z)$ is the map in $\vu$ induced by the open immersion
	$U \hookrightarrow X$.  This class is unique by semi-purity. 
\end{prop}
\begin{proof}  
	\begin{itemize}
		\item[\underline{Step 1}:] Let $\eta$ be the generic point of $Z$. Define 
		$$
		H_{\eta}^c(X,\Omega_{X/S}^c) := \varinjlim_{U \ni \eta} H^c_{U\cap Z}(U,\Omega^c_{C/S})
		$$
		where the limit runs over all open subschemes $U \subset X$ such that 
		$\eta \in U$. Choose $U$ such that $U\cap Z$ is regular, then the 
		image of $\cl(U\cap Z,U)$ in $H_{\eta}^c(X,\Omega_{X/S}^c)$ is 
		independent of the choice of $U$ by Proposition \ref{Hodgepushpullproposition}. 
		We denote this local class by $\cl(Z,X)_{\eta}$ or $\cl(Z)_{\eta}$.
		\item[\underline{Step 2}:] A class $\alpha \in H_{\eta}^c(X,\Omega_{X/S}^c)$ 
		extends to a global class, i.e. is in 
		the image of 
		$$
		H^c_Z(X,\Omega_{X/S}^c) \to H_{\eta}^c(X,\Omega_{X/S}^c),
		$$
		if and only if for any $1$-codimensional point $x \in Z$ there exists an open 
		subset $U \subset X$ containing $x$ so that $\alpha$ lies in the image 
		of 
		$$
		H^c_{Z\cap U}(U,\Omega_{U/S}^c) \to H_{\eta}^c(X,\Omega_{X/S}^c).
		$$ 
		This is proven with the Cousin resolution, exactly as in 
		\cite[Proposition 3.1.1., Step 2]{KayAndre}.
		\item[\underline{Step 3}:]  If $Z$ is normal, then $\cl(Z)_{\eta}$ extends 
		uniquely to a class in $H_Z^c(X,\Omega_{X/S}^c)$. 
		This is exactly like \cite[Proposition 3.1.1., Step 3]{KayAndre}, 
		except of course that we are looking an open $U\subset X$ such that 
		$U\cap Z$ is regular and $U\cap Z$ contains all points of codimension 1 of $Z$.
		\item[\underline{Step 4}:] 
		We may assume, by the preceding steps, that $X$ is affine. 
		We are working over an excellent base scheme $S$, so the normalization $\tilde{Z}
		\to Z$ is a finite, and hence a projective map. Therefore the
		normalization factors as 
		$$
		\tilde{Z} \to \P^n_Z\xrightarrow{pr} Z,
		$$
		for some $n$.
		Step 3 gives us a class
		$\cl(\tilde{Z},\P_X^n) \in H^{n+c}_{\tilde{Z}}( 
		\P_X^n,\Omega^{n+c}_{
			\P^n})$ and we consider 
			$$
			\hl(pr_1)(\cl(\tilde{Z},\P_X^n)) \in 
			H^c_Z(X,\Omega_{X/S}^c).
			$$ 
		To show that $\hl(pr_1)(\cl(\tilde{Z},\P_X^n))$ is the class we are looking for, we want to
		show that for any open $U \subset X$ such that $U\cap Z \neq \emptyset$ and 
		$U\cap Z$ is regular we have 
		$$
		\hu(j)\hl(pr_1)(\cl(\tilde{Z},\P_X^n)) = \cl(U\cap Z,U),
		$$
		where $j:(U,U\cap Z) \to (X,Z)$ is induced by the open immersion. Consider the 
		Cartesian square
		$$
		\xymatrix{
			\P_U^n \ar[d]_-{j^{\prime}} \ar[r]^-{pr_1^{\prime}} 
			&  U\ar[d]^-{j} \\
			\P_X^n \ar[r]^-{pr_1}
			& X }
		$$
		Since $j$ is smooth, it follows directly from the 
		definition of WCTS that
		\begin{align*}
			\hu(j)\hl(pr_1)(\cl(\tilde{Z})) &= \hl(pr_1^{\prime})\hu(j^{\prime})(\cl(\tilde{Z},\P_X^n)) \\
			&= \hl(pr_1^{\prime})(\cl(\tilde{Z}\cap  \P_U^n,\P_U^n)),
		\end{align*}
		and what is left to be shown is that
		$$
		\hl(pr_1^{\prime})(\cl(\tilde{Z}\cap \P_U^n,\P_U^n)) 
		= \cl(U\cap Z,U).
		$$
		Notice that
		\begin{align*}
			\tilde{Z}\cap \P_U^n &= \tilde{Z}\times_{\P_X^n}
			\P_U^n \\
			&= \tilde{Z} \times_{\P_X^n} \P_X^n \times_X U \\
			&=\tilde{Z} \times_X U.
		\end{align*}
		Furthermore, normalization respects smooth base 
		change, see for example \cite[Tag: 07TD]{Stacks},
		so if we denote the normalization of $Z_U := Z\cap U$
		by $Z_U^{\nu}$, then we have
		\begin{align*}
			Z_U^{\nu} &= Z_U \times_Z \tilde{Z} \\
			&= U\times_X Z \times_Z \tilde{Z} \\
			&= U \times_X \tilde{Z}.
		\end{align*}
		Therefore, $\tilde{Z}\cap \P_U^n \to Z\cap U$ is the
		normalization map, and since $Z\cap U$ is regular
		it is an isomorphism.
		
		We can thus, without loss of generality,
		consider the commutative 
		triangle
		$$
		\xymatrix{
			& \P_X^n \ar[d]^{\pi} \\
			Z \ar@{^{(}->}[r]^{i_X} \ar@{^{(}->}[ur]^{i_P} &
			X,
		}
		$$
		where $X$ is an $\NS$-scheme of $S$-dimension $d_X$,
		$Z$ is an integral regular 
		closed subscheme in $X$ of codimension $c$
		and a regular closed subscheme in $\P_X^n$ of 
		codimension $n+c$ and
		$\pi:\P_X^n \to X$ is the
		projection map. It suffices to show that
		\begin{equation}\label{spreadingoutclasswhattoshow}
			\hl(\pi)(\cl(Z,\P_X^n)) 
			= \cl(Z,X),
		\end{equation}
		in $H^c_Z(X,\Omega_{X/S}^c)$. Let $\sigma: X \to \P_X^n$ be a section 
		of $\pi$. In order to show \eqref{spreadingoutclasswhattoshow} it suffices to
		show
		\begin{equation} \label{spreadingoutclasswhattoshowsection}
			\hl(\sigma)(\cl(Z,X)) = \cl(Z,\P_X^n),
		\end{equation}
		because if \eqref{spreadingoutclasswhattoshowsection} holds then we have
		\begin{align*}
			\cl(Z,X) & = \hl(\pi\circ\sigma)(\cl(Z,X)) \\
			& = \hl(\pi)(\hl(\sigma)(Z,X)) \\
			& = \hl(\pi)(\cl(Z,\P_X^n)).
		\end{align*}
		This follows from Proposition \ref{pushforwardofclass} below.	
	\end{itemize}
\end{proof}

\begin{prop} \label{pushforwardofclass}
	Let $X,Y$ be $\NS$-schemes of 
	$S$-dimensions $d_X$ and $d_Y$ respectively, and $Z$
	a regular, separated $S$-scheme such that the diagram
	$$
	\xymatrix{
		&Y \\
		Z \ar@{^{(}->}[r]_-{i_X} \ar@{^{(}->}[ur]^-{i_Y} 
		& X, \ar@{^{(}->}[u]_-i
	}
	$$
	commutes, where the maps $i_X,  \, i,$ and $i_Y$ are regular closed immersions of codimensions
	$c,\, n,$ and $n+c$ respectively. Then
	\begin{equation}  \label{spreadingoutclasswhattoshowsectiongeneral}
		\hl(i)(\cl(Z,X)) = \cl(Z,Y).
	\end{equation}
\end{prop}
\begin{proof}
	By steps $(1)-(3)$ of the proof of Proposition \ref{spreadingoutclassproposition}
	we may without loss of generality assume that $S= \Spec(R)$, $Y = \Spec(A)$, 
	$X=\Spec(B)$, and $Z= \Spec(C)$. Furthermore, there exist ideals 
	$I\subset A$, $I_Y \subset A$, and $I_X \subset B$ such that
	$$
	B = \frac{A}{I}, \,\,\, \text{and}\,\,\, C = \frac{B}{I_X} = \frac{A}{I_Y}.
	$$
	As $X$ and $Y$ are smooth over $S$, we can assume that there exists 
	an \'etale map of $R$-algebras
	$$
	R[t_1,\ldots,t_{d_Y}] \to A,
	$$
	s.t. $I = (t_1,\ldots,t_n)$ and
	$$
	R[t_{n+1},\ldots, t_{d_Y}] \to B
	$$
	is \'etale. Furthermore, since $Z\hookrightarrow X$ is a regular embedding
	we may assume there exists a regular sequence $s_1,\ldots, s_c$ in $A$ s.t.
	$I_X = (s_1,\ldots,s_c)$. Let $r_1,\ldots,r_c \in B$ be any lifts of $s_1,\ldots,s_c$
	and then $(t_1,\ldots,t_n, r_1,\ldots, r_c)$ is a regular sequence generating $I_Y$. 
	Again, we may shrink $X,Y$ and $Z$ so we can without loss of generality 
	assume that $B$ is a local ring. For a Noetherian local ring, any permutation
	of a regular sequence is again a regular sequence, so we may assume that
	$(r_1,\ldots,r_c,t_1,\ldots,t_n)$ is a regular sequence generating $I_Y$. 
	
	To show \eqref{spreadingoutclasswhattoshowsectiongeneral} it suffices to 
	show that
	\begin{enumerate}
		\item $f(\bar{s}^{\vee}_1 \wedge \cdots \wedge \bar{s}^{\vee}_c \otimes 
		i_X^{\ast}(ds_c\wedge\cdots \wedge ds_1)) = \bar{r}^{\vee}_1 \wedge \cdots \wedge \bar{r}^{\vee}_c \wedge \bar{t}^{\vee}_1\wedge
		\cdots \wedge \bar{t}^{\vee}_n \otimes i_Y^{\ast}(dt_n\wedge\cdots \wedge dt_1 \wedge 
		dr_c\wedge\cdots\wedge dr_1)$, and
		\item the following square commutes
		$$
		\xymatrix{
			\Hom(\sO_Z,\omega_{Z/X}\otimes_{\sO_Z} i_X^{\ast}\Omega_{X/S}^c )
			\ar[d]_-{Hom(\sO_Z,f(-))} \ar[r]^-{\nu_{i_X}} &
			H_Z^c(X,\Omega_{X/S}^c)
			\ar[d]^-{\hl(i)} \\
			\Hom(\sO_Z,\omega_{Z/Y}\otimes_{\sO_Z} i_Y^{\ast}\Omega_{Y/S}^{n+c})
			\ar[r]^-{\nu_{i_Y}} &
			H_Z^{n+c}(Y,\Omega_{Y/S}^{n+c}),
		}
		$$
		where $$
		f: \omega_{Z/X}\otimes_{\sO_Z} i_X^{\ast}\Omega_{X/S}^c 
		\to \omega_{Z/Y}\otimes_{\sO_Z} i_Y^{\ast}\Omega_{Y/S}^{n+c}
		$$ 
		is the map
		described in \eqref{spreadingoutclassmapdescription}, and
		where $\nu_{i_X}$ and $\nu_{i_Y}$ are the compositions defined in 
		\eqref{spreadingoutclasscomposition}.
	\end{enumerate}
	The commutativity of the square is given in Lemma
	\ref{spreadingoutclasscommutativitylemma}.		
	
	If $b_1,\ldots,b_n$ is a basis for $\Omega_{X/S}^{d_X-c}$ then the map
	$\Omega_{X/S}^c  \to \omega_{X/S}\otimes_{\sO_X} \sH om(\Omega_{X/S}^{d_X-c},\sO_X)$ can be explicitly given as
	$$
	\alpha \mapsto \sum_{i=1}^{n}(\alpha \wedge b_i) \otimes b_i^{\vee}.
	$$
	A $B$-basis of $\Omega_{X/S}^{d_X-c} = \Omega_{B/R}^{d_X-c}$ is given by
	\begin{equation} \label{classspreadingoutBbasis}
		dt_I, \,\,\, I = (i_1 < \ldots < i_{d_X-c}) \,\,\, \text{with}\,\,\, i_j \in \{n+1,\ldots, d_Y\},
	\end{equation}
	and an $A$-basis of $\Omega_{Y/S}^{d_X-c} = \Omega_{A/R}^{d_Y-(n+c)}$ is given by
	\begin{equation} \label{classspreadingoutAbasis}
		dt_J, \,\,\, J = (i_1 < \ldots < i_{d_X-c}) \,\,\, \text{with}\,\,\, i_j \in \{1,\ldots, d_Y\}.
	\end{equation}
	Now we compute the image of $\bar{s}^{\vee}_1 \wedge \cdots \wedge \bar{s}^{\vee}_c \otimes i_X^{\ast}(
	ds_c\wedge\cdots \wedge ds_1) =: \bar{s}^{\vee} \otimes i_X^{\ast}ds$ 
	under the composition 
	\eqref{spreadingoutclassmapdescription}.
	\begin{align*}
		\bar{s}^{\vee} \otimes i_X^{\ast}ds 
		&\xrightarrow{\cong}\bar{s}^{\vee}\otimes \sum_I i_X^{\ast}(ds\wedge dt_I\otimes 
		(dt_I)^{\vee})\\
		&\xrightarrow{\zeta_{i,\pi_Y}'} \bar{s}^{\vee} \otimes \sum_I 
		i_X^{\ast}(\bar{t}_1^{\vee}\wedge\cdots \wedge \bar{t}_n^{\vee} \otimes 
		i^{\ast}(dt_n\wedge\cdots\wedge dt_1 \wedge dr_c\wedge\cdots \wedge dr_1 
		\wedge dt_I\otimes (dt_I)^{\vee}) \\
		&\xrightarrow{(\zeta_{i_X,i}')^{-1}} \bar{r}^{\vee} \wedge \bar{t}^{\vee}
		\otimes \sum_I i_Y^{\ast}(dt_n\wedge\cdots\wedge dt_1 \wedge dr_c\wedge\cdots \wedge dr_1 
		\wedge dt_I\otimes (dt_I)^{\vee}) \\
		&\xrightarrow{(i^{\ast})^{\vee}} \bar{r}^{\vee} \wedge \bar{t}^{\vee}
		\otimes \sum_I i_Y^{\ast}(dt_n\wedge\cdots\wedge dt_1 \wedge dr_c\wedge\cdots \wedge dr_1 
		\wedge dt_I\otimes (dt_I)^{\vee}) \\
		&\to \bar{r}_1^{\vee} \wedge\ldots\wedge\bar{r}_c^{\vee}
		\wedge \bar{t}_1^{\vee}\wedge\ldots \wedge\bar{t}_n^{\vee} 
		\otimes i_Y^{\ast}(dt_n\wedge\cdots\wedge dt_1 \wedge dr_c\wedge\cdots \wedge dr_1),
	\end{align*}
	where $\bar{r}^{\vee} := \bar{r}_1^{\vee} \wedge \cdots \wedge \bar{r}_c^{\vee}$ and 
	$\bar{t}^{\vee} := \bar{t}_1^{\vee}\wedge\cdots\wedge\bar{t}_n^{\vee}$.
\end{proof}

\begin{lem} \label{spreadingoutclasscommutativitylemma}
	Let $X,Y$ be $\NS$-schemes of $S$-dimensions $d_X$ and $d_Y$ respectively and
	$Z$ a regular, separated $S$-scheme of finite type 
	such that we have a commutative diagram of $S$-schemes and $S$-morphisms
	$$
	\xymatrix{
		&Y \\
		Z \ar@{^{(}->}[r]_-{i_X} \ar@{^{(}->}[ur]^-{i_Y} 
		& X \ar@{^{(}->}[u]_-i
	}
	$$
	where $i_X$, $i$ and $i_Y$  are regular closed immersions of codimensions
	$c$, $n$ and $n+c$ respectively. Let \linebreak $f: \omega_{Z/X}\otimes_{\sO_Z} i_X^{\ast} 
	(\Omega_{X/S}^c) \to \omega_{Z/Y}\otimes_{\sO_Z} i_Y^{\ast}(\Omega_{Y/S}^{n+c})$
	be the map given by the composition 
	\begin{align} \label{spreadingoutclassmapdescription}
		\omega_{Z/X}\otimes_{\sO_Z} i_X^{\ast}\Omega_{X/S}^c 
		&\cong 
		\omega_{Z/X}\otimes_{\sO_Z} i_X^{\ast}(\sH om(\Omega_{X/S}^{d_X-c},\omega_{X/S})) \\ \notag
		&\cong \omega_{Z/X}\otimes_{\sO_Z} i_X^{\ast}(\omega_{X/S}\otimes_{\sO_X} \sH om(\Omega_{X/S}^{d_X-c},\sO_X)) \\ \notag
		&\xrightarrow{\zeta'_{i,\pi_Y}}
		\omega_{Z/X}\otimes_{\sO_Z} i_X^{\ast}(\omega_{X/Y}\otimes_{\sO_X}i^{\ast}\omega_{Y/S}
		\otimes_{\sO_X} \sH om(\Omega_{X/S}^{d_X-c},\sO_X)) \\ \notag
		&\xrightarrow{(\zeta'_{i_X,i})^{-1}}
		\omega_{Z/Y}\otimes_{\sO_Z} i_Y^{\ast}\omega_{Y/S} 
		\otimes_{\sO_Z} i_X^{\ast}(\sH om(\Omega_{X/S}^{d_X-c},\sO_X)) \\ \notag
		&\xrightarrow{(i^{\ast})^{\vee}}
		\omega_{Z/Y}\otimes_{\sO_Z} i_Y^{\ast}\omega_{Y/S} 
		\otimes_{\sO_Z}i_Y^{\ast}(\sH om(\Omega_{Y/S}^{d_X-c},\sO_{Y})) \\ \notag
		&\cong 
		\omega_{Z/Y}\otimes_{\sO_Z} i_Y^{\ast}(\sH om(\Omega_{Y/S}^{d_Y-(n+c)},
		\omega_{Y/S})) \\ \notag
		&\cong \omega_{Z/Y}\otimes_{\sO_Z} i_Y^{\ast}\Omega_{Y/S}^{n+c},
	\end{align}
	where $i^{\ast}:i^{\ast}\Omega_{Y/S}^{d_X-c}
	\to  \Omega_{X/S}^{d_X-c}$ is the canonical map, and $\zeta'_{i,\pi_Y}: \omega_{X/S} \to \omega_{X/Y}\otimes_{\sO_X}i^{\ast}\omega_{Y/S}$ and 
	\linebreak $\zeta'_{i_X,i}: \omega_{Z/Y} \to \omega_{Z/X}\otimes_{\sO_Z}i_X^{\ast}\omega_{X/Y}$ 
	are isomorphisms, see Notation \ref{notation2}. 
	Then the following
	square commutes
	$$
	\xymatrix{
		\Hom(\sO_Z,\omega_{Z/X}\otimes_{\sO_Z} i_X^{\ast}\Omega_{X/S}^c )
		\ar[d]_-{Hom(\sO_Z,f(-))} \ar[r] &
		H_Z^c(X,\Omega_{X/S}^c)
		\ar[d]^-{\hl(i)} \\
		\Hom(\sO_Z,\omega_{Z/Y}\otimes_{\sO_Z} i_Y^{\ast}\Omega_{Y/S}^{n+c})
		\ar[r] &
		H_Z^{n+c}(Y,\Omega_{Y/S}^{n+c}).
	}  
	$$
\end{lem}
\begin{proof}
	We first notice that $i$ and $i_X$ are 
	l.c.i. morphisms and $\pi_{Y}$ is a separated smooth morphism. So the definitions 
	of $\zeta'_{i,\pi_Y}$ and $\zeta'_{i_X,i}$ are different. Namely, in \cite[\S 2.2.]{Conrad}
	the map $\zeta'_{i,\pi_Y}$ is defined in case $(c)$ and $\zeta'_{i_X,i}$ is defined 
	in case $(b)$.
	
	We break the square into the following two squares
	\begin{enumerate}
		\item $$
		\xymatrix{
			\Hom(\sO_Z,\omega_{Z/X}\otimes_{\sO_Z} i_X^{\ast}\Omega_{X/S}^c )
			\ar[d]_-{Hom(\sO_Z,f(-))} \ar[r]^-{\eta_{i_X}^{-1}} 
			& 	
			\Gamma(Z,\sE xt^c(\sO_Z,i_X^!\Omega_{X/S}^c))
			\ar[d]^-{\Sigma'} \\
			\Hom(\sO_Z,\omega_{Z/Y}\otimes_{\sO_Z} i_Y^{\ast}\Omega_{Y/S}^{n+c})
			\ar[r]^-{\eta_{i_Y}^{-1}} 
			&
			\Gamma(Z,\sE xt^{n+c}(\sO_Z,i_Y^!\Omega_{Y/S}^{n+c})), \,\,\, \text{and}
		}
		$$
		\item $$
		\xymatrix{
			\Ext^c((i_X)_{\ast}\sO_Z,\Omega_{X/S}^c) 
			\ar[d]_-{\Sigma}  \ar[r] 
			&
			H_Z^c(X,\Omega_{X/S}^c) 
			\ar[d]^-{\hl(i)} \\
			\Ext^{n+c}((i_Y)_{\ast}\sO_Z,\Omega_{Y/S}^{n+c})
			\ar[r] 
			&
			H_Z^{n+c}(Y,\Omega_{Y/S}^{n+c}),
		}
		$$
	\end{enumerate}
	where we define $\Sigma'$ such that the first square commutes, which we can do
	since $\eta^{-1}_{i_X}$ is an isomorphism,  and $\Sigma$ is 
	the corresponding map after making the identifications
	\begin{align*}
		\Gamma(Z,\sE xt^c(\sO_Z,i_X^!\Omega_{X/S}^c)) &\cong  \Ext^c((i_X)_{\ast}\sO_Z,\Omega_{X/S}^c), \,\,\, \text{and} \\
		\Gamma(Z,\sE xt^{n+c}(\sO_Z,i_Y^!\Omega_{Y/S}^{n+c})) &\cong \Ext^{n+c}((i_Y)_{\ast}\sO_Z,\Omega_{Y/S}^{n+c}).
	\end{align*}
	We can make these identifications since $\Omega_{X/S}^c$ 
	is a locally free, so in particular a coherent,  $\sO_X$-module and $\Omega_{Y/S}^{n+c}$
	is a locally free, so in particular a coherent, $\sO_Y$-module. 
	The schemes $X$ and $Y$ are regular, hence Cohen-Macaulay so
	we know that for any point $z \in Z$ we have 
	\begin{align*}
		depth_{\sO_{X,z}}((\Omega_{X/S}^c)_z) &\geq \dim(\sO_{X,z}) \geq \codim(Z,X) = c, \,\,\, \text{and} \\
		depth_{\sO_{Y,z}}((\Omega_{Y/S}^{n+c})_z) &\geq \dim(\sO_{Y,z}) \geq \codim(Z,Y) = n+c,
	\end{align*}
	so \cite[Expos\'e III, Proposition 3.3]{SGA2} tells us that
	\begin{align*}
		\sE xt^j(\sO_Z,i_X^!\Omega_{X/S}^c) &= \sE xt^j((i_X)_{\ast}\sO_Z,\Omega_{X/S}^c)
		= 0, \,\,\, \text{for all } j < c, \,\,\,\text{and} \\
		\sE xt^j(\sO_Z,i_Y^!\Omega_{Y/S}^{n+c}) &= \sE
		xt^j((i_Y)_{\ast}\sO_Z,\Omega_{Y/S}^{n+c})
		= 0, \,\,\, \text{for all } j < n+c.
	\end{align*}

	We then consider the square
		$$
		\xymatrix{
			\omega_{Z/X}\otimes_{\sO_Z} i_X^{\ast}\sG 
				\ar[r]^-{\eta_{i_X}^{-1}} \ar[d]_-{(\zeta_{i,i_X}')^{-1}}
				&
				i_X^!\sG[c] \ar[d]^-g\\
				\omega_{Z/Y}\otimes_{\sO_Z} i_Y^{\ast}\omega_{Y/S} 
				\otimes_{\sO_Z}i_X^{\ast}\sH om(\Omega_{X/S}^{d_X-c},\sO_X)
				\ar[r]^-{\eta_{i_Y}^{-1}}
				&
				i_Y^!\omega_{Y/S} 
				\otimes_{\sO_Z}i_X^{\ast}\sH om(\Omega_{X/S}^{d_X-c},\sO_X)[n+c],
			}
		$$
		where $g$ is given by the composition 
		\begin{align*}
				i_X^!(\omega_{X/Y}\otimes_{\sO_X}i^{\ast}\omega_{Y/S}
				&\otimes_{\sO_X} \sH om(\Omega_{X/S}^{d_X-c},\sO_X))[c] \\
				&\cong i_X^!(\omega_{X/Y}\otimes_{\sO_X}i^{\ast}\omega_{Y/S})
				\otimes_{\sO_Z}i_X^{\ast} \sH om(\Omega_{X/S}^{d_X-c},\sO_X)[c] \\
				&\xrightarrow{\eta_i^{-1}}
				i_X^!i^!\omega_{Y/S}[n]\otimes_{\sO_Z}i_X^{\ast} \sH 
				om(\Omega_{X/S}^{d_X-c},\sO_X)[c] \\
					&\cong
					i_Y^!\omega_{Y/S} 
					\otimes_{\sO_Z}i_X^{\ast}\sH om(\Omega_{X/S}^{d_X-c},\sO_X)[n+c].
				\end{align*}
			It is a combination of two commutative squares, see diagram 
			\cite[Diagram 2.5.7]{Conrad} and \cite[Theorem 2.5.1]{Conrad}, and thus commutes.
			The square
			$$
			\xymatrix{
					\omega_{Z/Y}\otimes_{\sO_Z} i_Y^{\ast}\omega_{Y/S} 
					\otimes_{\sO_Z}i_X^{\ast}\sH om(\Omega_{X/S}^{d_X-c},\sO_X)
					\ar[r]^-{\eta_{i_Y}^{-1}} \ar[d]_-{(i^{\ast})^{\vee}}
					&
					i_Y^!\omega_{Y/S} 
					\otimes_{\sO_Z}i_X^{\ast}\sH om(\Omega_{X/S}^{d_X-c},\sO_X)[n+c] 
					\ar[d]_-{(i^{\ast})^{\vee}}\\
					\omega_{Z/Y}\otimes_{\sO_Z} i_Y^{\ast}(\omega_{Y/S} 
					\otimes_{\sO_Y}\sH om(\Omega_{Y/S}^{d_Y-(n+c)},\sO_Y))
					\ar[r]^-{\eta_{i_Y}^{-1}}
					&
					i_Y^!(\omega_{Y/S} 
					\otimes_{\sO_Y}\sH om(\Omega_{Y/S}^{d_Y-(n+c)},\sO_Y))[n+c]
				}
			$$
			is commutative because of functoriality of the fundamental local isomorphism and again by 
			\cite[diagram  2.5.7]{Conrad}. Finally the square
			$$
			\xymatrix{
					\omega_{Z/Y}\otimes_{\sO_Z} i_Y^{\ast}(\omega_{Y/S} 
					\otimes_{\sO_Y}\sH om(\Omega_{Y/S}^{d_Y-(n+c)},\sO_Y))
					\ar[r]^-{\eta_{i_Y}^{-1}} \ar[d]
					&
					i_Y^!(\omega_{Y/S} 
					\otimes_{\sO_Y}\sH om(\Omega_{Y/S}^{d_Y-(n+c)},\sO_Y))[n+c]
					\ar[d] \\
					\omega_{Z/Y}\otimes_{\sO_Z} i_Y^{\ast}\Omega_{Y/S}^{n+c} \ar[r]^-{\eta_{i_Y}^{-1}}
					&
					i_Y^!\Omega_{Y/S}^{n+c}[n+c],
				}	
			$$
			where the vertical maps are given by the composition 
			\begin{align*}
					\omega_{Z/Y}\otimes_{\sO_Z} i_Y^{\ast}(\omega_{Y/S} 
					&\otimes_{\sO_Y}\sH om(\Omega_{Y/S}^{d_Y-(n+c)},\sO_Y)) \\ \notag
					&\cong 
					\omega_{Z/Y}\otimes_{\sO_Z} i_Y^{\ast}(\sH om(\Omega_{Y/S}^{d_Y-(n+c)},
					\omega_{Y/S})) \\ \notag
					&\cong \omega_{Z/Y}\otimes_{\sO_Z} i_Y^{\ast}\Omega_{Y/S}^{n+c},
				\end{align*}
			is commutative because of functoriality of the fundamental local isomorphism.
			
			This shows that square $(1)$ commutes and that $\Sigma$ is 
			given by the composition
			\begin{align*}
					\Ext^c((i_X)_{\ast}\sO_Z,\Omega_{X/S}^c) &\cong 
					\Ext^c((i_X)_{\ast}\sO_Z,\sH om(\Omega_{X/S}^{d_X-c},\omega_{X/S})) \\ \notag
					&\xrightarrow{\zeta'_{i,\pi_{Y}}}
					\Ext^c((i_X)_{\ast}\sO_Z,
					\sH om(\Omega_{X/S}^{d_X-c},\omega_{X/Y}\otimes_{\sO_X}i^{\ast}\omega_{Y/S})) \\ 
					\notag
					&\xrightarrow{\eta_i^{-1}}
					\Ext^{n+c}((i_X)_{\ast}\sO_Z, \sH om(\Omega_{X/S}^{d_X-c},i^!\omega_{Y/S})) \\\notag
					&\xrightarrow{(i^{\ast})^{\vee}}
					\Ext^{n+c}((i_X)_{\ast}\sO_Z,\sH om(i^{\ast}\Omega_{Y/S}^{d_Y-(n+c)},i^!\omega_{Y/S})) 	
					\\ 
					\notag
					&\cong 
					\Ext^{n+c}((i_X)_{\ast}\sO_Z,i^!\sH om(\Omega_{Y/S}^{d_Y-(n+c)},\omega_{Y/S})) \\ 
					\notag
					&\cong
					\Ext^{n+c}((i_Y)_{\ast}\sO_Z, \sH om(\Omega_{Y/S}^{d_Y-(n+c)},\omega_{Y/S})) \\ \notag
					&\cong
					\Ext^{n+c}((i_Y)_{\ast}\sO_Z, \Omega_{Y/S}^{n+c}).
				\end{align*}
	Now we show that square $(2)$ is commutative. The maps $\Ext^c((i_X)_{\ast}\sO_Z,\sF) \to
	H^c_Z(X,\sF)$ and \linebreak $\Ext^{n+c}((i_Y)_{\ast}\sO_Z,\sF) \to
	H^{n+c}_Z(Y,\sF)$ are respectively induced by the natural transformations \linebreak $R\sH om((i_X)_{\ast}\sO_Z,-) \to \rugz(-)$ and $R\sH om((i_Y)_{\ast}\sO_Z,-) \to \rugz(-)$.

	We expand the left vertical map $\Sigma$ in square $(2)$, and we 
	expand the right vertical map $\hl(i)$ as the definition of the pushforward. We get a diagram 
	whose commutativity boils down to the commutativity of
	\begin{equation*} 
		\xymatrix{
			\Ext^{n+c}((i_X)_{\ast}\sO_Z,i^!\sH om(\Omega_{Y/S}^{d_Y-(n+c)},\omega_{Y/S}))
			\ar[d] \ar[r]
			&
			H_Z^{c+n}(Y,i_{\ast}i^!\sH om(\Omega_{Y/S}^{d_Y-(n+c)},\omega_{Y/S})) \ar[d]^-{\Tr_i}
			\\
			\Ext^{n+c}((i_Y)_{\ast}\sO_Z, \sH om(\Omega_{Y/S}^{d_Y-(n+c)},\omega_{Y/S}))
			\ar[r]
			&
			H_Z^{n+c}(Y,\sH om(\Omega_{Y/S}^{d_Y-(n+c)},\omega_{Y/S}))
		}
	\end{equation*}	
	which follows from the definition of $\Tr_i$ as the counit of adjunction for the adjoint pair 
	$(Ri_{\ast},i^!)$;

\begin{equation*} 
	\xymatrix{
		\pi_X^{\sharp}\sO_S \ar[d]_-{\psi_{i,\pi_Y}} 
		& \pi_X^!\sO_S \ar[l]_-{e^{-1}_{\pi_X}} \ar[d]^-{c_{i,\pi_Y}} 
		\\ i^{\flat}\pi_Y^{\sharp}\sO_S
		& i^!\pi_Y^!\sO_S \ar[l]_-{d^{-1}_i\circ e^{-1}_{\pi_Y}} 
	}
\end{equation*}
where $e_f$, $c_{f,g}$ are maps defined in Notation \ref{notation1},  see 
 see \cite[III. Theorem 8.7, Var 5).]{ResAndDual}; and

	\begin{equation*}\label{spreadingoutclasslemmafinaldiagramtriangle}
		\xymatrix{
			\omega_{X/S} \ar[d]_-{\zeta_{i,\pi_Y}'} \ar[r]^-{\psi_{i,\pi_Y}} 
			& i^{\flat}\pi_Y^{\sharp}\sO_S[-d_X] \ar[dl]^-{\eta_i}
			\\ \omega_{X/Y} \otimes_{\sO_X} i^{\ast}\omega_{Y/S}.
		}
	\end{equation*}
	 see \cite[Lemma 3.5.3.]{Conrad}.\footnote{Note that in
	 	\cite[Lemma 3.5.3.]{Conrad} it is claimed that this triangle commutes up to a sign
	 	depending on the relative dimension of $\pi_Y$ and the codimension of $i$. This is 
	 	however not true and is corrected in the online errata to  \cite{Conrad}.}
\end{proof}

The following lemma gives a useful description of the 
local cycle class at generic point of a regular integral closed
subscheme in terms of symbols.
\begin{lem} \label{Classsymbolnotation}
	Let $Y$ be an $\NS$-scheme and $X \subset Y$ be a regular integral  closed subscheme of codimension
	$c$ in $Y$, $U\subset Y$ and open affine subscheme such that the ideal $I$ of $X\cap U$ in $\sO_U$ is generated by
	global sections $t_1,\ldots,t_c$ on $Y$, and let $\eta$ be the generic point of $X$. Then
	$$
	\cl(X,Y)_{\eta} = (-1)^c \begin{bmatrix}
		dt_1\wedge\cdots\wedge dt_c \\
		t_1,\ldots,t_c
	\end{bmatrix}
	$$
	in $H^c_{\eta}(Y,\Omega_{Y/S}^c)$.
\end{lem}
\begin{proof}
	We can without loss of generality assume that $Y = \Spec(A)$ is affine, and that 
	$X = \frac{\Spec(A)}{(t_1,\ldots,t_c)}$.
	By definition we have that $\cl(X,Y)_{\eta}$ is the image of $1$ under the composition
	\begin{align*}
		i_{\ast}\sO_X &\xrightarrow{\phi} i_{\ast}\omega_{X/Y} \otimes_{\sO_Y} \Omega_{Y/S}^c \\
		&\xrightarrow{\eta_i} i_{\ast}i^!(\sO_Y[c]) \otimes_{\sO_Y} \Omega_{Y/S}^c \\
		&\xrightarrow{Tr_i} \Omega_{Y/S}^c[c] \\
		&\to H^c_{\eta}(Y,\Omega_{Y/S}^c),
	\end{align*}
	where $i:X\hookrightarrow Y$ is the closed immersion, 
	$\eta_i$ is the Fundamental Local Isomorphism, see Notation \ref{notation2}, and 
	$\phi$ is the map sending $1$ to $t_1^{\vee}\wedge\cdots\wedge t_c^{\vee} \otimes dt_c\wedge\cdots\wedge dt_1$.
	By applying $\rug_X$ we see that the composition 
	$$
	i_{\ast}\omega_{X/Y} \xrightarrow{\eta_i} i_{\ast}i^!(\sO_Y[c]) \xrightarrow{Tr_i} \sO_Y[c]
	$$
	factors through
	\begin{equation} \label{classsymbolnotationcomposition}
		i_{\ast}\omega_{X/Y} \xrightarrow{\eta_i} i_{\ast}i^!(\sO_Y[c]) \xrightarrow{Tr_i} \rug_X\sO_{Y}[c],
	\end{equation}
	and by \cite[Lemma A.2.5.]{KayAndre} there is a natural isomorphism $\rug_X\sO_{Y}[c] \cong 
	\sH^c_X(\sO_{Y})$ in $D^b_{qc}(Y)$ s.t. \eqref{classsymbolnotationcomposition} composed 
	with this isomorphism is 
	given by
	\begin{align} \label{classsymbolnotationcompositionaftercohomology}
		i_{\ast}\omega_{X/Y} &\to \sH^c_X(\sO_{Y}), \\
		at_1^{\vee}\wedge\cdots\wedge t_c^{\vee} &\mapsto (-1)^{c(c+1)/2}\begin{bmatrix}
			\tilde{a} \\
			t_1,\ldots,t_c 
		\end{bmatrix} \notag
	\end{align}
	where $\tilde{a} \in A$ is any lift of $a \in \frac{A}{(t_1,\ldots,t_c)}$. We therefore have 
	\begin{align}
		\cl(X,Y)_{\eta} & = (-1)^{c(c+1)/2}\begin{bmatrix}
			1 \\
			t_1,\ldots,t_c 
		\end{bmatrix} \otimes dt_c\wedge\cdots\wedge dt_1  \\ \notag
		&=(-1)^{c(c+1)/2} \begin{bmatrix}
			dt_c\wedge\cdots\wedge dt_1 \\
			t_1,\ldots,t_c 
		\end{bmatrix} \\ \notag
		&= (-1)^{c(c+1)/2}(-1)^{c(c-1)/2} \begin{bmatrix}
			dt_1\wedge\cdots\wedge dt_c \\
			t_1,\ldots,t_c 
		\end{bmatrix} \\ \notag
		&= (-1)^c \begin{bmatrix}
			dt_1\wedge\cdots\wedge dt_c \\
			t_1,\ldots,t_c 
		\end{bmatrix} \notag
	\end{align}
\end{proof}

\begin{lem} \label{injectivityoflocalizationlemma}
	Let $X$ be a regular scheme, $V\subset X$ an irreducible closed subset of codimension $c$ with a generic point 
	$\eta$, and let $\sF$ be a finite locally free $\sO_X$-module. Then the localization 
	$$
	H^c_V(X,\sF) \to H^c_{\eta}(X,\sF)
	$$
	is injective.
\end{lem}
\begin{proof}
	Let $U\subset X$ be an open subscheme such that $U\cap Z \neq \emptyset$, i.e. $U$ is an open neighborhood
	of $\eta$. Then we have a long exact sequence
	\begin{equation} \label{localizationinjectiveexaxtsequence}
		\cdots \to H^c_{V'}(X,\sF) \to H^c_V(X,\sF) \to H^c_{V\cap U}(U,\sF)\to \cdots,
	\end{equation}
	where $V' := V\setminus(V\cap U) = V\cap (X\setminus U)$.
	This is obtained from the standard long exact sequence for local cohomology
	$$
	\cdots H^i_{X\setminus U}(X,\sG) \to H^i(X,\sG) \to H^i(U,\sG)\to \cdots,
	$$
	applied to $\sG = \ug_V(\sF)$ which is quasicoherent since $X$ is Noetherian, see \cite[Tag: 07ZP]{Stacks}. Since
	$X$ is regular, hence Cohen-Macaulay, we know that for any point $x \in V'$ we have 
	$$
	depth_{\sO_{X,x}}(\sF_x) \geq \dim(\sO_{X,x}) \geq \codim(V',X) = c+1,
	$$
	so \cite[Expos\'e III, Proposition 3.3]{SGA2} tells us that $H^c_{V'}(X,\sF) = 0$ and thus the long exact 
	sequence
	\eqref{localizationinjectiveexaxtsequence} tells us that 
	$H^c_V(X,\sF) \to H^c_{V\cap U}(U,\sF)$
	is injective. The map $H^c_V(X,\sF) \to H^c_{\eta}(X,\sF)$ is then obtained by taking the direct
	limit over all such neighborhoods $U$ of $\eta$ and is also injective.
\end{proof}

\begin{prop} \label{purehodgesemipurity}
	The weak cohomology theory with 
	supports $(\hpl, \hpu, T, e)$ satisfies semi-purity.
\end{prop}
\begin{proof}
	Without loss of generality we may assume that we have a connected $\NS$-scheme $X$ 
	and an irreducible closed subset $W \subset X$. 
	We denote by c := $\codim(W,X) = \dim_S(X)-\dim_S(W)$.
	Then we need to prove the following
	\begin{enumerate}
		\item $H^p_W(X,\Omega_{X/S}^p) = 0$ when $p > c$, and
		\item The map $\hu(j): H^c_W(X,\Omega_{X/S}^c)\to  H^c_{U\cap W}(U,\Omega_{X/S}^c)$ is injective
		where $U$ is an open subscheme of $X$ that intersects $W$ and $j: (U,W\cap U) \to (X,W)$ is 
		induced by the open immersion $U \hookrightarrow X$.
	\end{enumerate}
	Condition $(1)$ is well known, see \cite[Exp. III, Proposition 3.3]{SGA2},
	and condition $(2)$ has been proven as part of the proof of Lemma \ref{injectivityoflocalizationlemma}. 
\end{proof}

\begin{lem} \label{Cupproductsymbolnotation}
	Let $X = \Spec(A)$ be  an affine $\NS$-scheme and let
	$V \subset X$ and $W \subset X$ be regular integral closed subschemes 
	of codimensions $c$ and $e$ respectively. Furthermore we write 
	$I_V = (t_1,\ldots,t_c)$ and $I_W = (s_1,\ldots,s_e)$ where $t_1,\ldots,t_c$
	and $s_1,\ldots,s_e$ are regular sequences in $A$. 
	Let $M$ and $N$ be $A$-modules, then for any $m \in M$
	and $n \in N$ we have 
	\begin{equation} \label{cupproductofsymbolswhattoshow}
			\begin{bmatrix}
					m \\
					t_1,\ldots,t_c
				\end{bmatrix} \otimes_A 
			\begin{bmatrix}
					n \\
					s_1,\ldots,s_e
				\end{bmatrix} =
			\begin{bmatrix}
					m \otimes n \\
					t_1,\ldots,t_c, s_1,\ldots, s_e
				\end{bmatrix}.
		\end{equation}
\end{lem}
\begin{proof}
	Recall that we construct $\begin{bmatrix}
			m \\
			t_1,\ldots,t_c
		\end{bmatrix}$
	as the image of $m$ under the composition
	$$
	M \to \frac{M}{I_V M} \xrightarrow{\cong} H^c(t,M) \to H^c(X,\tilde{M}),
	$$
	where $\tilde{M}$ is the $\sO_X$-module associated with $M$. 
	We furthermore know that 
	$$
	H^c(t,M) \cong \Ext^c(A/I_V,M),
	$$
	so we can consider the class 
	$$
	\begin{bmatrix}
			m \\
			t_1,\ldots,t_c
		\end{bmatrix}' \in \Ext^c(A/I_V,M),
	$$
	as the image of $m\in M$ under the composition 
	$$
	M \to \frac{M}{I_V M} \xrightarrow{\cong} H^c(t,M)  \xrightarrow{\cong}
	\Ext^c(A/I_V,M),
	$$
	and if we can prove
	\begin{equation} \label{cupproductssymbolswhattoshowreally}
			\begin{bmatrix}
					m \\
					t_1,\ldots,t_c
				\end{bmatrix}' \otimes_A
			\begin{bmatrix}
					n \\
					s_1,\ldots,s_e
				\end{bmatrix}' =
			\begin{bmatrix}
					m \otimes n \\
					t_1,\ldots,t_c, s_1,\ldots, s_e
				\end{bmatrix}',
		\end{equation}
	then \eqref{cupproductofsymbolswhattoshow} follows.
	
	We note that 
	$$
	\Ext^c(A/I_V,M) = \Hom_{D(A)}(A/I_V,M[c]) = H^0(\Hom_{K(A)}(K^{-\bullet}(t),M[c])),
	$$
	where $K(A)$ is the homotopy category of the category $A-Mod$ and 
	the second equality follows from the fact that $K^{-\bullet}(t)$ is a free
	resolution of $A/I_V$ in $A-Mod$. In $H^0(\Hom_{K(A)}(K^{-\bullet}(t),M[c])),$ the symbol
	$\begin{bmatrix}
			m \\
			t_1,\ldots,t_c
		\end{bmatrix}'$ corresponds to the map that is the zero map 
	in all degrees except degree $-c$ and is the map 
	\begin{align*}
			\bigwedge^c(A^c) &\to M, \\
			e_1\wedge\cdots\wedge e_c &\mapsto  (-1)^{\frac{c(c+1)}{2}} m,
		\end{align*} 
	in degree $-c$.  Similarly $\begin{bmatrix}
			n \\
			s_1,\ldots,s_e
		\end{bmatrix}'$ corresponds to the map in 
	$H^0(\Hom^{\bullet}_{K(A)}(K^{-\bullet}(s),N[e]))$ that is the zero map in all 
	degrees except $-e$ and is 
	\begin{align*}
			\bigwedge^e(A^e) &\to N, \\
			f_1\wedge\cdots\wedge f_e &\mapsto (-1)^{\frac{e(e+1)}{2}} n,
		\end{align*}
	in degree $-e$, and $\begin{bmatrix}
			m \otimes n \\
			t_1,\ldots,t_c,  s_1,\ldots,  s_e
		\end{bmatrix}'$ corresponds to the map in \linebreak
	$H^{0}(\Hom^{\bullet}_{K(A)}(K^{-\bullet}(t,s),M\otimes_A N[c+e]))$ 
	that is the zero map in all degrees except $-c-e$ and is the map
	\begin{align*}
			\bigwedge^{c+e}(A^{c+e}) &\to M\otimes N, \\
			e_1\wedge\cdots\wedge e_c \wedge f_1\wedge 
			\cdots \wedge  f_e &\mapsto \mapsto (-1)^{\frac{(c+e)(c+e+1)}{2}}m\otimes n,
		\end{align*}
	in degree $-c-e$, where $t,s$ denotes the regular sequence  
	$t_1,\ldots,t_c,s_1,\ldots,s_e$. But then \eqref{cupproductssymbolswhattoshowreally} follows from 
	the definition of Koszul complexes as tensor products, see for example
	\cite[\S 18.D]{matscommalg}.
\end{proof}

\begin{lem} \label{pushforwardfinitemaphodgecohomology}
	Let $f: X \to 
	Y$ be a 
	morphism of $\NS$-schemes. Let $W\subset X$ be a 
	regular closed integral
	subscheme such that the restricted map
	$$
	f|_W:W \to f(W)
	$$
	is  finite of degree $d$. Then 
	$$
	\hl(f)(\cl(W,X)) = d \cdot \cl(f(W),Y).
	$$
\end{lem}
\begin{proof}
	We write $c := \codim(W,X)$ and $e := \codim(f(W),Y)$. By Lemma \ref{injectivityoflocalizationlemma}
	we know that 
	$$
	H^e_{f(W)}(Y,\Omega_{Y/S}^e) \to H^e_{\xi}(Y,\Omega_{Y/S}^e)
	$$
	is injective, where $\xi$ is the generic point of $f(W)$, so we can shrink $Y$ around $\xi$. Furthermore, if 
	$\eta$ is the generic point of $W$ then $\eta \in f^{-1}(\xi)$ is a closed point, since $f$ is finite. Therefore
	we can shrink $X$ around $\eta$.  Without loss of generality we may therefore assume that we can factor
	$f$ through some $\A^n_Y$, i.e. we may assume that $Y = \Spec(A)$ and 
	$X= \Spec(R)$ are affine and then $R$ is a finitely generated $A$-algebra 
	so there exists some $n$ such that we have a surjection 
	$$
	A[x_1,\ldots, x_n] \twoheadrightarrow R,
	$$
	i.e. a closed immersion $i:X \to \A^n_Y$ such
	that the following diagram commutes
	$$
	\xymatrix{
		W \ar@{^{(}->}[r] \ar[d]_-{f|_W}
		&
		X \ar[d]^-{f} \ar@{^{(}->}[r]^-i 
		&
		\A^n_Y \ar[dl]^-{p},
		\\
		f(W) \ar@{^{(}->}[r] 
		&
		Y
	}
	$$
	where $p$ denotes the projection $\A^n_Y \to Y$.
	Since $i:X\to  \A^n_Y$ is a separated morphism of finite type over 
	the Noetherian base $Y$, and since 
	$W$ is proper over $Y$,  we can view $W$ as a regular integral 
	closed subscheme of $\A^n_Y$.
	Furthermore
	$$
	\xymatrix{
		H_W^c(X,\Omega_{X/S}^c) \ar[r]^-{\hl(i)} \ar[d]^-{\hl(f)} 
		&
		H_W^{n+e}( (\A^n_Y, \Omega_{ \A^n_Y/S}^{n+e}) \ar[dl]^-{\hl(p)}
		\\
		H^e_{f(W)}(Y,\Omega_{Y/S}^e)
	}
	$$
	commutes because of the functoriality of the pushforward, and by Proposition \ref{pushforwardofclass}
	we have 
	$$
	\hl(i)(\cl(W,X)) = \cl(W,\A^n_Y).
	$$ 
	We can therefore without loss of generality reduce to the situation
	\begin{equation} \label{finiteconditionwhattoshowprojective1}
		\xymatrix{
			W \ar@{^{(}->}[r] \ar[d]_-{p|_W} 
			&
			\A^n_Y \ar[d]^-p
			\\
			p(W) \ar@{^{(}->}[r]
			&
			Y,
		}
	\end{equation}
	where $p$ is the projection, and $p|_W:W \to p(W)$ is finite of degree $d$. Furthermore if $n\geq 2$, we can factor 
	\eqref{finiteconditionwhattoshowprojective1} as 
	\begin{equation}
		\xymatrix{
			W \ar@{^{(}->}[r] \ar[d]_-{(q_1)|_W} 
			&
			\A^n_Y \ar[d]^-{q_1}
			\\
			q_1(W) \ar@{^{(}->}[r]\ar[d]_-{(q_2)|_{q_1(W)}} 
			&
			\A^{n-1}_Y \ar[d]^-{q_2}
			\\
			p(W)\ar@{^{(}->}[r]
			&
			Y,
		}
	\end{equation}
	where $q_1: \A^n_Y  \to \A^{n-1}_Y$ and $q_2: \A^{n-1}_Y \to Y$ are the projections 
	and then $(q_1)|_W$ and $(q_2)|_{q_1(W)}$ will be finite of degrees $d_1$ and $d_2$ respectively, with
	$d = d_1d_2$. Furthermore via an imbedding $\A^1_Y \to \P^1_Y$ we can view 
	$W$ as a regular integral closed subscheme of $\P^1_Y$ and we can therefore without loss of generality further reduce to the case
	\begin{equation}
		\xymatrix{
			W \ar@{^{(}->}[r] \ar[d]_-{p|_W} 
			&
			\P^1_S \times_S Y \ar[d]^-p
			\\
			p(W) \ar@{^{(}->}[r]
			&
			Y,
		}
	\end{equation}
	where  $p$ is the projection, and $p|_W: W \to p(W)$ is finite
	of degree $d$. 
	We can further shrink $Y$ around $\xi$ to assume 
	$p(W)$ is cut out by a regular sequence in $A$, $p(W) = V(s_1,\ldots, s_e)$, and $W = V(s_1,\ldots,s_e, g)$ where $g$ is a monic irreduble polynomial in $A[t]$
	of degree $d$. By Lemmas \ref{Classsymbolnotation} and \ref{Cupproductsymbolnotation} we see that 
	$$
	\cl(W,\P^1_Y)_{\eta} = (-1)^{e+1}\begin{bmatrix}
		ds_1\wedge\cdots\wedge ds_e\wedge g \\
		s_1,\ldots,s_e,g
	\end{bmatrix} = (-1)^e\begin{bmatrix}
		ds_1\wedge\cdots\wedge ds_e \\
		s_1,\ldots,s_e
	\end{bmatrix} \cup(-1) \begin{bmatrix}
		dg \\
		g
	\end{bmatrix}.
	$$
	But by  \cite[Exp. II, Proposition 5.]{SGA2} and  Lemma \ref{Classsymbolnotation} we have that
	$$
	(-1)^e\begin{bmatrix}
		ds_1\wedge\cdots\wedge ds_e \\
		s_1,\ldots,s_e
	\end{bmatrix} = \hu(p)(\cl(p(W),Y)_{\xi})
	$$
	and 
	$$
	(-1) \begin{bmatrix}
		dg \\
		g
	\end{bmatrix} = \cl(Z,\P_Y)_{\zeta},
	$$
	where $Z = V(g) \subset \P^1_Y$ is a divisor,
	and $\zeta$ is its generic point, i.e. we have 
	$$
	\cl(W,\P^1_Y)_{\eta} = \hu(p)(\cl(p(W),Y)_{\xi}) \cup \cl(Z,\P_Y)_{\zeta},
	$$
	and using the second projection formula,  we see that
	\begin{align*}
		\hl(p)(\cl(W,\P^1_Y)_{\eta}) &= \hl(p)(\hu(p)(\cl(p(W),Y)_{\xi}) \cup \cl(Z,\P_Y)_{\zeta}) \\
		&= \cl(p(W),Y)_{\xi} \cup \hl(p)(\cl(Z,\P_Y)_{\zeta}).
	\end{align*}
	This shows that to prove the lemma, it suffices to prove
	\begin{equation} \label{finitepushforwardwhattoshowdivisor}
		\hl(p)(\cl(Z,\P^1_Y)) = d \in H^0(Y,\sO_Y),
	\end{equation}
	where $Y = \Spec(A)$, $Z = V(g) \subset \P_Y^1$, and $g$ is a monic irreducible polynomial in $A[t]$.
	We can base-change to the function field $K = \kappa(Y)$ and without loss of generality it suffices 
	to show that
	\begin{equation} \label{finitepushforwardwhattoshowfield}
		\hl(p)(\cl(z,\P_K^1)) = d \in K,
	\end{equation}
	where $K$ is a field, $z$ is a closed point of $\P^1_K$ of degree $d$ and $p:\P^1_K \to K$ is the projection.
	Locally we write $z = (g) \in K[t]$ where $g$ is monic, irreducible and of degree $d$. 
	
	Now let $x \in \P^1_K$ be any closed point, say $x = (h) \in K[t]$. Write $R$ for the regular local ring
	$\sO_{\P_K^1,x}$, $\mathfrak{m} = (h) \subset R$ for the maximal ideal, and $P_{(x)} := \Spec(R)$. 
	The standard long exact sequence in local cohomology for $P_{(x)}, U := P_{(x)}\setminus \{x\}$ 
	and $\sF := \Omega_{P_{(x)}/S}^1$ 
	is 
	\begin{align} \label{pushforwardfinitelongexact}
		0\to H^0_x(P_{(x)},\sF) \to H^0(P_{(x)}, \sF)& \to \\
		\to H^0(U,\sF) &\to H^1_x(P_{(x)},\sF) \to H^1(P_{(x)}, \sF)\to \ldots  \notag
	\end{align} 
	We note that 
	$$
	H^0_x(P_{(x)},\sF) = 0,
	$$
	since any section of the locally free $\sF$ that vanishes everywhere
	except possibly at $x$ must also vanish at $x$, and 
	$$
	H^1(P_{(x)}, \sF)\ = 0,
	$$
	since $P_{(x)}$ is affine.
	So we have a short exact sequence
	$$
	0 \to H^0(P_{(x)}, \sF) \to 
	H^0(U,\sF) \to H^1_x(P_{(x)},\sF) \to 0,
	$$
	and so 
	\begin{equation*}
		H^1_x(P_{(x)},\Omega_{P_{(x)}/S}^1) = \frac{H^0(U,\Omega_{P_{(x)}/S}^1)}{H^0(P_{(x)}, \Omega_{P_{(x)}/S}^1)}.
	\end{equation*}
	We futhermore note that 
	\begin{align*}
		H^0(U,\Omega_{P_{(x)}/S}^1) &= \Omega^1_{K(t)/S}, \,\,\, \text{and} \\
		H^0(P_{(x)}, \Omega_{P_{(x)}/S}^1) &= \Omega^1_{\P^1_K/S, x}.
	\end{align*}
	Consider the  commutative diagram
	$$
	\mathclap{
		\begin{tikzcd}[ampersand replacement=\&]
			H^1_x(P_{(x)},\Omega_{P_{(x)}/S}^1) = \frac{\Omega^1_{K(t)/S}}{\Omega^1_{\P^1_K/S, x}} \arr[d,equal] 
			\&\& \Omega^1_{R/S}[\frac{1}{h}] \arr[ll] 
			\\
			\lim\limits_{\to} \frac{\Omega_{R/S}^1}{h^n\Omega_{R/S}^1} 
			\&\& \Omega_{R/S}^1 \arr[ll,"\begin{bmatrix}
				\alpha \\
				h
			\end{bmatrix} \mapsfrom \alpha "'] \arr[dl] \arr[u,"\alpha \mapsto \alpha/h"']
			\\ \& \frac{\Omega_{R/S}^1}{h\Omega_{R/S}^1}. \arr[ul] 
		\end{tikzcd}
	}
	$$
	We consider this specifically for $x = z$ and $\alpha = dg$, i.e. we are considering
	\begin{align*}
		\Omega_{R/S}^1 \to \Omega_{R/S}^1[\frac{1}{g}] &\to H^1_z(\P^1_K,\Omega_{K/S}^1) \\
		dg \mapsto \frac{dg}{g} &\mapsto \begin{bmatrix}
			dg \\
			g
		\end{bmatrix}.
	\end{align*}
	The Cousin complex yields an exact sequence
	\begin{equation} \label{finitepushforwardcousin}
		\Omega^1_{K(t)/K} \to \bigoplus_{x\in \P_K^1}H^1_x(\P_K^1,\Omega_{\P_K^1/S}^1) \to
		H^1(\P_K^1,\Omega_{\P_K^1/S}^1) \to 0,
	\end{equation}
	where the sum is taken over all closed points $x$ in $\P_K^1$. We clearly have a commutative triangle
	\begin{equation} \label{finitepushforwardcommutativetriangle}
		\xymatrix{ \bigoplus_{x\in \P_K^1}H^1_x(\P_K^1,\Omega_{\P_K^1/S}^1) 
			\ar[r]^-{\Sigma} \ar[dr]_-{\oplus \hl(p)} 
			&  H^1(\P_K^1,\Omega_{\P_K^1/S}^1) \ar[d]^-{\hl(p)} 
			\\&K
		}
	\end{equation}
	Now $g \in K[t]$ is an irreducible monic polynomial of degree $d$, say 
	$$
	g(t) = t^d+a_{d-1}t^{d-1}+ \cdots + a_1 t + a_0
	$$ and $\dlog(g) \in \Omega^1_{K(t)/L}$ where $L$ is the image of $K$ in $S$. We have 
	$$
	\dlog(g) \in \Omega^1_{\P^1_K/L,x},
	$$
	for all $x \in \P^1_K \setminus \{z,\infty\}$. Write $\mu = \frac{1}{t}$ and then
	$$
	g = \mu^{-d}(1+a_{d-1}\mu+\cdots+a_0\mu^d),
	$$
	and note that $1+a_{d-1}\mu+\cdots+a_0\mu^d \in \sO_{\P_K^1,\infty}^{\times}$ and this implies
	that 
	\begin{equation}
		\dlog(g) = -\dlog(\mu) + \dlog(1+a_{d-1}\mu+\cdots+a_0\mu^d) = -\dlog(\mu),
	\end{equation}
	in $\Omega_{\P_K^1,\infty}[\frac{1}{\mu}] /\Omega_{\P_K^1,\infty}$. The Cousin complex
	\eqref{finitepushforwardcousin} gives
	\begin{align}
		\Omega^1_{K(t)/K} \to &\bigoplus_{x\in \P_K^1}H^1_x(\P_K^1,\Omega_{\P_K^1/S}^1) \to
		H^1(\P_K^1,\Omega_{\P_K^1/S}^1) \\
		\dlog(g) \mapsto &(\alpha_x)_{x\in\P_K^1} \mapsto 0, \notag
	\end{align}
	where 
	$$
	\alpha_x = \begin{cases}
		0 & x\neq z, \infty, \\
		\cl(z,\P_K^1) & x=z,\\
		d\cdot \cl(\infty, \P_K^1) & x = \infty.\\
	\end{cases} 
	$$
	By \eqref{finitepushforwardcommutativetriangle}, this means that
	$$
	\hl(p)(\cl(z,\P_K^1)) = d\cdot \hl(p)(\cl(\infty,\P_K^1))
	$$
	in $K$. Note that we have a commutative triangle
	$$
	\xymatrix{
		& \P_K^1 
		\\ \infty \ar@{^{(}->}[ur] \ar[r]^-{\cong} 
		& \Spec(K) \ar@{^{(}->}[u]
	}
	$$
	so by the proof of Proposition \ref{spreadingoutclassproposition},
	specifically \eqref{spreadingoutclasswhattoshow} we see that 
	$$
	\hl(p)(\cl(\infty,\P_K^1)) = \cl(\Spec(K),\Spec(K)) = 1,
	$$
	i.e. 
	$$
	\hl(p)(\cl(z,\P_K^1)) = d.
	$$
\end{proof}

\begin{lem} \label{pullbacksmoothhodgecohomology}
	If $f:X 
	\to Y$ is a smooth morphism between
	$\NS$-schemes $X$ and $Y$, and 
	$W \subset Y$ 
	is a regular integral closed subscheme, then
	$$
	\hu(f)(\cl(W,Y)) = \cl(f^{-1}(W),X).
	$$
\end{lem}
\begin{proof}
	Without loss of generality we may assume that $f^{-1}(W)$ has a unique generic point. 
	Denote the generic point of $W$ by $\eta$ and the generic point of $f^{-1}(W)$
	by $\nu$.	By Lemma \ref{injectivityoflocalizationlemma}, it suffices to show that 
	$$
	\hu(f)(\cl(W)_{\eta}) = \cl(f^{-1}(W))_{\nu}.
	$$
	From the definition of the pullback we have a commutative diagram
	$$
	\xymatrix{
		H_{\eta}^c(Y,\fpush\Omega_{X/S}^c) \ar[r] 
		& H_{\eta}^c(Y,\rfpush\Omega_{X/S}^c) \ar[r]
		&H_{\nu}^c(X,\Omega_{X/S}^c) 
		\\H_{\eta}^c(Y,\Omega_{Y/S}^c), \ar[u]\ar[rru]
	}
	$$   
	and  \cite[Exp. II, Proposition 5.]{SGA2} tells us that the square
	$$
	\xymatrix{
		\fpush\Omega_{X/S}^c \ar[r] 
		& H_{\eta}^c(Y,\fpush\Omega_{X/S}^c)
		\\ \Omega_{Y/S}^c \ar[r] \ar[u] 
		&H_{\eta}^c(Y,\Omega_{Y/S}^c) \ar[u]
	}	
	$$
	commutes. Combining these diagrams with Lemma \ref{Classsymbolnotation} then shows
	that
	$$
	\hu(f)(\cl(W)_{\eta}) = \cl(f^{-1}(W))_{\nu}.
	$$
\end{proof}
We have the following corollary to Proposition \ref{pushforwardofclass}
and Lemma \ref{pullbacksmoothhodgecohomology} that shows that when the integral closed 
subscheme
$X \subset Y$ is smooth, our class element is defined in an analogous manner to the definition
in \cite{KayAndre}. In particular, when the base scheme $S$ is $\Spec(k)$, where 
$k$ is perfect field of positive characteristic, then our definitions coincide. 

\begin{cor} \label{corollarysmoothclass}
	For any $\NS$-scheme $Y$ we have 
	$$
	1_X = \cl(Y,Y).
	$$
	Furthermore,
	let $Y$ be an $\NS$-scheme and let $i:X\hookrightarrow Y$ be an 
	integral closed subscheme of 
	$Y$ that is smooth over $S$. Then 
	$$
	\hl(i)(1_X) = \cl(X,Y).
	$$
\end{cor}
\begin{proof}
	It is clear from the definition that $1_S = \cl(S,S)$. Therefore Lemma
	\ref{pullbacksmoothhodgecohomology} applied to the smooth structure morphism
	$\pi_Y: Y\to S$ tells us that 
	\begin{align*}
		1_Y &:= \hu(\pi_Y)(e(1))\\
		&= \hu(\pi_Y)(\cl(S,S)) \\
		&= \cl(\pi_Y^{-1}(S),Y) \\
		&= \cl(Y,Y).
	\end{align*}
	Now let $Y$ be an $\NS$-scheme and let $i:X\hookrightarrow Y$ be an 
	integral closed subscheme of 
	$Y$ that is smooth over $S$. Then $X$ is an $\NS$-scheme and by letting $Z = X$, 
	$i_X = id_X$ and $i_Y = i$ in Proposition \ref{pushforwardofclass}, the result follows 
	immediately.
\end{proof}

\begin{lem} \label{injectivitylemmafordivisorcondition}
	Let $X$ be an $\NS$-scheme and $\imath:D\subset X$ be the inclusion of a smooth divisor over $S$.
	Let $\Phi$ be a family of supports on $D$ and denote by $\imath_1:(D,\Phi) \to (X,\Phi)$ the 
	map in $\vl$ induced by $\imath$. Then $\hl(\imath_1): H_{\Phi}^i(D,\Omega_{D/S}^j) \to 
	H_{\Phi}^{i+1}(X,\Omega_{X/S}^{j+1})$ is the connecting homomorphism of the long
	exect cohomology sequence associated to the short exact sequence
	$$
	0 \to \Omega_{X/S}^{j+1} \to \Omega_{X/S}^{j+1}(log D) \xrightarrow{Res} \imath_{\ast}\Omega_{D/S}^j\to 0,
	$$
	where $Res(\frac{dt}{t}\alpha) =\imath^{\ast}(\alpha)$ for $t \in \sO_X$ a regular element defining 
	$D$ and $\alpha \in \Omega_{X/S}^j$. In particular, if $\Phi\subset X$ is supported in codimension 
	$\geq i+1$ in $X$, then $\hl(\imath_1)$ is injective on $H_{\Phi}^i$. 
\end{lem}
\begin{proof}
	This is \cite[Lemma 2.3.8.]{KayAndre} and the proof there works in our situation as well.
\end{proof}

\begin{lem} \label{hodgedividorcondition} \rm{(}cf.\cite[Lemma 3.1.5.]{KayAndre})
	Let $i:X\to Y$ be the closed immersion 
	of an irreducible, regular, closed 
	$S$-subscheme $X$ into an 
	$\NS$-scheme $Y$. For any 
	effective smooth divisor $D \subset Y$
	such that 
	\begin{itemize}
		\item $D$ meets $X$ properly, thus 
		$D\cap X := D\times_Y X$ is a divisor
		on $X$,
		\item $D' := (D\cap X)_{\rm red}$ 
		is regular and irreducible, so 
		$D\cap X = n\cdot D'$ as divisors
		(for some $n\in \Z, n\geq 1$).
	\end{itemize}
	We define
	$g:(D,D') \to (Y,X)$ in $\vu$ as the 
	map induced by the
	inclusion $D\subset Y$. Then the following
	equality holds:
	\begin{equation} \label{hodgedivisorcondtitionwhattoshow}
		\hu(g)(\cl(X,Y)) = n\cdot\cl(D',D).
	\end{equation}
\end{lem}
\begin{proof}
	In light of Corollary \ref{corollarysmoothclass}, the proof of \cite[Lemma 3.1.5.]{KayAndre}
	carries over without change. 
\end{proof}

The following lemma tells us what happens when $V \subset
X$ is a regular integral closed subscheme that \textit{lies 
entirely in the fibre over a closed point}. This is of course a 
new phenomenon that we don't see in the case over a field.
\begin{lem} \label{cycleclasseslivinginfibresvanish}
	Assume $\dim(S)=1$. Let $X$ be an $\NS$-scheme and $V \subset X$ be a regular integral closed subscheme. If $V$
	lies in the fiber over a closed point $s\in S$, then 
	$$
	\cl(V,X) = 0.
	$$
\end{lem}
\begin{proof}
	In light of Lemma \ref{injectivityoflocalizationlemma} we note that it suffices to show
	that 
	$$
	\cl(V,X)_{\eta} = 0,
	$$
	where $\eta$ is the generic point of $V$. Without loss of generality we 
	may restrict to the case where $S = \Spec(R)$ for some ring $R$ and 
	then $s = \Spec(R/(\sigma)$ for some $\sigma \in R$. Furtheremore we may
	assume that $V$ is globally cut out by a regular sequence. Then we may
	choose that regular sequence to be $\sigma, t_1, \ldots, t_{c-1}$ for some
	$t_1, \ldots, t_{c-1}$ where $c = \codim(V,X)$. But then Lemma 
	\ref{Classsymbolnotation} we have 
	$$
	\cl(V,X)_{\eta} = (-1)^c\begin{bmatrix}
		d\sigma\wedge dt_1\wedge \cdots \wedge dt_{c-1}\\
		\sigma, t_1, \ldots, t_{c-1}
	\end{bmatrix} = 0,
	$$
	since $d\sigma = 0$.
\end{proof}

\begin{lem} \label{Tcondition}
	Let $X$ and $Y$ be $\NS$-schemes and $V$ and $W$ be regular integral closed 
	subschemes in $X$ and $Y$ respectively.  Then
	\begin{equation}\label{Tconditionwhattoshow}
		T(\cl(V,X)\otimes \cl(W,Y)) = \cl(V\times_S W,X\times_S Y).
	\end{equation}
\end{lem}
\begin{proof}
	By Lemma \ref{cycleclasseslivinginfibresvanish} we see that if either of $V$ or $W$ are not 
	dominant over $S$, then both sides of \eqref{Tconditionwhattoshow} vanish and the statement
	holds trivially. So we may assume that both $V$ and $W$ are dominant over $S$. 
	Let $\codim(V,X) = c$ and $\codim(W,Y) = e,$ and note that since the 
	construction of the cycle class is a local question, we may assume 
	that $X$ and $Y$ are affine. The statement follows 
	directly from writing the cycle class in symbol notation, see 
	Lemma \ref{Classsymbolnotation}, and the cup product of 
	symbols, see Lemma \ref{Cupproductsymbolnotation}.
\end{proof}

\begin{lem} \label{condition1}
	Let $i_0: S\to \P_S^1$ and $i_{\infty}: S\to \P_S^1$ be the zero-section and 
	the infinity-section respectively. Let $e: \Z\to H(S,S)$ be the unit. Then
	$$
	\hl(i_0)\circ e = \hl(i_{\infty}) \circ e.
	$$
\end{lem}
\begin{proof}
	It is enough to show that 
	$$
	\hl(i_0)\circ e(1) = \hl(i_{\infty}) \circ e(1),
	$$
	and since $e(1) = \cl(S,S)$, it follows from Lemma \ref{pushforwardfinitemaphodgecohomology}
	that it suffices to show that
	\begin{equation} \label{condition1whattoshow}
		\cl(0,\P_S^1) = \cl(\infty,\P_S^1).
	\end{equation}
	Furthermore we note that by assumption $S$ is integral, and we may without loss
	of generality assume it is affine, say $S = \Spec(R)$. Denote by $K$ the fraction
	field of $R$. A \v{C}ech cohomology computation shows that  
	$$
	H^1(\P_S^1, \Omega^1_{\P_S^1/S}) \cong 
	\frac{\Omega^1_{R[t,1/t]/R}}{\{a-b|a\in \Omega^1_{R[t]/R}, b\in \Omega^1_{R[1/t]/R}\}},
	$$
	and the map 
	\begin{align*}
		R &\to \Omega^1_{R[t,1/t]/R},  \\
		\lambda &\mapsto \lambda\cdot \dlog(t),
	\end{align*}
	induces an isomorphism
	$$
	H^0(S,\sO_S) \cong H^1(\P_S^1, \Omega^1_{\P_S^1/S}).
	$$
	We have a commutative square
	$$
	\xymatrix{
		H^0(S,\sO_S)  \ar[d]^-{\cong} \ar@{^{(}->}[r] 
		&K \ar[d]^-{\cong}
		\\
		H^1(\P_S^1, \Omega^1_{\P_S^1/S}) \ar[r]
		&H^1(\P_K^1, \Omega^1_{\P_K^1/K}),
	}
	$$
	and since the map $H^0(S,\sO_S) \to K$ is injective, this implies that the map
	$$
	H^1(\P_S^1, \Omega^1_{\P_S^1/S}) \to H^1(\P_K^1, \Omega^1_{\P_K^1/K})
	$$ 
	is
	also injective. Therefore, it suffices to show \eqref{condition1whattoshow}
	holds in $H^1(\P_K^1, \Omega^1_{\P_K^1/K})$, i.e. to show
	$$
	\cl(0,\P_K^1) = \cl(\infty,\P_K^1).
	$$
	This follows directly from the Cousin argument in the proof of 
	Lemma \ref{pushforwardfinitemaphodgecohomology}, for $g=t$.
\end{proof}

\begin{thm} \label{ExistenceofmorphismfromChowtoHodge}
	There exists a morphism $\cl: \CH \to H$ in $\TT$.
\end{thm}
\begin{proof}
	By Proposition \ref{purehodgesemipurity}, $HP$ satisfies 
	semi-purity, and Lemmas \ref{condition1}, \ref{pullbacksmoothhodgecohomology},
	\ref{hodgedividorcondition}, \ref{pushforwardfinitemaphodgecohomology}, and
	\ref{Tcondition}, show that there exists a morphism $\CH \to HP$ in 
	$\TT$ and we obtain the desired morphism by composing this with
	the inclusion $HP \subset H$.
\end{proof}

\section{Proof of the Main Theorem}

Having set up the machinery of the actions of correspondences on Hodge
cohomology with supports in the previous four sections, the proof of Theorem \ref{maintheorem} is the same as the proof of the case 
over a perfect field of positive characteristic, i.e. the proof of \cite[Theorem 
3.2.8.]{KayAndre}.  We record it here for completeness. 

We start with an important vanishing theorem for 
correspondence actions.

\begin{thm}(cf. \cite[Proposition 3.2.2.]{KayAndre}) \label{applicationvanishing}
	Let $X$ and $Y$ be connected $\NS$-schemes and let 
	$$
	\alpha \in \Hom_{\Cor_{\CH}}(X,Y)^0 = \CH^{d_X}(X\times_S Y, P(\Phi_X,\Phi_Y))
	$$
	be a correspondence from $X$ to $Y$, where $d_X := \dim_S(X)$.
	\begin{enumerate}
		\item If the support of $\alpha$ projects to an $r$-codimensional subset in $Y$, then the restriction of 
		$\rho_{H}\circ \Cor(\cl)(\alpha)$ to $\oplus_{j < r,i}H^i(X,\Omega_{X/S}^j)$ vanishes.
		\item If the support of $\alpha$ projects to an $r$-codimensional subset in $X$, then the restriction of 
		$\rho_{H}\circ \Cor(\cl)(\alpha)$ to $\oplus_{j\geq \dim_SX-r+1,i}H^i(X,\Omega_{X/S}^j)$ vanishes.
	\end{enumerate}
\end{thm}
\begin{proof}
	\begin{enumerate}
		\item Without loss of generality we can assume that $\alpha = [V]$ where $V\subset X\times_S Y$ is 
		an integral closed subscheme of $S$-dimension $\dim_S(V) = \dim_S(Y) =: d_Y$, and such that
		$pr_2(V) \subset Y$ has codimension $r$, where $pr_2: X\times_S Y \to Y$ is the projection morphism. 
		Recall that by definition 
		$$
		\rho_{H}\circ \Cor(\cl)([V])(\beta) = \hl(pr_2)(\hu(pr_1)(\beta)\cup \cl(V,X\times_S Y)),
		$$
		for $\beta \in H(X,\Phi_X)$.  Without loss of generality we can assume $\beta \in H^i(X,\Omega_{X/S}^j)$
		and so $\hu(pr_1)(\beta) \in H^{i}(X\times_S Y,\Omega_{X\times_SY/S}^j)$. Consider the diagram
		$$
		\xymatrix{
			\bigoplus\limits_{a+b=d_X} H^{d_X}_V(pr_1^{\ast}\Omega_{X/S}^a\otimes_{\sO_{X\times_S Y}} 
			pr_2^{\ast}\Omega_{Y/S}^b) \ar[d]^-{proj.} \ar[r]^-{\cong}
			&
			H^{d_X}_V(\Omega_{X\times_S Y/S}^{d_X}) \ar[r]^-{\hu(pr_1)(\beta)\cup} 
			&
			H^{d_X+i}_V(\Omega_{X\times_S Y/S}^{d_X+j})\ar[d]^-{proj.}
			\\
			H^{d_X}_V(pr_1^{\ast}\Omega_{X/S}^{d_X-j}\otimes_{\sO_{X\times_S Y}}pr_2^{\ast}\Omega_{Y/S}^{j}) 
			\ar[rr]^-{\hu(pr_1)(\beta)\cup} 
			&&
			H_V^{d_X+i}(pr_1^{\ast}\Omega_{X/S}^{d_X}\otimes_{\sO_{X\times_S Y}} pr_2^{\ast}\Omega_{Y/S}^{j}) \ar[d]
			\\&&
			H^i_{pr_2(V)}(Y,\Omega_{Y/S}^j),
		}
		$$
		where we write $H^p_V(\sF)$ for $H^p(X\times_S Y, \sF)$ for readability. First of all, we notice that 
		the lower vertical map on the right is chosen so that the composition is exactly $\hl(pr_2)$
		which we know we can do by \cite[Lemma 2.3.4]{KayAndre}. Secondly we notice that 
		the square commutes. This is because the projection on the left is precisely the one such that 
		cupping with $\hu(pr_1)(\beta)$ lands in 
		$H_V^{d_X+i}(pr_1^{\ast}\Omega_{X/S}^{d_X}\otimes_{\sO_{X\times_S Y}} pr_2^{\ast}\Omega_{Y/S}^{j})$,
		i.e. if the left arrow projects to  $H^{d_X}_V(pr_1^{\ast}\Omega_{X/S}^a\otimes_{\sO_{X\times_S Y}} 
		pr_2^{\ast}\Omega_{Y/S}^b)$, then cupping with $\hu(pr_1)(\beta)$ maps to 
	$H^{d_X+i}_V(pr_1^{\ast}\Omega_{X/S}^{a+j}\otimes_{\sO_{X\times_S Y}} 
		pr_2^{\ast}\Omega_{Y/S}^b)$ forcing $a = d_X-j$ to hold. 
		
		To show that this vanishes it thus suffices to show that 
		$\cl(V,X))$ vanishes under the map
		$$
		H^{d_X}_V(X\times_S Y,\Omega_{X\times_S Y/S}^{d_X}) \xrightarrow{proj.}
		H^{d_X}_V(X\times_S Y, pr_1^{\ast}\Omega_{X/S}^{d_X-j} \otimes pr_2^{\ast}\Omega_{Y/S}^{j}),
		$$
		for any $0 \leq j \leq r-1$.
		Furthermore, by Lemma \ref{injectivityoflocalizationlemma} we may localize to the generic point
		$\eta$ of $V$ and thus it suffices to show that $\cl(V,X)_{\eta}$
		vanishes under the projection map
		$$
		H^{d_X}_{\eta}(X\times_S Y, \Omega_{X\times_SY/S}^{d_X}) \xrightarrow{proj.} 
		H^{d_X}_{\eta}(X\times_S Y, pr_1^{\ast}\Omega_{X/S}^{d_X-j} \otimes pr_2^{\ast}\Omega_{Y/S}^{j}),
		$$
		for all $0 \leq q \leq r-1$. 
	
		We write $B = \sO_{X\times_S Y,\eta}$ and $\sO_{Y,pr_2(\eta)}$. $A$ is a regular local ring of dimension 
		$r$ and $B$ is formally smooth over $A$. Let $t_1,\ldots, t_r \in A$ be a regular sequence of parameters.
		$B/(1\otimes t_1,\ldots,1\otimes t_r)$ is a regular local ring so there exist elements $s_{r+1},\ldots,s_{d_X}
		\in B$ such that $1\otimes t_1,\ldots, 1\otimes t_r, s_{r+1},\ldots,s_{d_X}$ is a regular sequence of 
		parameters for $B$. The explicit description of the cycle class given in Lemma \ref{Classsymbolnotation} 
		gives 
		$$
		\cl(V,X)_{\eta} = (-1)^{d_X} \begin{bmatrix}
			d(1\otimes t_1)\wedge \cdots \wedge d(1\otimes t_r) \wedge ds_{r+1}\wedge\cdots\wedge ds_{d_X} \\
			1\otimes t_1, \ldots, 1\otimes t_r, s_{r+1}, \ldots, s_{d_X}
		\end{bmatrix}.
		$$
		The construction of the element $\begin{bmatrix}
			m \\ t
		\end{bmatrix}$  is functorial, see
		\cite[EXP. II, Proposition 5]{SGA2},  
		and therefore in order to show that $\cl(V,X\times_S Y)_{\eta}$ vanishes under 
		$$
		H^{d_X}_{\eta}(X\times_S Y, \Omega_{X\times_SY/S}^{d_X}) \xrightarrow{proj.} 
		H^{d_X}_{\eta}(X\times_S Y, pr_1^{\ast}\Omega_{X/S}^{d_X-j} \otimes pr_2^{\ast}\Omega_{Y/S}^{j}),
		$$ it suffices to 
		show that $d(1\otimes t_1)\wedge \cdots \wedge d(1\otimes t_r) \wedge ds_{r+1}\wedge\cdots\wedge ds_{d_X}$
		vanishes under the corresponding projection 
 	$$
		\Omega_{B/R}^{d_X} \to \Omega_{C/R}^{d_X-j}\otimes_R \Omega_{A/R}^j,
		$$
		where $R = \sO_S(S)$, and $C = \sO_{X,pr_1(\eta)}$. Since $0 \leq j \leq r-1$ this is clear; every 
		term of the image must have at least one $d(1) = 0$ occuring in the $ \Omega_{C/R}^{d_X-j}$ part
		and hence all terms are zero.
		\item The proof of this part is by symmetry the same as in part $(1)$. It suffices to show that 
		$\cl(V,X\times_S Y)$ vanishes under the projection map
		$$
		H^{d_X}_{\eta}(X\times_S Y, \Omega_{X\times_SY/S}^{d_X}) \xrightarrow{proj.} 
		H^{d_X}_{\eta}(X\times_S Y, pr_1^{\ast}\Omega_{X/S}^{j} \otimes pr_2^{\ast}\Omega_{Y/S}^{d_X-j}),
		$$
		and from here the argument is the same.
	\end{enumerate}
\end{proof}

Let $S'$ be a separated $S$-scheme and $f:X\to S'$ and $g:Y \to S'$ be 
integral $S'$-schemes that are $\NS$-schemes. Let $Z \subset 
X\times_{S'} Y $ be a closed integral subscheme s.t. $\dim_S(Z) = 
\dim_S(Y)$ and s.t. $pr_2|_Z:Z\to Y$ is proper, where $pr_2:X\times_{S'} Y
\to Y$ is the projection. For an open subscheme $U \subset S'$, we write
$Z_U$ for the pullback of $Z$ over $U$ inside $f^{-1}(U)\times_U g^{-1}(U)$.
This gives a correspondence $[Z_U] \in \Hom_{\Cor_{\CH}}(f^{-1}(U),
g^{-1}(U))^0$, which induces a morphism of $\sO_S$-modules
$$
\rho_{H}\circ \Cor(\cl)([Z_U]): H^i(f^{-1}(U),\Omega^j_{f^{-1}(U)/S})
\to H^i(g^{-1}(U),\Omega^j_{g^{-1}(U)/S}),
$$
for all $i,j$. 

In this situation we have the following proposition.
\begin{prop} \label{rhoOSmodulemap}
	The set $\{\rho_{H}\circ Cor(\cl)([Z_U]) | U \subset Z \,\,\, 
	\text{open}\}$ induces a morphism of quasi-coherent 
	$\sO_{S'}$-modules
	$$
	\rho_{H}(Z/S'): R^i\fpush\Omega_{X/S}^j \to R^i\gpush\Omega_{Y/S}^j,
	$$
	for all $i,j$.
\end{prop}
\begin{proof}

	We have to show two statements:
	\begin{enumerate}
		\item The maps $\rho_{H}\circ \Cor(\cl)([Z_U])$ are compatible 
		with restrictions to opens sets.
		\item The maps $\rho_{H}\circ \Cor(\cl)([Z_U])$ are $\sO(U)$-linear. 
	\end{enumerate}
	We denote by
	$$
	pr_{1,U}: f^{-1}(U)\times g^{-1}(U) \to f^{-1}(U)
	$$
	the map in $\vu$ induced by the first projection 
	$f^{-1}(U)\times g^{-1}(U) \to f^{-1}(U)$ and by
	$$
	pr_{2,U}:(f^{-1}(U)\times_{S'} g^{-1}(U), P(\Phi_{f^{-1}(U)},\Phi_{g^{-1}(U)})) 
	\to g^{-1}(Y)
	$$
	the map in $\vl$ induced by the first projection 
	$f^{-1}\times g^{-1}(U) \to g^{-1}(U)$, and denote by
	\begin{align*}
		j_f:f^{-1}(V) &\to f^{-1}(U) \,\,\, \text{and} \\
		j_g:g^{-1}(V) &\to g^{-1}(U) 
	\end{align*}
	the morphisms in $\vu$ induced by  
	an open immersion  $j:V\hookrightarrow U$.
	To show $(1)$ we have to show that for any $\alpha \in 
	H^i(f^{-1}(U),\Omega^j_{f^{-1}(U)/S})$ we have 
	\begin{align} \label{applicationstep1whattoshow}
		\hu(j_g)\hl(pr_{2,U})(\hu(&pr_{1,U})(\alpha)\cup \Cor(\cl)([Z_U])) \\
		&= \hl(pr_{2,V})(\hu(pr_{1,V})(\hu(j_f)(\alpha)\cup 
		\Cor(\cl)([Z_V])). \notag
	\end{align}
	Consider the Cartesian square
	$$
	\xymatrix{
		(f^{-1}(U)\times_{S'} g^{-1}(V),\Phi) \ar[d]_-{id_{f^{-1}(U)}\times j_g}
		\ar[r]^-{pr'_{2,V}} 
		& g^{-1}(V) \ar[d]^-{j_g} 
		\\(f^{-1}(U)\times_{S'} g^{-1}(U), P(\Phi_{f^{-1}(U)},\Phi_{g^{-1}(U)})) 
		\ar[r]^-{pr_{2,U}}
		&g^{-1}(U),
	}
	$$
	where $\Phi$ is defined as $(id_{f^{-1}(U)}\times j)^{-1}(P(\Phi_{f^{-1}(U)},\Phi_{g^{-1}(U)}))$ and 
	$$
	pr'_{2,V}: (f^{-1}(U)\times_{S'} g^{-1}(V),\Phi) \to
	g^{-1}(V) 
	$$
	is the map in $\vl$ induced by the first projection 
	$ f^{-1}(U)\times_{S'} g^{-1}(V) \to g^{-1}(V)$.
	Since $j_g$ is induced by a smooth morphism we see that 
	\begin{equation} \label{applicationpushpullstep}
		\hu(j_g)\hl(pr_{2,U}) = \hl(pr_{2,V}')\hu(id_{f^{-1}(U)}\times j_g).
	\end{equation}
	Denote by
	$pr_{1,U}':f^{-1}(U)\times_{S'} g^{-1}(V) \to f^{-1}(U)$
	the morphism in $\vu$ induced by the first projection, then 
	applying \eqref{applicationpushpullstep} to the left-hand 
	side of 
	\eqref{applicationstep1whattoshow} gives
	\begin{align*}
		\hu(j_g)\hl(pr_{2,U})(\hu(pr_{1,U})(\alpha)&\cup \Cor(\cl)([Z_U])) \\&=
		\hl(pr_{2,V}')\hu(id_{f^{-1}(U)}\times j_g)(\hu(pr_{1,U})(\alpha)\cup \Cor(\cl)([Z_U])) \\
		&=  \hl(pr_{2,V}')(\hu(pr_{1,U}')(\alpha)\cup \Cor(\cl)([Z_V])),
	\end{align*}
	where the last equality follows from the fact that $pr_{1,U}'
	= j_f \circ pr_{1,V}$, $\hu(id_{f^{-1}(U)}\times j_g)(\Cor(\cl)([Z_U])) =
	\Cor(\cl)([Z_V])$ and pullbacks commute with cup products. We introduce
	the morphisms 
	$$
	j_f\times id_{g^{-1}(V)}; f^{-1}(V)\times_{S'}g^{-1}(V) \to 
	f^{-1}(U)\times_{S'} g^{-1}(V)
	$$
	in $\vu$, and
	$$
	\tau: (f^{-1}(V)\times_{S'}g^{-1}(V),Z_V) \to (f^{-1}(U)\times_{S'} g^{-1}(V), \Phi),
	$$
	and 
	$$
	id':(f^{-1}(V)\times_{S'}g^{-1}(V),Z_V) \to
	(f^{-1}(V)\times_{S'} g^{-1}(V), P(\Phi_{f^{-1}(V)},\Phi_{g^{-1}(V)}))
	$$ 
	in $\vl$, where $\tau$ is induced by $j\times id_{g^{-1}(V)}$ and 
	$id'$ is induced by the identity. Applying the second 
	projection
	formula to 
	$\hl(pr_{2,V}')(\hu(pr_{1,U}')(\alpha)\cup \Cor(\cl)([Z_V]))$
	gives
	\begin{align*}
		\hl(pr_{2,V}')(\hu(pr_{1,U}')(\alpha)&\cup \Cor(\cl)([Z_V])) =\\
		&= \hl(pr_{2,V}')(\hu(pr_{1,U}')(\alpha)\cup \Cor(\cl(\CH_{\ast}(\tau))([Z_V])) \\
		&= \hl(pr_{2,V}') \hl(\tau) (\hu(j_f\times id_{g^{-1}(V)})\hu(pr_{1,U}')(\alpha)\cup \Cor(\cl)([Z_V])),
	\end{align*}
	and the equalities 
	\begin{align*}
		\hl(pr'_{2,V})\hl(\tau) &= \hl(pr_{2,V})\hl(id') \,\,\, \text{and}, \\
		\hu(j_f\times id_{g^{-1}(V)})\hu(pr_{1,U}') &= \hu(pr_{1,V})\hu(j_f)
	\end{align*}
	imply that \eqref{applicationstep1whattoshow} holds. 
	
	To show $(2)$ we note that it suffices to consider the case $U = S' = \Spec(R')$.
	We have to show that the following equality holds for all $r' \in R'$ and all
	$a \in H^i(X,\Omega_{X/S}^j)$:
	\begin{align}\label{prop7.2whattoshow1}
		g^{\ast}(r')\cup \hl(pr_2)(\hu(pr_1)(a)&\cup \cl([Z])) \\
		&= 
		\hl(pr_2)(\hu(pr_1)(f^{\ast}(r')\cup a)\cup \cl([Z]), \notag
	\end{align}
	where $g^{\ast}:R'\to H^0(X,\sO_X)$ and $f^{\ast}: R'\to H^0(Y,\sO_Y)$ are 
	the ring homomorphisms inducing a $R'$-module structures on $H(X)$ and $H(Y)$, 
	respectively. Notice that if we have 
	\begin{equation} \label{prop7.2whattoshow2}
		\hu(pr_2)(g^{\ast}(r')) \cup \cl([Z]) =
		\hu(pr_1)(f^{\ast}(r'))\cup \cl([Z]),
	\end{equation}
	in $H_Z^{d_X}(X\times_S Y, \Omega_{X\times_S Y}^{d_X})$, where $d_X := \dim_S(X)$,
	then 
	\begin{align*}
		\hl(pr_2)(\hu(pr_1)(f^{\ast}(r')&\cup a)\cup \cl([Z])) \\
		&= \hl(pr_2)(\hu(pr_1)(f^{\ast}(r'))\cup \hu(pr_1)(a)\cup \cl([Z])) \\
		&= \hl(pr_2)(\hu(pr_2)(g^{\ast}(r')) \cup \cl([Z])  \cup \hu(pr_1)(a)) \\
		&=g^{\ast}(r')\cup \hl(pr_2)(\hu(pr_1)(a)\cup \cl([Z])),
	\end{align*}
	where the first equality holds since the pullback commutes with the cup product,  the second equality is simply 
	\eqref{prop7.2whattoshow2}, and the final equality follows from the second projection 
	formula.
	
	To finish the proof, it suffices to show that \eqref{prop7.2whattoshow2}
	holds for any $r' \in R'$. By  Lemma \ref{injectivityoflocalizationlemma}, we see that
	it suffices to check this locally around  the generic point $\eta \in Z$. We can, without
	loss of generality,  further shrink the open set around $\eta$ and assume $Z$ is 
	regular and such that the ideal of $X$ is generated by a regular sequence
	$t_1,\ldots,t_{d_X}$. We can shrink further around $\eta$ and 
	assume $X\times_{S'} Y$ and $X\times_S Y$ are affine. 
	By Lemma \ref{Classsymbolnotation}, we see that 
	$$
	\cl(Z,X\times_S Y)_{\eta} = (-1)^{d_X}\begin{bmatrix}
		dt_1\wedge\cdots\wedge dt_{d_X} \\
		t_1,\ldots, t_{d_X}
	\end{bmatrix},
	$$
	and we write
	\begin{align*}
		r'\otimes 1 &:= \hu(pr_1)(f^{\ast}(r')), \,\,\, \text{and} \\
		1 \otimes r' &:= \hu(pr_2)(g^{\ast}(r')).
	\end{align*}
	Then it suffices to show that 
	\begin{equation} \label{prop7.2whattoshow3}
		r' \otimes 1 \cup (-1)^{d_X}\begin{bmatrix}
			dt_1\wedge\cdots\wedge dt_{d_X} \\
			t_1,\ldots, t_{d_X}
		\end{bmatrix} = 1\otimes r' \cup (-1)^{d_X}\begin{bmatrix}
			dt_1\wedge\cdots\wedge dt_{d_X} \\
			t_1,\ldots, t_{d_X}
		\end{bmatrix}.
	\end{equation}
	It follows from Lemma \ref{Cupproductsymbolnotation} that 
	\begin{align*}
		r' \otimes 1 \cup (-1)^{d_X}\begin{bmatrix}
			dt_1\wedge\cdots\wedge dt_{d_X} \\
			t_1,\ldots, t_{d_X}
		\end{bmatrix} &= (-1)^{d_X}\begin{bmatrix}
			(r' \otimes 1 )dt_1\wedge\cdots\wedge dt_{d_X} \\
			t_1,\ldots, t_{d_X}
		\end{bmatrix}, \,\,\, \text{and} \\
		1\otimes r' \cup (-1)^{d_X}\begin{bmatrix}
			dt_1\wedge\cdots\wedge dt_{d_X} \\
			t_1,\ldots, t_{d_X}
		\end{bmatrix} &= (-1)^{d_X}\begin{bmatrix}
			(1\otimes r')dt_1\wedge\cdots\wedge dt_{d_X} \\
			t_1,\ldots, t_{d_X}
		\end{bmatrix},
	\end{align*}
	So the Equation \eqref{prop7.2whattoshow3} follows if we can proof 
	\begin{equation} \label{prop7.2whattoshow4}
		\begin{bmatrix}
			rdt_1\wedge\cdots\wedge dt_{d_X} \\
			t_1,\ldots, t_{d_X}
		\end{bmatrix} = 0,
	\end{equation}
	where $r:= r' \otimes 1 - 1\otimes r'$. Note that since $S' \to S$ is separated
	by assumption,  $X\times_{S'} Y \to X\times_S Y$ is a 	
	closed immersion, 
	and if we pull $r$ back to $H_Z^0(X\times_{S'} Y, \sO_{X\times_{S'} Y})$
	it clearly vanshes. In particular it lies in the ideal of $Z$ in $X\times_{S'} Y$
	which is a subset of the ideal of $Z$ in $X\times_S Y$. Then Equation 
	\eqref{prop7.2whattoshow4} follows from the additivity of
	symbols.
\end{proof}

Recall the following definition.

\begin{defn}
	Two integral schemes $X$ and $Y$ over a base scheme $S$ are called 
	properly birational over $S$ if there exists an integral scheme $Z$ over
	$S$ and proper birational $S$-morphisms
	$$
	\xymatrix{
		&Z \ar[dr]\ar[dl] 
		\\X 
		&& Y.
	}
	$$
\end{defn}

\begin{thm}(cf. \cite[Theorem 3.2.8.]{KayAndre}) \label{Thefinaltheorem}
	Let $S$ be a 
	Noetherian, excellent, regular, separated, irreducible scheme of dimension at 
	most 1.  Let $S'$ be a separated $S$-scheme of finite type,
	and let $X$ and $Y$ be integral $\NS$-schemes, and
	$f:X\to S'$ and $g:Y\to S'$ be morphisms of $S$-schemes
	such that $X$ and $Y$ are
	properly birational over $S'$. Let $Z$
	be an integral scheme and let $Z\to X$ and $Z\to Y$ be
	proper birational morphisms
	such that 
	$$
	\xymatrix{
		&Z \ar[dr]\ar[dl] 
		\\X\ar[dr]_-f 
		&& Y\ar[dl]^-g
		\\&S'
	}	
	$$
	commutes. We denote the image of $Z$ in $X\times_{S'} Y$ by 
	$Z_0$. Then $\rho(Z_0/S')$ induces isomorphisms of $\sO_{S'}$-
	modules
	\begin{align*}
		R^i\fpush\sO_X &\xrightarrow{\cong} R^i\gpush \sO_Y \,\,\, \text{and} \\
		R^i \fpush \Omega_{X/S}^{d} &\xrightarrow{\cong} R^i\gpush
		\Omega_{Y/S}^d,
	\end{align*}
	for all $i$, where $d := \dim_S(X) = \dim_S(Y)$. 
\end{thm}
\begin{proof} 
	First we recall that $\rho(Z_0/S')$ is defined as the sheafification of 
	the maps
	\begin{equation} \label{whatwesheafify}
		\rho_{H}\circ \Cor(\cl)([Z_{0,U}]) : H^i(f^{-1}(U),\Omega^j_{f^{-1}(U)/S})
		\to H^i(g^{-1}(U),\Omega^j_{g^{-1}(U)/S}),
	\end{equation}
	where $U$ runs over all open subsets of $S'$ and $Z_{0,U}$ is the
	restriction of $Z_0$ to $f^{-1}(U)\times_U g^{-1}(U)$. It clearly suffices 
	then to show that \eqref{whatwesheafify} is an isomorphism for 
	$j = 0, i = d,$ and every open $U \subset S'$. We can therefore 
	without loss of generality suppose that $U = S', \, f^{-1}(U) = X, \,
	g^{-1}(U) = Y,$ and  $Z_{0, U} = Z_0$, and we need to show that
	\begin{align*}
		\rho_{H}\circ \Cor(\cl)([Z_0]): H^i(X,\sO_{X}) &\to H^i(Y,\sO_Y) \,\,\, \text{and}\\
		\rho_{H}\circ \Cor(\cl)([Z_0]): H^i(X,\Omega^d_{X/S}) &\to 
   H^i(Y,\Omega^d_{Y/S})
	\end{align*}
	are isomorphisms for all $i$. None of the cohomology groups, 
	$H^i(X,\sO_{X}), H^i(Y,\sO_Y), H^i(X,\Omega^d_{X/S}),$ or 
	$H^i(Y,\Omega^d_{Y/S})$ depend on $S'$, and it follows from 
	the universal property of fibre products that \linebreak $\rho_{H}\circ \Cor(cl)([Z_0])$
	does not depend on $S'$. Furthermore, since $Z_0\subset X\times_{S'} Y$
	is closed, and $X\times_{S'} Y \subset X\times_S Y$ is closed because 
	we choose $S'$ to be separated over $S$, then $Z_0 \subset X\times_S Y$ is closed.
	We can therefore reduce to the case
	where $S' = S$. Furthermore, since it is clear that 
	$\rho_{H}\circ \Cor(\cl)([Z_0])$ only depends on the image of $Z$ in 
	$X\times_S Y$, we may assume that $Z\subset X\times_S Y$ and 
	$Z = Z_0$. 
	
	By assumption on $Z, X, Y$ there exist open subsets $Z' \subset Z,
	X' \subset X,$ and $Y' \subset Y$, s.t. $pr_1^{-1}(X') = Z'$ 
	and $pr_2^{-1}(Y') = Z'$ and such that 
	$pr_1|_{Z'}:Z' \to X'$ and $pr_2|_{Z'}:Z'\to Y'$ are
	isomorphisms, where $pr_1:X\times_S Y\to X$ and $pr_2:X\times_S Y
	\to Y$ denote the projections. 
	
	The subset $Z$ defines a correspondence $[Z]\in \Hom_{\Cor_{\CH}}(X,Y)^0$ and 
	we denote by $[Z^t]$ the transpose, i.e. the correspondence
	$[Z^t] \in \Hom_{\Cor_{\CH}}(Y,X)^0$ defined by viewing $Z$
	as a subset of $Y\times_S X$. 
	
	We claim that 
	\begin{align}
		[Z]\circ [Z^t] &= \Delta_{Y/S} + E_1, \,\,\, \text{and} \\
		[Z^t]\circ [Z] &= \Delta_{X/S} + E_2, \notag
	\end{align}
	where $E_1$ and $E_2$ are cycles supported in 
	$(Y\setminus Y') \times_S (Y\setminus Y')$ and 
	$(X\setminus X') \times_S (X\setminus X')$ respectively.
	
	By \cite[Lemma 1.3.4]{KayAndre} we see that 
	$[Z^t]\circ [Z]$ is naturally supported in 
	$$
	\supp(Z,Z') = \left\{(x_1,x_2)\in X\times_S X | (x_1,y)\in X, 
	(y,x_2)\in Z', \,\,\, \text{for some} \,\,\, y \in Y\right\}.
	$$
	Applying \cite[Lemma 1.3.6]{KayAndre}  for the open 
	subset $X' \subset X$ tells us that $[Z']\circ [Z]$ maps to
	$[\Delta_{X'/S}]$ via the localization map
	$$
	\CH(\supp(Z,Z^t)) \to \CH(\supp(Z,Z')\cap (X'\times_S X')).
	$$
	Therefore 
	$$
	[Z^t]\circ [Z] = \Delta_{X/S} + E_2
	$$
	where $E_2$ is supported in $\supp(Z,Z^t)\setminus (X'\times_S X')$.
	Furthermore, 
	$$
	\supp(Z,Z^t) \cap ((X'\times_S X)\cup (X\times_S X')) 
	= \Delta_{X'/S} = \supp(Z,Z^t)\cap (X'\times_S X'),
	$$
	and therefore $E_2$ is supported in $(X\times_S X) \setminus 
	((X'\times_S X)\cup (X\times_S X')) = (X\setminus X') 
	\times_S (X\setminus X')$. The same argument shows that 
	$[Z]\circ [Z^t] = \Delta_{Y/S} + E_1$ where $E_1$ is supported in 
	$(Y\setminus Y') \times_S (Y\setminus Y')$. 
	
	Theorem \ref{applicationvanishing} now tells us that 
	$\rho_{H}\circ \Cor(\cl)(E_2)$ vanishes
	on $H^i(X,\sO_X)$ and $H^i(X,\Omega^d_{X/S})$ for all $i$,
	and that $\rho_{H}\circ \Cor(\cl)(E_1)$ vanishes
	on $H^i(Y,\sO_Y)$ and $H^i(Y,\Omega^d_{Y/S})$ for all $i$.
	this implies that 
	\begin{align*}
		\rho_{H}\circ \Cor(\cl)([Z]): H^i(X,\sO_{X}) &\to H^i(Y,\sO_Y) \,\,\, \text{and}\\
		\rho_{H}\circ \Cor(\cl)([Z]): H^i(X,\Omega^d_{X/S}) &\to 
		H^i(Y,\Omega^d_{Y/S})
	\end{align*}
	are isomorphisms for all $i$.
\end{proof}

Let $X = Y, S' = Y$ and $g= id_Y$. Then the 
following corollary is immediately evident.

\begin{cor}
	Let $S$ be a 
	Noetherian, excellent, regular, separated, irreducible scheme of dimension at 
	most 1, let $X,Y$ be integral smooth $S$-schemes, and $f:X\to Y$ a 
	proper birational morphism. Then
	$$
	R^if_{\ast}\sO_X = 0,
	$$
	for all $i\geq 0$.
\end{cor}

The following corollary compares the cohomology of the structure sheaf of fibres over closed 
points of two 
models $\sX$ and $\sY$ of a smooth proper scheme $X$ over a 
number field $K$.

\begin{cor}
	Let $K$ be a number field and $\sO_K$ be its ring of integers. Let $X$ be a smooth and proper
	$K$-scheme. Let $\sX$ and $\sY$ be two integral, smooth, proper models of $X$ over some dense 
	open subscheme $U \subseteq S = \Spec(\sO_K)$. Then for any dense open $V\subseteq U$
	such that $\h^j(\sX,\sO_{\sX})$ is $\sO_S(V)$-torsion-free we have
		$$
		\h^j(\sX_t,\sO_{\sX_t}) = 	\h^j(\sY_t,\sO_{\sY_t})
		$$
	for all closed points $t \in V$. 
\end{cor}
\begin{proof}
		Let $Z$ be the closure of the diagonal of $X$ in $\sX\times_U \sY$. Then we have 
		$$
		\xymatrix{
			&Z \ar[dr]\ar[dl] 
			\\ \sX\ar[dr]_-f 
			&& \sY\ar[dl]^-g
			\\&U,
		}	
		$$
		where $Z\to \sX$ and $Z\to \sY$ are proper birational, and $f$ and $g$ are the smooth 
		proper structure morphisms.  By Theorem \ref{maintheorem}, we have
		$$
		\rmR^jf_{\ast}(\sO_{\sX}) = \rmR^jg_{\ast}(\sO_{\sY})
		$$
		for all $j$. 
		
		It is clear that such a $V$ exists and by base change of $\sX$ and $\sY$ to it, we get 
		models of $X$ over $V$ and the fibres over closed points are unchanged. We can therefore
		without loss of generality assume $V = U$, i.e. that $\h^j(\sX,\sO_{\sX})$ is torsion-free over 
		$U$ for all $j$, and so $\rmR^jf_{\ast}(\sO_{\sX})$ and  $\rmR^jg_{\ast}(\sO_{\sY})$
		are locally free $\sO_U$-modules.  
		
		By starting from $j= \dim(\sX)+1$, we can use cohomology and base change, see for 
		example \cite[\S 5, Corollary 2]{AVar}, to work our way down and show
		$$
		\h^j(\sX_t,\sO_{\sX_t}) = 	\h^j(\sY_t,\sO_{\sY_t})
		$$
		for all $j \geq 0$.
\end{proof}

\bibliography{Dedekindbib}{}

\begin{thebibliography}{{Sta}18}

\bibitem[Con00]{Conrad}
Brian Conrad.
\newblock {\em Grothendieck Duality and Base Change}.
\newblock Lecture Notes in Mathematics. Springer Berlin Heidelberg, 2000.

\bibitem[Con07]{ConradNagata}
Brian Conrad.
\newblock Deligne's notes on {N}agata compactifications.
\newblock {\em J. Ramanujan Math. Soc.}, 22(3):205--257, 2007.

\bibitem[CR11]{KayAndre}
Andre Chatzistamatiou and Kay R\"{u}lling.
\newblock Higher direct images of the structure sheaf in positive
  characteristic.
\newblock {\em Algebra and Number Theory}, 5(6):693--775, 2011.

\bibitem[CR15]{KayAndreVanishing}
Andre Chatzistamatiou and Kay R\"{u}lling.
\newblock Vanishing of the higher direct images of the structure sheaf.
\newblock {\em Compositio Mathematica}, 151(11):2131--2144, 2015.

\bibitem[Cut90]{Cutkosky}
S.D. Cutkosky.
\newblock A new characterization of rational surface singularitues.
\newblock {\em Invent. Math.}, 102(1):157--177, 1990.

\bibitem[EZ78]{ElZein1978}
Fouad El~Zein.
\newblock Complexe dualisant et applications à la classe fondamentale d'un
  cycle.
\newblock {\em Mémoires de la Société Mathématique de France}, 58:1--66,
  1978.

\bibitem[Ful98]{Fultonbook}
William Fulton.
\newblock {\em Intersection Theory}.
\newblock Springer-Verlag, second edition, 1998.

\bibitem[GM02]{gelfman}
S.I. Gelfand and Y.I. Manin.
\newblock {\em Methods of Homological Algebra}.
\newblock Springer Monographs in Mathematics. Springer Berlin Heidelberg, 2002.

\bibitem[Gro61]{EGA2}
Alexander Grothendieck.
\newblock {\em \'El\'ements de g\'eom\'etrie alg\'ebrique : II. \'Etude globale
  \'el\'ementaire de quelques classes de morphismes}, volume~8.
\newblock Institut des Hautes \'Etudes Scientifiques, 1961.

\bibitem[Gro68]{SGA2}
Alexander Grothendieck.
\newblock {\em Cohomologie locale des faisceaux coh\'erents et th\'eor\`emes de
  {L}efschetz locaux et globaux {$(SGA$} {$2)$}}, volume~2 of {\em Advanced
  Studies in Pure Mathematics}.
\newblock North-{H}olland {P}ublishing {C}o., 1968.
\newblock Augment\'e d'un expos\'e par Mich\`ele Raynaud, S\'eminaire de
  G\'eom\'etrie Alg\'ebrique du Bois-Marie, 1962.

\bibitem[Har66]{ResAndDual}
Robin Hartsgorne.
\newblock {\em Residues and Duality}.
\newblock Lecture Notes in Mathematics. Springer Verlag, 1966.

\bibitem[Lod24]{Lodh}
R\'{e}mi Lodh.
\newblock Birational invariance of {$H^1(\mathcal{O})$}.
\newblock {\em Archiv der Mathematik}, 122:163--170, 2024.

\bibitem[Mat70]{matscommalg}
H.~Matsumura.
\newblock {\em Commutative Algebra}.
\newblock Mathematics lecture note series. Benjamin, 1970.

\bibitem[Mum70]{AVar}
David Mumford.
\newblock {\em Abelian varieties}, volume~5 of {\em Tata Institute of
  Fundamental Research Studies in Mathematics}.
\newblock Oxford University Press, 1970.

\bibitem[Nag63]{Nagata}
Masayoshi Nagata.
\newblock A generalization of the imbedding problem of an abstract variety in a
  complete variety.
\newblock {\em J. Math. Kyoto Univ.}, 3:89--102, 1963.

\bibitem[{Sta}18]{Stacks}
The {Stacks project authors}.
\newblock The stacks project.
\newblock online, 2018.

\bibitem[SV00]{SuslinVoevodsky}
Andrei Suslin and Vladimir Voevodsky.
\newblock Relative cycles and chow sheaves.
\newblock {\em Ann. of Math. Stud.}, 143, 02 2000.

\bibitem[Web15]{AndreasWeberThesis}
Andreas Weber.
\newblock Intersection theory on regular schemes via alterations and
  deformations to the normal cone, 2015.

\end{thebibliography}
\bibliographystyle{alpha}
\, \newline
\, \newline

\end{document}